\documentclass{article}

\usepackage{graphicx}
\usepackage{amsmath}
\usepackage{amssymb}
\usepackage{mathrsfs}
\usepackage{IEEEtrantools}
\usepackage{amsthm}
\usepackage{multirow}
\usepackage{tikz-cd}
\usepackage{amsmath}
\usepackage{bm}
\usepackage{fullpage}

\newcommand{\CC}{\mathbb{C}}
\newcommand{\RR}{\mathbb{R}}
\newcommand{\cP}{\mathcal{P}}

\newcommand{\dd}{\,\mathrm{d}}

\newtheoremstyle{plain}
{\topsep}
{1.5\topsep}
{\itshape}
{0pt}
{\bfseries}
{.}
{5pt plus 1pt minus 1pt}
{}

\newtheoremstyle{definition}
{\topsep}
{1.5\topsep}
{\normalfont}
{0pt}
{\bfseries}
{.}
{5pt plus 1pt minus 1pt}
{}

\newtheoremstyle{remark}
{0.5\topsep}
{0.8\topsep}
{\normalfont}
{0pt}
{\itshape}
{.}
{5pt plus 1pt minus 1pt}
{}

\theoremstyle{plain} \newtheorem{lemma}{Lemma}
\theoremstyle{plain} \newtheorem{thm}[lemma]{Theorem}
\theoremstyle{plain} \newtheorem{prop}[lemma]{Proposition}
\theoremstyle{plain} \newtheorem{lem}[lemma]{Lemma}
\theoremstyle{definition} \newtheorem{definition}[lemma]{Definition}
\theoremstyle{remark} 
\theoremstyle{plain} \newtheorem{cor}[lemma]{Corollary}

\newcommand{\tRe}{\mathrm{Re}\,}
\newcommand{\tIm}{\mathrm{Im}\,}

\begin{document}
\allowdisplaybreaks
\title{Strict condition for the $L^{2}$-wellposedness of fifth and sixth order dispersive equations}
\date{}
\author{Taehun Kim}
\maketitle
\pagestyle{plain}

\begin{abstract}
We provide a set of conditions that is necessary and sufficient for the $L^{2}$-wellposedness of the Cauchy problem for fifth and sixth order variable-coefficient linear dispersive equations. The necessity of these conditions had been presented by Tarama, and we scrutinized their proof to split the conditions into several parts so that an inductive argument is applicable. This inductive argument simplifies the engineering process of the appropriate pseudodifferential operator needed for the proof of $L^{2}$-wellposedness.
\end{abstract}

\section{Introduction}\label{sec1}

\subsection{Variable-coefficient linear dispersive equations and known results}\label{subsec1.1}
In this paper, we are concerned with the $L^{2}$-wellposedness of the Cauchy problem for variable-coefficient linear dispersive equations in $1$-dimension: \begin{eqnarray}\label{20241006eq1}
\begin{cases}
L_{k}u(x, t) = 0 & (x, t) \in \RR^{2} \\
u(x, t=0)=u_{0}(x) & x \in \RR
\end{cases},
\end{eqnarray} where $k \ge 2$ is an integer and $L_{k}$ is a variable-coefficient linear dispersive operator of the form \begin{eqnarray}\label{20241006eq2}
L_{k}u=D_{t}u-D_{x}^{k}u - \sum_{j=0}^{k-1} b_{j}(x)D_{x}^{j}(u).
\end{eqnarray} Here, $D_{t}$ and $D_{x}$ denote $\frac{1}{i}\partial_{t}$ and $\frac{1}{i}\partial_{x}$ respectively, and $b_{j}(0 \le j \le k-1)$ are given $\overline{C}^{\infty}(\RR)$ functions, complex-valued smooth functions whose partial derivatives of every order are bounded. Note that when $k=2$, $iL_{2}=\partial_{t}+i\partial_{xx}-b_{1}(x)\partial_{x}-ib_{0}(x)$ is a Schr\"odinger-type operator, and when $k=3$, $iL_{3}=\partial_{t}+\partial_{xxx}+ib_{2}(x)\partial_{xx}-b_{1}(x)\partial_{x}-ib_{0}(x)$ is a KdV-type operator.

The $L^{2}$-wellposedness of the Cauchy problem \eqref{20241006eq1} for small integers $k$ are studied in various papers. For $k=2$, Takeuchi \cite[\S 1]{Takeuchi74} proved that the following condition on $b_{1}$ is sufficient for the $L^{2}$-wellposedness of \eqref{20241006eq1}: \[
\left | \int_{x}^{y} \tIm \big (b_{1}(\widetilde{x})\big)\dd\widetilde{x} \right  | \lesssim 1.
\] (Here, and throughout this article, $X \lesssim Y$ means that $X \le CY$ holds for some absolute constant $C>0$. If $C$ depends on some parameters $S$, then we write $X \lesssim_{S}Y$.) In \cite[Lecture VII]{MizohataOtCP}, Mizohata proved that this condition is also necessary, and proved similar necessary conditions for higher dimensional case. For $k=3$, the following set of conditions is equivalent to the $L^{2}$-wellposedness of \eqref{20241006eq1}: \[
\begin{cases}
\left | \int_{x}^{y} \tIm \big ( b_{2}(\widetilde{x}) \big )\dd \widetilde{x} \right | \lesssim_{L_{3}} 1, \\
\left | \int_{x}^{y} \tIm \big ( b_{1}(\widetilde{x})-\frac{1}{3}b_{2}(\widetilde{x})^{2}\big ) \dd \widetilde{x} \right | \lesssim_{L_{3}} |x-y|^{1/2}
\end{cases}
\] This is proved by Tarama \cite{Tarama97} under additional assumption $b_{2} \equiv 0$, and by Mizuhara \cite{Mizuhara06} in full generality. In the same paper, Mizuhara also obtained an equivalent condition for the $L^{2}$-wellposedness of \eqref{20241006eq1} with $k=4$ under some additional assumption, and the same result without this assumption is proved by Tarama \cite{Tarama11}. In \cite{Tarama11}, Tarama generalized the arguments to obtain a set of necessary conditions for the $L^{2}$-wellposedness of \eqref{20241006eq1} with general $k$. One of these necessary conditions is that \begin{eqnarray}\label{20241006eq4}
\left | \int_{x}^{y} \tIm \big ( b_{k-1}(\widetilde{x}) \big ) \dd\widetilde{x} \right | \lesssim_{L_{k}} 1.
\end{eqnarray} Under this condition, $\varphi(x) = \exp \left ( -\frac{i}{k}\int_{0}^{x}b_{k-1}(\widetilde{x})\dd\widetilde{x} \right )$ and $\frac{1}{\varphi}$ are both $\overline{C}^{\infty}$. If $u$ satisfies \eqref{20241006eq1}, the equation satisfied by $v=\varphi^{-1} u$ is \begin{eqnarray}\label{20241006eq3}
\begin{cases}
\left (D_{t}-D_{x}^{k}-\sum_{j=0}^{k-2}\widetilde{b}_{j}D_{x}^{j}\right )v =0& (x, t) \in \RR^{2} \\
v(x, t=0)=\varphi(x)^{-1}u_{0}(x)&  x \in \RR
\end{cases},
\end{eqnarray} where $\widetilde{b}_{j} (0 \le j \le k-2)$ are $\overline{C}^{\infty}(\RR)$ functions determined by $b_{j'} (0 \le j' \le k-2)$. Since $\varphi, \varphi^{-1} \in \overline{C}^{\infty}$, the $L^{2}$-wellposedness of \eqref{20241006eq1} is equivalent to that of \eqref{20241006eq3}, under the condition \eqref{20241006eq4}. Since \eqref{20241006eq4} is necessary for \eqref{20241006eq1} to be $L^{2}$-wellposed, it is enough to study the $L^{2}$-wellposedness of \eqref{20241006eq3}. That is, we may assume without loss of generality that $b_{k-1}\equiv 0$.

The necessary conditions for $L^{2}$-wellposedness of \eqref{20241006eq1} obtained in \cite[Theorem 2.3]{Tarama11} are the following $(k-2)$ conditions \[
\left | \int_{x}^{y} \tIm \big ( d_{j}(\widetilde{x})\big)\dd\widetilde{x} \right | \lesssim_{L_{k}} |x-y|^{\frac{j}{k-1}} \quad (1 \le j \le k-2),
\] where $d_{j}\in \overline{C}^{\infty}(\RR) (1 \le j \le k-2)$ are functions determined by $b_{j'} (0 \le j' \le k-2)$. If the value of $k$ is specified, these $\left \{d_{j}\right \}$ can be explicitly calculated in terms of $\left \{b_{j'}\right \}$, and for $3 \le k \le 6$, these conditions read as
\begin{IEEEeqnarray*}{rCl}
\text{(i)} &\phantom{=}&\left | \int_{x}^{y} \tIm \big (b_{k-2}(\widetilde{x})\big)\dd\widetilde{x} \right  | \lesssim_{L_{k}} |x-y|^{\frac{1}{k-1}}, \\
\text{(ii)} & & \left | \int_{x}^{y} \tIm \big (b_{k-3}(\widetilde{x}) \big ) \dd \widetilde{x} \right | \lesssim_{L_{k}} |x-y|^{\frac{2}{k-1}} \quad \text{ if }k \ge 4, \\
\text{(iii)} & & \left | \int_{x}^{y} \tIm \big (b_{k-4}(\widetilde{x}) - \frac{k-3}{2k}b_{k-2}(\widetilde{x})^{2} \big ) \dd \widetilde{x} \right | \lesssim_{L_{k}} |x-y|^{\frac{3}{k-1}} \quad \text{ if }k \ge 5, \\
\text{(iv)} & & \left | \int_{x}^{y} \tIm \big (b_{k-5}(\widetilde{x}) - \frac{k-4}{k}b_{k-2}(\widetilde{x})b_{k-3}(\widetilde{x}) \big ) \dd \widetilde{x} \right | \lesssim_{L_{k}} |x-y|^{\frac{4}{k-1}} \quad \text{ if }k \ge 6.
\end{IEEEeqnarray*}

If $k=3, 4$ these are exactly the necessary and sufficient conditions for the $L^{2}$-wellposedness of \eqref{20241006eq1} obtained by Tarama (\cite{Tarama97}, \cite{Tarama11}). In this paper, we prove that this set of necessary conditions is also sufficient for the $L^{2}$-wellposedness of \eqref{20241006eq1} when $k$ is equal to $5$ or $6$.

\subsection{Main results}\label{subsec1.2}

The first main result of this paper gives a single necessary condition for the $L^{2}$-wellposedness of \eqref{20241006eq1}.
\begin{thm} \label{20241015thm1}
Let $k \ge 3$ and $1 \le m \le k-2$ be given integers. Let $A$ be a differential operator of order $\le k-2$, which is self-adjoint. Let $b_{j} \in \overline{C}^{\infty}(\RR) (1 \le j \le m)$ be given functions. Then, for the Cauchy problem \[
\begin{cases}
Lu := \left ( D_{t} - D_{x}^{k} - A - \sum_{j=0}^{m} b_{j}(x)D_{x}^{j} \right )u = 0 & (x, t) \in \RR^{2} \\
u(x, t=0)=u_{0}(x) & x \in \RR
\end{cases}
\] to be $L^{2}$-wellposed, it is necessary that \[
\left | \int_{x}^{y} \tIm \big ( b_{m}(\widetilde{x}) \big )\dd\widetilde{x} \right | \lesssim_{L} |x-y|^{\frac{k-1-m}{k-1}}.
\]
\end{thm}
This result is not novel in the sense that its proof goes on similarly as in \cite[Theorem 2.3]{Tarama11}. However, Theorem \ref{20241015thm1} does not follow directly from the cited theorem. To prove Theorem \ref{20241015thm1}, we need to analyze the coefficients ($\left \{P_{k-2-m;\ell, j}\right \}$ in Lemma \ref{20241007lem1} or $\{d_{j}\}$ in \cite[Proposition 2.1]{Tarama11}) more carefully, and then use this result for self-adjoint operators to prove that we may disregard self-adjoint part $A$ in Theorem \ref{20241015thm1}. The main advantage in proving this theorem is that it enables inductive argument. Indeed, Theorem \ref{20241015thm1} states that after disregarding the self-adjoint part as zeros, the integral of the coefficient of the highest term must satisfy a certain estimate. Using this estimate, we can switch the $L^{2}$-wellposedness of the original operator into the $L^{2}$-wellposedness of the modified operator with smaller $m$, after which induction is applicable. We will discuss more on this in~Section \ref{subsec1.3}. Moreover, unlike Tarama's proof which strongly uses the contradiction argument, we will obtain quantitative behavior of illposedness by using the so-called `duality testing argument' appearing in the paper of Jeong and Oh \cite[Appendix A]{JeongOh23}.

The second and third main results give a set of necessary and sufficient conditions for \eqref{20241006eq1} to be $L^{2}$-wellposed, where $k=5, 6$. Although these results are already introduced in Section \ref{subsec1.1}, we state them here once again for the sake of completeness.

\begin{thm} \label{20241015thm2}
Let $b, c, d \in \overline{C}^{\infty}(\RR)$ be given functions. Then the Cauchy problem \begin{eqnarray}\label{20241015eq2}
\begin{cases}
Lu := \left ( D_{t}-D_{x}^{5}-b(x)D_{x}^{3}-c(x)D_{x}^{2}-d(x)D_{x} \right )u =0 & (x, t) \in \RR^{2}, \\
u(x, t=0)=u_{0}(x) & x \in \RR
\end{cases}
\end{eqnarray} is $L^{2}$-wellposed if and only if the following three conditions are satisfied:
\begin{IEEEeqnarray*}{rCl}
\text{(i)} &\phantom{=}&\left | \int_{x}^{y} \tIm \big (b(\widetilde{x})\big)\dd\widetilde{x} \right  | \lesssim_{L} |x-y|^{\frac{1}{4}}, \\
\text{(ii)} & & \left | \int_{x}^{y} \tIm \big (c(\widetilde{x}) \big ) \dd \widetilde{x} \right | \lesssim_{L} |x-y|^{\frac{1}{2}}, \\
\text{(iii)} & & \left | \int_{x}^{y} \tIm \big (d(\widetilde{x}) - \frac{1}{5}b(\widetilde{x})^{2} \big ) \dd \widetilde{x} \right | \lesssim_{L} |x-y|^{\frac{3}{4}}.
\end{IEEEeqnarray*}
\end{thm}

\begin{thm} \label{20241015thm3}
Let $b, c, d, e \in \overline{C}^{\infty}(\RR)$ be given functions. Then the Cauchy problem \begin{eqnarray}\label{20241015eq13}
\begin{cases}
Lu:=\left ( D_{t}-D_{x}^{6}-b(x)D_{x}^{4}-c(x)D_{x}^{3}-d(x)D_{x}^{2}-e(x)D_{x} \right )u =0 & (x, t) \in \RR^{2}, \\
u(x, t=0)=u_{0}(x) & x \in \RR
\end{cases}
\end{eqnarray} is $L^{2}$-wellposed if and only if the following four conditions are satisfied:
\begin{IEEEeqnarray*}{rCl}
\text{(i)} &\phantom{=}&\left | \int_{x}^{y} \tIm \big (b(\widetilde{x})\big)\dd\widetilde{x} \right  | \lesssim_{L} |x-y|^{\frac{1}{5}}, \\
\text{(ii)} & & \left | \int_{x}^{y} \tIm \big (c(\widetilde{x}) \big ) \dd \widetilde{x} \right | \lesssim_{L} |x-y|^{\frac{2}{5}}, \\
\text{(iii)} & & \left | \int_{x}^{y} \tIm \big (d(\widetilde{x}) - \frac{1}{4}b(\widetilde{x})^{2} \big ) \dd \widetilde{x} \right | \lesssim_{L} |x-y|^{\frac{3}{5}},\\
\text{(iv)} & & \left | \int_{x}^{y} \tIm \big (e(\widetilde{x}) - \frac{1}{3}b(\widetilde{x})c(\widetilde{x}) \big ) \dd \widetilde{x} \right | \lesssim_{L} |x-y|^{\frac{4}{5}}.
\end{IEEEeqnarray*}
\end{thm}

Recall that we have assumed $b_{k-1}\equiv 0$ by using gauge transformation. Transforming back into the original equation gives the following corollaries:

\begin{cor}
Let $a, b, c, d \in \overline{C}^{\infty}(\RR)$ be given functions. Then the Cauchy problem \[
\begin{cases}
Lu:=\left ( D_{t}-D_{x}^{5}-a(x)D_{x}^{4}-b(x)D_{x}^{3}-c(x)D_{x}^{2}-d(x)D_{x} \right )u =0 & (x, t) \in \RR^{2}, \\
u(x, t=0)=u_{0}(x) & x \in \RR
\end{cases}
\] is $L^{2}$-wellposed if and only if the following four conditions are satisfied:
\begin{IEEEeqnarray*}{rCl}
\text{(i)} &\phantom{=}&\left | \int_{x}^{y} \tIm \big (a(\widetilde{x})\big)\dd\widetilde{x} \right  | \lesssim_{L} 1, \\
\text{(ii)} &\phantom{=}&\left | \int_{x}^{y} \tIm \left (b(\widetilde{x}) - \frac{2}{5}a(\widetilde{x})^{2}\right)\dd\widetilde{x} \right  | \lesssim_{L} |x-y|^{\frac{1}{4}}, \\
\text{(iii)} & & \left | \int_{x}^{y} \tIm \left (c(\widetilde{x}) - \frac{3}{5}a(\widetilde{x})b(\widetilde{x}) + \frac{4}{25}a(\widetilde{x})^{3} \right) \dd \widetilde{x} \right | \lesssim_{L} |x-y|^{\frac{1}{2}}, \\
\text{(iv)} & & \left | \int_{x}^{y} \tIm \left (d(\widetilde{x}) - \frac{1}{5}b(\widetilde{x})^{2}-\frac{2}{5}a(\widetilde{x})c(\widetilde{x})+\frac{7}{25}a(\widetilde{x})^{2}b(\widetilde{x})-\frac{7}{125}a(\widetilde{x})^{4} - \frac{1}{5}ia'(\widetilde{x})b(\widetilde{x}) + \frac{1}{5}a'(\widetilde{x})^{2} \right ) \dd \widetilde{x} \right | \lesssim_{L} |x-y|^{\frac{3}{4}}.
\end{IEEEeqnarray*}
\end{cor}

\begin{cor}
Let $a, b, c, d, e \in \overline{C}^{\infty}(\RR)$ be given functions. Then the Cauchy problem \[
\begin{cases}
Lu:=\left ( D_{t}-D_{x}^{6}-a(x)D_{x}^{5}-b(x)D_{x}^{4}-c(x)D_{x}^{3}-d(x)D_{x}^{2}-e(x)D_{x} \right )u =0 & (x, t) \in \RR^{2}, \\
u(x, t=0)=u_{0}(x) & x \in \RR
\end{cases}
\] is $L^{2}$-wellposed if and only if the following five conditions are satisfied:
\begin{IEEEeqnarray*}{rCl}
\text{(i)} &\phantom{=}&\left | \int_{x}^{y} \tIm \big (a(\widetilde{x})\big)\dd\widetilde{x} \right  | \lesssim_{L} 1, \\
\text{(ii)} &\phantom{=}&\left | \int_{x}^{y} \tIm \left (b(\widetilde{x})-\frac{5}{12}a(\widetilde{x})^{2} \right)\dd\widetilde{x} \right  | \lesssim_{L} |x-y|^{\frac{1}{5}}, \\
\text{(iii)} & & \left | \int_{x}^{y} \tIm \left (+c(\widetilde{x})-\frac{2}{3}a(\widetilde{x})b(\widetilde{x})+\frac{5}{27}a(\widetilde{x})^{3} \right) \dd \widetilde{x} \right | \lesssim_{L} |x-y|^{\frac{2}{5}}, \\
\text{(iv)} & & \left | \int_{x}^{y} \tIm \left (d(\widetilde{x}) - \frac{1}{4}b(\widetilde{x})^{2} - \frac{1}{2}a(\widetilde{x})c(\widetilde{x}) + \frac{3}{8}a(\widetilde{x})^{2}b(\widetilde{x}) - \frac{5}{64}a(\widetilde{x})^{4} - \frac{1}{4}ia'(\widetilde{x})b(\widetilde{x})+\frac{5}{16}a'(\widetilde{x})^{2} \right ) \dd \widetilde{x} \right | \lesssim_{L} |x-y|^{\frac{3}{5}}, \\
\text{(v)} & & \bigg | \int_{x}^{y} \tIm \bigg (e(\widetilde{x}) -\frac{1}{3}b(\widetilde{x})c(\widetilde{x}) - \frac{1}{3}a(\widetilde{x})d(\widetilde{x}) + \frac{2}{9}a(\widetilde{x})b(\widetilde{x})^{2}+\frac{2}{9}a(\widetilde{x})^{2}c(\widetilde{x})-\frac{14}{81}a(\widetilde{x})^{3}b(\widetilde{x})+\frac{7}{243}a(\widetilde{x})^{5} \\
&& \qquad \qquad-\frac{4}{9}a''(\widetilde{x})b(\widetilde{x}) - \frac{1}{3}ia'(\widetilde{x})c(\widetilde{x}) + \frac{2}{9}ia(\widetilde{x})a'(\widetilde{x})b(\widetilde{x})-\frac{10}{27}a(\widetilde{x})a'(\widetilde{x})^{2} \bigg ) \dd \widetilde{x} \bigg | \lesssim_{L} |x-y|^{\frac{4}{5}}
\end{IEEEeqnarray*}
\end{cor}

\subsubsection{Classes of pseudodifferential operators}

In the heuristic reasoning in Section \ref{subsec1.3} and proofs in Section \ref{sec3} and \ref{sec4}, we will use pseudodifferential operators. We collect here the definitions of the classes of pseudodifferential operators which will be used later.

\begin{definition}
Let $m \in \RR$. 
\begin{itemize}
\item[(a)] A pseudodifferential operator $B(x, D)=\mathrm{Op}(b(x, \xi))$ of order $m$ is an operator defined by \[
B(x, D)f(x) = \int_{\RR} e^{ix\xi} b(x, \xi)\hat{f}(\xi)\dd \xi,
\] where $b \in C^{\infty}(\RR)$ is a function satisfying \[
\left | \partial_{x}^{\alpha}\partial_{\xi}^{\beta}b(x, \xi)\right | \lesssim_{\alpha, \beta} \langle \xi \rangle^{m-\beta} \quad \forall \; (x, \xi) \in \RR^{2}
\] for any nonnegative integers $\alpha$ and $\beta$. Here, $\langle \xi \rangle := \sqrt{1+|\xi|^{2}}$. If this is the case, we say that $B$ is in the class $S^{m}$. We will abuse notation and identify pseudodifferential operator $B$ with its symbol $b$.
\item[(b)] Let $B(l;x, D)$ be a pseudodifferential operator with parameter $l$, defined by \[
B(l;x, D)f(x) = \int_{\RR}e^{ix\xi}b(l;x, \xi)\hat{f}(\xi)\dd\xi.
\] $B(l;x, D)$ is said to be in class $S_{(\ell)}^{m}$ iff for any nonnegative integers $\alpha$ and $\beta$, one has \[
\left | \partial_{x}^{\alpha}\partial_{\xi}^{\beta}b(l;x, \xi) \right | \lesssim_{\alpha, \beta} \langle \xi \rangle_{l}^{m-\beta} \quad \forall \; l \ge 1, \; \forall \; (x, \xi) \in \RR^{2},
\] where $\langle \xi \rangle_{l} := \sqrt{\xi^{2}+l^{2}}$. We will abuse notation by identifying $B$ and $b$.
\item[(c)] Let $A, B$ be pseudodifferential operators. Then the expression $A \equiv B \pmod{S^{m}}$ means that $A-B \in S^{m}$. Similarly, the expression $A \equiv B \pmod{S_{(\ell)}^{m}}$ means that $A-B \in S_{(\ell)}^{m}$.
\end{itemize}
\end{definition}

Several properties of these pseudodifferential operators are studied extensively in \cite{HormanderLPDO3} and \cite{Kumano-goPDO}. In this paper, an operator $A$ is said to be self-adjoint iff $A \equiv A^{\ast} \pmod{S^{0}}$.

\subsection{Playing with toy case $k=4$} \label{subsec1.3}

In this section, we will briefly explain the scheme for the proof of Theorem \ref{20241015thm2} and \ref{20241015thm3}. This abstract scheme will be exemplified by a heuristic proof in case $k=4$. The expositions below are organized to emphasize the parallel reasoning between the way how necessary condition pops up and the way it serves as a sufficient condition.

We started with a differential operator \[
L=D_{t}-D_{x}^{k}-\sum_{j=0}^{k-1}b_{j}(x)D_{x}^{j}
\] and used gauge transformation to get rid of the $(k-1)$th term. Our next target is to remove the $(k-2)$th term. This goal is not achievable by using simple gauge transformations. Instead, we make a change of variable $u=v+B_{-1}v$ where $B_{-1}$ is a pseudodifferential operator ($\Psi$DO) of order $-1$. Assuming that $u$ is a solution to $Lu=0$, the equation satisfied by $v$ is \[
(I+B_{-1})^{-1} \circ L \circ (I+B_{-1}) v = 0.
\] Expand $L \circ (I+B_{-1})$ by using formal symbolic calculus to get \[
L \circ (I+B_{-1}) \equiv D_{t}+B_{-1}D_{t} - D_{x}^{k}-B_{-1}D_{x}^{k} + ik\mathrm{Op}(\xi^{k-1}\partial_{x}B_{-1}) - b_{k-2}(x)D_{x}^{k-2} \pmod{S^{k-3}} 
\] Hence, if \[
ik\mathrm{Op}(\xi^{k-1}\partial_{x}B_{-1}) \equiv \mathrm{Op}(b_{k-2}(x)\xi^{k-2}) \pmod{S^{k-3}}, \quad \text{i.e., } \quad \partial_{x}B_{-1} \equiv -\frac{i}{k}b_{k-2}(x)D_{x}^{-1} \pmod{S^{-2}},
\] then we would have \[
(I+B_{-1})^{-1} \circ L \circ (I+B_{-1}) \equiv D_{t} - D_{x}^{k} \pmod{S^{k-3}},
\] so that the process to eliminate the $(k-2)$th term of $L$ is successful at least in our dream. Hence our objective turns down into the construction of a $\Psi$DO $B_{-1} \in S^{-1}$ such that the symbol of $\partial_{x}B_{-1}$ is approximately equal to $-\frac{i}{k}\frac{1}{\xi}b_{k-2}(x)$. Of course this is impossible without any further assumption on $b_{k-2}$. This is where the necessary condition obtained in Theorem \ref{20241015thm1} comes into play. The necessary condition is that \[
\left |\int_{x}^{y}\tIm \big (b_{k-2}(\widetilde{x})\big )\dd\widetilde{x}\right | \lesssim_{L} |x-y|^{\frac{1}{k-1}}.
\] If we assume this, then in \cite[Lemma 2.1]{Mizuhara06} it is proved that the following Tarama-type operator \[
B=\frac{1}{k\xi}\int_{-\infty}^{x}\chi \left ( \frac{y-x}{\xi^{k-1}} \right )\tIm \big (b_{k-2}(y)\big )\dd y
\] satisfies $B \in S^{0}$ and $\partial_{x}B \equiv \frac{1}{k\xi}\tIm \big (b_{k-2}(x)\big ) \pmod{S^{-(k-1)}}$, where $\chi \in C_{c}^{\infty}(-2, 2)$ is any given function satisfying $\chi(x) \equiv 1 \; \forall \;|x| \le 1$. (In fact, all $\xi$'s in the denominator must be replaced by $\langle \xi \rangle_{\ell}$'s and $S^{j}$'s must be replaced by $S_{(\ell)}^{j}$'s, but let's stay in the dream a little longer.) Hence by using this $\Psi$DO $B$, the necessary condition attained in Theorem \ref{20241015thm1} allows us to get rid of the term $\tIm \big (b_{k-2}(x)\big )D_{x}^{k-2}$ from the original differential operator $L$. The remaining coefficient of $(k-2)$th order term is now real-valued, so we can consider it as a part of a self-adjoint differential operator by altering $b_{j} (0 \le j \le k-3)$. Then we have \[
L_{(k-3)} = D_{t}-D_{x}^{k}-A_{(1)} - \sum_{j=0}^{k-3}b_{j}(x)D_{x}^{j}
\] on our hand, where $A_{(1)}$ is a self-adjoint differential operator of order $k-2$. Now we try eliminating the $(k-3)$th term. This time we use a change of variables $u=v+B_{-2}v$, where $B_{-2} \in S^{-2}$. After calculating formally just as before, we see that we want $B_{-2}$ to satisfy $\partial_{x}B_{-2} \approx -\frac{i}{k}\frac{1}{\xi^{2}}b_{k-3}(x)$. Again, this construction utilizes the necessary condition popping out from Theorem \ref{20241015thm1}, which is \[
\left | \int_{x}^{y}\tIm \big (b_{k-3}(\widetilde{x})\big)\dd\widetilde{x} \right | \lesssim |x-y|^{\frac{2}{k-1}}.
\] Note that Theorem \ref{20241015thm1} allowed us to ignore self-adjoint part $A_{(1)}$ of $L_{(k-3)}$ and concentrate solely on $b_{k-3}$. Assuming this condition to be satisfied, we have that the following operator \[
B=\frac{1}{k\xi^{2}}\int_{-\infty}^{x}\chi\left ( \frac{y-x}{\xi^{k-1}} \right )\tIm\big(b_{k-3}(y)\big)\dd y
\] satisfies $B \in S^{0}$ and $\partial_{x}B =\frac{1}{k\xi^{2}}\tIm \big (b_{k-3}(x)\big )$. (Again, we are being sloppy.) By using this $B$, we can only leave real parts in the coefficient of the $(k-3)$th order term, which can again be absorbed into self-adjoint part $A$. Now we have \[
L_{(k-4)} = D_{t}-D_{x}^{k}-A_{(2)}-\sum_{j=0}^{k-4}b_{j}(x)D_{x}^{j},
\] and the proof goes on by induction on $m$, the highest order in $L$ except for self-adjoint parts.

But how does the necessary condition obtained from Theorem \ref{20241015thm1} arise from the equation? The key scheme, duality testing argument, is to test the solution with an approximate solution for $L^{\ast}$. More precisely, let $u$ be a solution to $Lu=0$, and let $v$ be an approximate solution to $L^{\ast}v=0$ which is constructed explicitly. Then integrating $D_{t}\langle u, v \rangle = \langle u, L^{\ast}v \rangle$ with respect to $t$ gives \[
|\langle u(t=0), v(t=0)\rangle| \le \|u\|_{L^{2}}\|v\|_{L^{2}} + \int_{0}^{t}\|u\|_{L^{2}} \|L^{\ast}v\|_{L^{2}}.
\] Hence if $v$ is constructed so that $L^{\ast}v \approx 0$ and $v(t=0)=u(t=0)$ are satisfied, then $\|v\|_{L^{2}}$ must be large enough. This gives some condition that comes out from the explicit formula of $v$. However, a single use of this argument does not give the result in Theorem \ref{20241015thm1} at once. With one iteration of the duality testing argument, one can only prove \[
\left | \int_{x}^{y}\tIm \big (b_{k-2}(\widetilde{x})\big)\dd\widetilde{x} \right | \lesssim |x-y|^{\frac{1}{2}}.
\] Using this estimate, one can construct another approximate solution $v$ and use the  duality testing argument again to prove \[
\left | \int_{x}^{y}\tIm \big ( b_{k-2}(\widetilde{x})\big)\dd\widetilde{x} \right | \lesssim |x-y|^{\frac{1}{3}}, \quad  \text{and }\quad \left | \int_{x}^{y}\tIm \big ( b_{k-3}(\widetilde{x})\big)\dd\widetilde{x} \right | \lesssim |x-y|^{\frac{2}{3}}.
\] After $r$ iterations of the duality testing argument, one gets \[
\left | \int_{x}^{y}\tIm \big (d_{k-1-j}(\widetilde{x})\big)\dd\widetilde{x} \right | \lesssim |x-y|^{\frac{j}{r+1}} \quad (1 \le j \le r),
\] where $d_{k-1-j} \in \overline{C}^{\infty}(\RR)$ are functions depending on $b_{k-2}, \cdots, b_{k-1-j}$. (Although $d_{k-2}=b_{k-2}$ and $d_{k-3}=b_{k-3}$, $d_{k-1-j} = b_{k-1-j}$ is not true in general.) This is how one gets the conclusion of Theorem \ref{20241015thm1}.

\begin{figure}[h!]
\[\begin{tikzcd}[column sep = 3cm, row sep=1cm]
 L=D_{t}-D_{x}^{4}-b(x)D_{x}^{2}-c(x)D_{x} \arrow[d] & {} \\
\boxed{\begin{tabular}{c}\text{Necessary condition from Theorem \ref{20241015thm1}} \\ $\left | \int_{x}^{y}\tIm \big (b(\widetilde{x})\big)\dd\widetilde{x}\right | \lesssim |x-y|^{\frac{1}{3}} \cdot\cdot\cdot\cdot\cdot \raisebox{.5pt}{\textcircled{\raisebox{-.9pt}{1}}}$\end{tabular}} \arrow[r, 
"False"] \arrow[d, "True"] & {\text{Illposed}} \\
\begin{tabular}{c}Construct $\Psi$DO $\Phi_{1}$ using \raisebox{.5pt}{\textcircled{\raisebox{-.9pt}{1}}}. \\
$L^{2}$-wellposedness of $L \;\;\;\qquad$ \\
$\Leftrightarrow$ $L^{2}$-wellposedness of $L_{(1)}$, where \\
$L_{(1)} = \Phi_{1}^{-1} \circ L \circ \Phi_{1}$$\quad$$\quad$$\quad$$\quad$ \\
$\quad\quad\quad\quad\,\;\:= D_{t}-D_{x}^{4}-A_{(1)}-c_{(1)}(x)D_{x}$ \\ $\quad\quad\;\;$($A_{(1)}$ is self-adjoint) \end{tabular} \arrow[d]  & {} \\
\boxed{\begin{tabular}{c}\text{Necessary condition from Theorem \ref{20241015thm1}} \\ $\left | \int_{x}^{y}\tIm \big (c_{(1)}(\widetilde{x})\big)\dd\widetilde{x}\right | \lesssim |x-y|^{\frac{2}{3}} \cdot\cdot\cdot\cdot\cdot \raisebox{.5pt}{\textcircled{\raisebox{-.9pt}{2}}}$\end{tabular}} \arrow[r, 
"False"] \arrow[d, "True"] & {\text{Illposed}} \\
\begin{tabular}{c}Construct $\Psi$DO $\Phi_{2}$ using \raisebox{.5pt}{\textcircled{\raisebox{-.9pt}{2}}}. \\
\;\;\: $L^{2}$-wellposedness of $L_{(1)} \;\;\;\qquad$ \\
$\Leftrightarrow$ $L^{2}$-wellposedness of $L_{(2)}$, where \\
$L_{(2)} = \Phi_{2}^{-1} \circ L_{(1)} \circ \Phi_{2}$ \\
$\quad$\;\,\,$= D_{t}-D_{x}^{4}-A_{(2)}$ \\ $\quad \quad \quad \quad \;\;\;$($A_{(2)}$ is self-adjoint) \end{tabular} \arrow[d] & {} \\
{\text{Wellposed}} & {}
\end{tikzcd}
\]
\caption{Proof scheme in case $k=4$}
\end{figure}
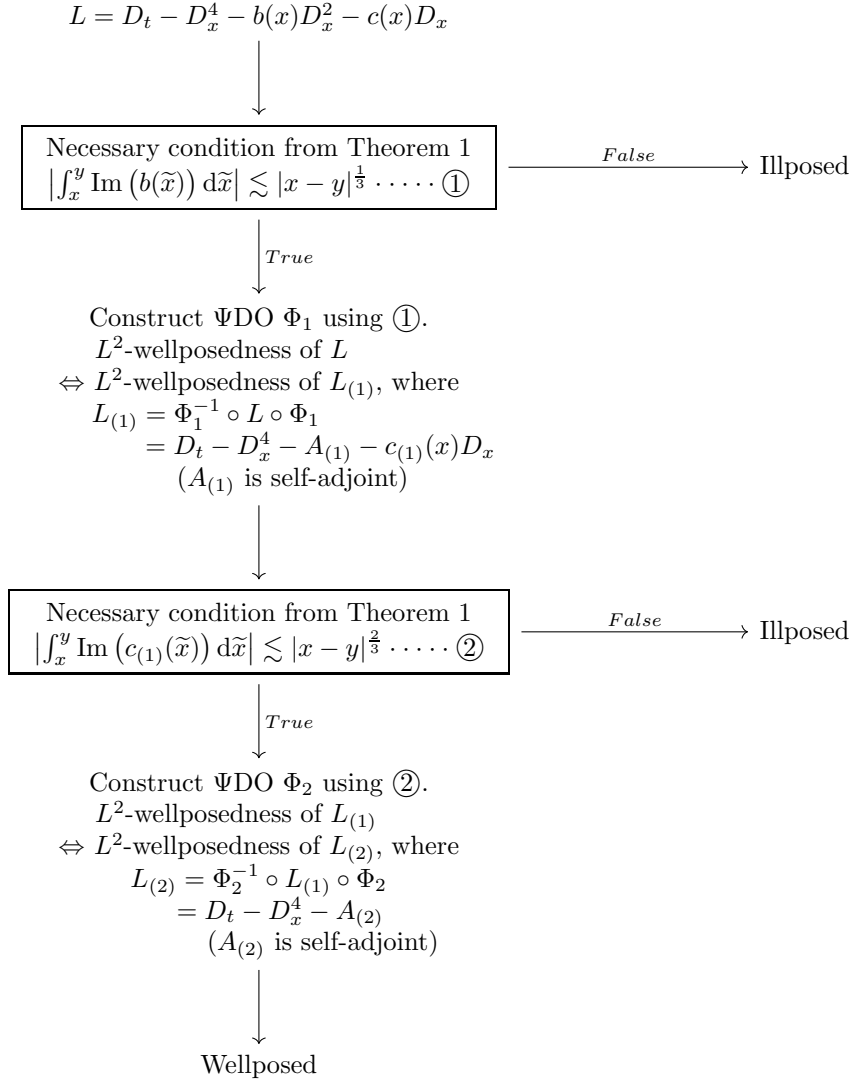

As a specific example, let's consider the $L^{2}$-wellposedness of the Cauchy problem \begin{eqnarray}\label{20241007eq1}
\begin{cases}
Lu := (D_{t}-D_{x}^{4}-b(x)D_{x}^{2}-c(x)D_{x}-d(x))u=0 &  (x, t )\in \RR^{2} \\
u(x, t=0)=u_{0}(x) & x \in \RR
\end{cases},
\end{eqnarray} where $b, c, d\in \overline{C}^{\infty}(\RR)$. This is \eqref{20241006eq1} in case $k=4$.

First of all, if \eqref{20241007eq1} is $L^{2}$-wellposed, then using duality testing argument twice gives \begin{eqnarray}\label{20241007eq2}
\left | \int_{x}^{y} b(\widetilde{x})\dd\widetilde{x} \right | \lesssim |x-y|^{\frac{1}{3}}.
\end{eqnarray} Assume this condition from now on. Define two formal $\Psi$DOs \[
\Phi_{0, 1}(x, \xi) := \frac{1}{4\xi}\int_{-\infty}^{x}\chi\left ( \frac{y-x}{\xi^{3}}\right )\tIm \big (b(y)\big)\dd y \quad \text{and }\quad \Phi_{1}:=e^{\Phi_{0, 1}}\left ( 1 + \frac{1}{4\xi}\partial_{\xi}\Phi_{0, 1} \tRe\big(b(x)\big) \right ).
\] Then using \eqref{20241007eq2}, one has $\Phi_{0, 1}, \Phi_{1} \in S^{0}$ at least formally. Hence $\Phi_{1}$ is an $L^{2}$-bounded operator, and assume that the inverse of $\Phi_{1}$ is also $L^{2}$-bounded. Then the $L^{2}$-wellposedness of \eqref{20241007eq1} turns down into an $L^{2}$-wellposedness problem for $L_{(1)}u=0$, where \[
L_{(1)} := \Phi_{1}^{-1} \circ L \circ \Phi_{1} \equiv D_{t}-D_{x}^{4}-\tRe\big(b(x)\big)D_{x}^{2}-\left ( c(x) - \frac{3}{2}\tIm\big(b'(x)\big) \right )D_{x} \pmod{S^{0}}.
\] Observe that $A_{(1)}:=\tRe\big(b(x)\big)D_{x}^{2}-i\tRe\big(b'(x)\big)D_{x}$ is self-adjoint. Hence, \[
L_{(1)} \equiv D_{t}-D_{x}^{4}-A_{(1)}-\left ( c(x) - \frac{3}{2}\tIm\big(b'(x)\big)+i\tRe\big(b'(x)\big) \right )D_{x} =: D_{t}-D_{x}^{4}-A_{(1)}-c_{(1)}(x)D_{x} \pmod{S^{0}}.
\] Now by duality testing argument again (Theorem \ref{20241015thm1}), for $L_{(1)}$ to be $L^{2}$-wellposed we must have \begin{eqnarray}\label{20241007eq3}
\left | \int_{x}^{y} c_{(1)}(\widetilde{x})\dd\widetilde{x} \right | \lesssim |x-y|^{\frac{2}{3}}.
\end{eqnarray} Assume that this condition is true and define two formal $\Psi$DOs \[
\Phi_{0,2}(x, \xi) := \frac{1}{4\xi^{2}}\int_{-\infty}^{x} \chi\left(\frac{y-x}{\xi^{3}}\right)\tIm\big(c_{(1)}(y)\big)\dd y \quad \text{and }\quad \Phi_{2} := e^{\Phi_{0, 2}} \left ( 1 + \frac{1}{4\xi}\partial_{\xi}\Phi_{0,1}\tRe\big(b(x)\big)\right ).
\] Then from \eqref{20241007eq3} one formally has $\Phi_{0, 2}, \Phi_{2} \in S^{0}$, and assuming that $\Phi_{2}$ has bounded inverse in $L^{2}$, the $L^{2}$-wellposedness of $L_{(1)}u=0$ is equivalent to that of $L_{(2)}u=0$ where \[
L_{(2)} := \Phi_{2}^{-1} \circ L \circ \Phi_{2} \equiv D_{t} - D_{x}^{4}-A_{(1)}-\tRe\big(c_{(1)}(x)\big)D_{x} \pmod{S^{0}}.
\] Since $D_{x}^{4}+A_{(1)}+\tRe\big(c_{(1)}(x)\big)D_{x}$ is self-adjoint, the $L^{2}$-wellposedness of $L_{(2)}u=0$ is now trivial. Thus, the $L^{2}$-wellposedness of the original Cauchy problem \eqref{20241007eq1} is equivalent to a set of two conditions \eqref{20241007eq2} and \eqref{20241007eq3}.

\subsection{Organization of the paper}

The rest of this paper is organized as follows. Section \ref{sec2} is devoted to Theorem \ref{20241015thm1}, which gives a necessary condition for the $L^{2}$-wellposedness of \eqref{20241006eq1}. In Sections \ref{sec3} and \ref{sec4}, we study equivalent conditions for $L^{2}$-wellposedness of \eqref{20241006eq1} in case $k=5$ (Theorem \ref{20241015thm2}) and $k=6$ (Theorem \ref{20241015thm3}) respectively, where some calculations are skipped. These calculations are presented in the Appendix \ref{appB}. In Appendix \ref{appA}, we prsent the proof of $L^{2}$-wellposedness for the self-adjoint case using pseudodifferential operators.

\subsection{Acknowledgments}

The author would like to thank In-Jee Jeong for introducing the problem and providing thoughtful discussions. The author also thanks Hideaki Sunagawa for suggesting Mizuhara's and Tarama's works as references. The author was supported by the 2024 Undergraduate Research Internship Program from College of Natural Sciences, Seoul National University.

\section{Illposedness criteria}\label{sec2}

First of all, we record one calculation result as a lemma. Similar calculations can be found in \cite[Proposition 2.1]{Tarama11}, but our closer scrutiny reveals more specific information about the polynomials $\left \{P_{k-2-m;\ell, j}\right \}$ defined below.

Recall that for an infinite set $S$, a polynomial ring $\CC[S]$ with infinitely many indeterminates is defined by the set of finite linear combination of monomials, where each monomials are expressed as a finite multiplication of elements in $S$. Let $\CC\Big[\big\{x_{\alpha}^{(\beta)}\big\}_{0 \le \alpha \le k-2, \beta \ge 0}\Big]$ be a polynomial ring with indeterminates $\big\{x_{\alpha}^{(\beta)}: 0 \le \alpha \le k-2, \beta \ge 0\big\}$. An element in this ring will be denoted by $P\Big (\big \{x_{\alpha}^{(\beta)}\big \}_{0 \le \alpha \le k-2, \beta \ge 0}\Big )$. This polynomial evaluated at $\big\{x_{\alpha}^{(\beta)}\big\}=\big\{y_{\alpha}^{(\beta)}\big\}$ will be denoted by $P\Big(\big(y_{\alpha}^{(\beta)}\big)_{0 \le \alpha \le k-2, \beta \ge 0}\Big)$, i.e., if the variables in $P$ are enclosed by curly braces then the variables are used as indeterminates to express $P$, and if the variables in $P$ are enclosed by parentheses then it is an evaluation of $P$ at those variables. For example, if $P_{1}\Big(\big\{x_{\alpha}^{(\beta)}\big\}_{0 \le \alpha \le k-2, \beta \ge 0}\Big)=x_{0}^{(2)}+x_{1}^{(1)}x_{2}^{(0)}$, then $P_{1}\Big(\big(\partial_{x}^{\beta}b_{\alpha}(x)\big)_{0 \le \alpha \le k-2, \beta \ge 0}\Big)=\partial_{x}^{2}b_{0}+(\partial_{x}b_{1})b_{2}$. Note that if $P\Big(\big\{x_{\alpha}^{(\beta)}\big\}_{0 \le \alpha \le k-2, \beta \ge 0}\Big) \in \CC\Big[\big\{x_{\alpha}^{(\beta)}\big\}_{0 \le \alpha \le k-2, \beta \ge 0}\Big]$, then $P\Big(\big(\partial_{x}^{\beta}b_{\alpha}(x)\big)_{0 \le \alpha \le k-2, \beta \ge 0}\Big)$ is a function in $x$, which is actually in $\overline{C}^{\infty}(\RR)$.

For convenience, define the following polynomial rings
\begin{IEEEeqnarray*}{rCl}
\cP_{p, q} &:=& \begin{cases}\left \{ P(\{x_{\alpha}^{(\beta)}\}_{p \le \alpha \le q, \beta \ge 0}): P \in \CC[\{x_{\alpha}^{(\beta)}\}_{p \le \alpha \le q, \beta \ge 0}] \subset \CC[\{x_{\alpha}^{(\beta)}\}_{0 \le \alpha \le k-2, \beta \ge 0}] \right \} & \text{ if }0 \le p \le q \le k-2, \\
\CC & \text{ if }p>q, \\
\cP_{\max\{p, 0\}, \min\{q, k-2\}} & \text{ otherwise.} \\
\end{cases}
\end{IEEEeqnarray*}
So, an element of $\cP_{p, q}$ is a polynomial in $ \CC[\{x_{\alpha}^{(\beta)}\}_{0 \le \alpha \le k-2, \beta \ge 0}]$ which is independent of all the $x_{\alpha}^{(\beta)}$ with $\alpha<p$ or $\alpha>q$. In particular, $\CC[\{x_{\alpha}^{(\beta)}\}_{0 \le \alpha \le k-2, \beta \ge 0}]=\cP_{0, k-2}$. Now we are ready to state our lemma.

\begin{lem} \label{20241007lem1}
Fix an integer $k \ge 2$. Then there exist polynomials $P_{k-2-m;\ell, j} \in  \cP_{0,k-2}$ $(0 \le m \le k-2, \; -\frac{(m+1)mk}{2} \le \ell \le k-2, \; 0 \le j \le k)$ which satisfy the following four properties.
\begin{itemize}
\item[(i)] Let $b_{j}(0 \le j \le k-2)$ be arbitrary $\overline{C}^{\infty}(\RR)$ functions and let \[
L=D_{t}-D_{x}^{k}-\sum_{j=0}^{k-2}b_{j}D_{x}^{j}.
\] For each integer $0 \le m \le k-2$, define \[
\widetilde{\psi}_{k-2-m}(\xi; x, t):= ix\xi+it\xi^{k}-\sum_{q=1}^{m}\frac{i}{k\xi^{q}}\int_{0}^{x}P_{k-1-q;k-1-q,0}\Big(\big(\partial_{x}^{\beta}\overline{b}_{\alpha}(\widetilde{x})\big)_{0 \le \alpha \le k-2, \beta \ge 0}\Big)\dd\widetilde{x}.
\] Then for any $0 \le m \le k-2$, \[
e^{-\widetilde{\psi}_{k-2-m}} \circ L^{\ast} \circ e^{\widetilde{\psi}_{k-2-m}}=D_{t}-k\xi^{k-1}D_{x} - \sum_{\ell=-\frac{(m+1)mk}{2}}^{k-2}\sum_{j=0}^{k}\xi^{\ell}P_{k-2-m;\ell, j}\Big(\big(\partial_{x}^{\beta}\overline{b}_{\alpha}(x)\big)_{0 \le \alpha \le k-2, \beta \ge 0}\Big)D_{x}^{j}.
\] 
\item[(ii)] For any $0 \le m \le k-2$ and $k-m-1 \le \ell \le k-2$, $P_{k-2-m;\ell, 0} \equiv 0$.
\item[(iii)] For each $0 \le m \le m' \le k-2$, $P_{k-2-m; k-2-m', 0}\Big(\big\{x_{\alpha}^{(\beta)}\big\}_{0 \le \alpha \le k-2, \beta \ge 0}\Big)-x_{k-2-m'}^{(0)}$ is a polynomial in $\left \{x_{\alpha}^{(\beta)}: k-1-m' \le \alpha \le k-2, \; \beta \ge 0\right \}$. In other words, \[
P_{k-2-m; k-2-m', 0}\Big(\big\{x_{\alpha}^{(\beta)}\big\}_{0 \le \alpha \le k-2, \beta \ge 0}\Big)-x_{k-2-m'}^{(0)} \in \cP_{k-1-m', k-2}.
\]
\item[(iv)] For any $0 \le m \le k-2$, $0 \le \ell \le k-2$, and $1 \le j \le k$, \[
P_{k-2-m;\ell,j} \in {\cP}_{\min\{j+\ell,k-1-m\}, k-2}.
\]
\end{itemize}
\end{lem}

\begin{proof}
We will prove the lemma by induction on $m$. For each $m$, we will define $P_{k-2-m;\ell, j}$ so that (i) holds, and then prove that they satisfy (ii), (iii) and (iv).

We start by calculating $L^{\ast}$. Observe that 
\begin{IEEEeqnarray*}{rCl}
L^{\ast} &=& \left ( D_{t}-D_{x}^{k}-\sum_{\ell=0}^{k-2}b_{\ell}(x)D_{x}^{\ell} \right )^{\ast} \\
&=& D_{t} - D_{x}^{k} - \sum_{\ell=0}^{k-2} \sum_{j=0}^{\ell} \binom{\ell}{j}D_{x}^{\ell-j}\big (\overline{b}_{\ell}(x)\big )D_{x}^{j} \\
&=& D_{t} - D_{x}^{k} - \sum_{j=0}^{k-2}\left (\sum_{\ell=j}^{k-2}\binom{\ell}{j}D_{x}^{\ell-j}\big (\overline{b}_{\ell}(x)\big )\right )D_{x}^{j} \\
&=:& D_{t}-D_{x}^{k}-\sum_{j=0}^{k-2}B_{j}\Big (\big (\partial_{x}^{\beta}\overline{b}_{\alpha}(x)\big)_{0 \le \alpha \le k-2, \beta \ge 0}\Big)D_{x}^{j},
\end{IEEEeqnarray*}
where $B_{j} \in \cP_{0, k-2}$ $(0 \le j \le k-2)$ are polynomials which depend only on $k$. Note also that \begin{eqnarray}\label{20241008eq3}
B_{j}\Big (\big\{x_{\alpha}^{(\beta)}\big\}_{0 \le \alpha \le k-2, \beta \ge 0}\Big)-x_{j}^{(0)} \in \cP_{j+1, k-2} \quad \forall \; 0 \le j \le k-2.
\end{eqnarray} Now we start our induction on $m$. If $m=0$, then $\widetilde{\psi}_{k-2}(\xi, x, t)=ix\xi+it\xi^{k}$, so that 
\begin{IEEEeqnarray*}{rCl}
L^{\ast}_{k-2}&:=&e^{-\widetilde{\psi}_{k-2}} \circ L^{\ast} \circ e^{\widetilde{\psi}_{k-2}} \\
&=& (e^{-ix\xi}e^{-it\xi^{k}}) \circ \left (D_{t}-D_{x}^{k}-\sum_{j=0}^{k-2}B_{j}\Big (\big (\partial_{x}^{\beta}\overline{b}_{\alpha}(x)\big)_{0 \le \alpha \le k-2, \beta \ge 0}\Big)D_{x}^{j}\right ) \circ (e^{ix\xi}e^{it\xi^{k}}) \\
&=& \xi^{k} + D_{t} - \sum_{\ell=0}^{k}\binom{k}{\ell}\xi^{\ell}D_{x}^{k-\ell} - \sum_{j=0}^{k-2}B_{j}\Big (\big (\partial_{x}^{\beta}\overline{b}_{\alpha}(x)\big)_{0 \le \alpha \le k-2, \beta \ge 0}\Big)\sum_{\ell=0}^{j}\binom{j}{\ell}\xi^{\ell}D_{x}^{j-\ell} \\
&=& D_{t} - k\xi^{k-1}D_{x} - \sum_{\ell=0}^{k-2}\binom{k}{\ell}\xi^{\ell}D_{x}^{k-\ell} - \sum_{\ell=0}^{k-2}\sum_{j=0}^{k-2-\ell}B_{j+\ell}\Big (\big (\partial_{x}^{\beta}\overline{b}_{\alpha}(x)\big)_{0 \le \alpha \le k-2, \beta \ge 0}\Big)\binom{j+\ell}{\ell}\xi^{\ell}D_{x}^{j}  \\
&=:& D_{t} - k\xi^{k-1}D_{x} - \sum_{\ell=0}^{k-2}\sum_{j=0}^{k} \xi^{\ell} P_{k-2; \ell, j}\Big (\big (\partial_{x}^{\beta}\overline{b}_{\alpha}(x)\big)_{0 \le \alpha \le k-2, \beta \ge 0}\Big)D_{x}^{j}\yesnumber\label{20241008eq2}
\end{IEEEeqnarray*}
Here, the last line \eqref{20241008eq2} defines the polynomials $P_{k-2;\ell, j}$ by considering $\partial_{x}^{\beta}\overline{b}_{\alpha}$'s as independent indeterminates. These polynomials are dependent only on $k$. Now we check (ii), (iii), and (iv), for $m=0$.
\begin{itemize}
\item[(ii)] Since $m=0$, there is nothing to check.
\item[(iii)] Let $0 \le m' \le k-2$. Then from \eqref{20241008eq2}, \[
P_{k-2;k-2-m',0}\Big (\big\{x_{\alpha}^{(\beta)}\big\}_{0 \le \alpha \le k-2, \beta \ge 0}\Big)=B_{k-2-m'}\Big( \big\{x_{\alpha}^{(\beta)}\big\}_{0 \le \alpha \le k-2, \beta \ge 0}\Big).
\]
Hence by \eqref{20241008eq3}, $P_{k-2;k-2-m',0}\Big (\big\{x_{\alpha}^{(\beta)}\big\}_{0 \le \alpha \le k-2, \beta \ge 0}\Big)-x_{k-2-m'}^{(0)} \in \cP_{k-1-m', k-2}$, as desired.
\item[(iv)] For any $0 \le \ell \le k-2$ and $1 \le j \le k-2-\ell$, \eqref{20241008eq2} gives \[
P_{k-2;\ell, j}\Big (\big\{x_{\alpha}^{(\beta)}\big\}_{0 \le \alpha \le k-2, \beta \ge 0}\Big)=B_{j+\ell}\Big (\big\{x_{\alpha}^{(\beta)}\big\}_{0 \le \alpha \le k-2, \beta \ge 0}\Big)\binom{j+\ell}{\ell} \in \cP_{j+\ell, k-2}.
\] If $0 \le \ell \le k-2$ and $j \ge k-1-\ell$, then \eqref{20241008eq2} shows that $P_{k-2;\ell, j}$ is either zero or a constant polynomial. Hence $P_{k-2;\ell, j} \in \CC = \cP_{k-1,k-2} = \cP_{\min\{j+\ell,k-1\}, k-2}$.
\end{itemize}
Now assume that for some $1 \le m_{0} \le k-2$, $P_{k-1-m_{0};\ell, j}$ $(-\frac{m_{0}(m_{0}-1)k}{2} \le \ell \le k-2, \; 0 \le j \le k)$ are defined so that (i)-(iv) holds. In particular, we have \begin{IEEEeqnarray*}{rCl}
L^{\ast}_{k-1-m_{0}}&:=&e^{-\widetilde{\psi}_{k-1-m_{0}}}\circ L^{\ast} \circ e^{\widetilde{\psi}_{k-1-m_{0}}} \\
&=& D_{t}-k\xi^{k-1}D_{x} - \sum_{\ell=-\frac{m_{0}(m_{0}-1)k}{2}}^{k-2}\sum_{j=0}^{k}\xi^{\ell}P_{k-1-m_{0};\ell, j}\Big(\big(\partial_{x}^{\beta}\overline{b}_{\alpha}(x)\big)_{0 \le \alpha \le k-2, \beta \ge 0}\Big)D_{x}^{j}. \yesnumber \label{20241008eq4}
\end{IEEEeqnarray*} Now consider the case $m=m_{0}$. Noting that \[
\widetilde{\psi}_{k-2-m}=\widetilde{\psi}_{k-1-m}-\frac{i}{k\xi^{m}}\int_{0}^{x}P_{k-1-m;k-1-m,0}\Big ( \big (\partial_{x}^{\beta}\overline{b}_{\alpha}(\widetilde{x})\big )_{0 \le \alpha \le k-2, \beta \ge 0} \Big )\dd\widetilde{x}=:\widetilde{\psi}_{k-1-m} - \frac{1}{\xi^{m}}\phi_{k-1-m}(x),
\] one can calculate $L_{k-2-m}^{\ast}:=e^{-\widetilde{\psi}_{k-2-m}} \circ L^{\ast} \circ e^{\widetilde{\psi}_{k-2-m}}$ starting from \eqref{20241008eq4} as follows.
\begin{IEEEeqnarray*}{rCl}
\IEEEeqnarraymulticol{3}{l}{e^{-\widetilde{\psi}_{k-2-m}} \circ L^{\ast} \circ e^{\widetilde{\psi}_{k-2-m}}}\\
&=& \exp \left ( \frac{1}{\xi^{m}}\phi_{k-1-m}(x)\right ) \circ  L^{\ast}_{k-1-m} \circ
\exp\left ( -\frac{1}{\xi^{m}} \phi_{k-1-m}(x)\right ) \\
&=& D_{t}-k\xi^{k-1}D_{x} + \xi^{k-1-m} P_{k-1-m;k-1-m,0}\Big(\big(\partial_{x}^{\beta}\overline{b}_{\alpha}(x)\big)_{0 \le \alpha \le k-2, \beta \ge 0}\Big) \\
&& - \sum_{\ell=-\frac{m(m-1)k}{2}}^{k-2}\sum_{j=0}^{k}\bigg [\xi^{\ell}P_{k-1-m_;\ell, j}\Big(\big(\partial_{x}^{\beta}\overline{b}_{\alpha}(x)\big)_{0 \le \alpha \le k-2, \beta \ge 0}\Big) \\
&& \qquad \qquad \qquad \qquad \;\;\sum_{p=0}^{j} \binom{j}{p} \exp \left (\frac{1}{\xi^{m}}\phi_{k-1-m}(x)\right )D_{x}^{j-p}\left ( \exp \left (-\frac{1}{\xi^{m}}\phi_{k-1-m}(x)\right )\right ) D_{x}^{p}\bigg ] \\
&=& D_{t}-k\xi^{k-1}D_{x} + \xi^{k-1-m} P_{k-1-m;k-1-m,0}\Big(\big(\partial_{x}^{\beta}\overline{b}_{\alpha}(x)\big)_{0 \le \alpha \le k-2, \beta \ge 0}\Big) \\
&& - \sum_{\ell=-\frac{m(m-1)k}{2}}^{k-2}\xi^{\ell}P_{k-1-m;\ell, 0}\Big(\big(\partial_{x}^{\beta}\overline{b}_{\alpha}(x)\big)_{0 \le \alpha \le k-2, \beta \ge 0}\Big)  \\
&& - \sum_{\ell=-\frac{m(m-1)k}{2}}^{k-2}\sum_{j=1}^{k}\sum_{p=0}^{j-1}\xi^{\ell}P_{k-1-m;\ell, j}\Big(\big(\partial_{x}^{\beta}\overline{b}_{\alpha}(x)\big)_{0 \le \alpha \le k-2, \beta \ge 0}\Big) \binom{j}{p} \sum_{q=1}^{j-p}\xi^{-mq} \underbrace{c_{j-p, q} \Big ( \big ( -\partial_{x}^{\gamma}\phi_{k-1-m}(x)\big )_{\gamma \ge 1} \Big )}_{(\ast)} D_{x}^{p} \\
&& - \sum_{\ell=-\frac{m(m-1)k}{2}}^{k-2}\sum_{j=1}^{k}\xi^{\ell}P_{k-1-m;\ell, j}\Big(\big(\partial_{x}^{\beta}\overline{b}_{\alpha}(x)\big)_{0 \le \alpha \le k-2, \beta \ge 0}\Big) D_{x}^{j} \\
&=:& D_{t}-k\xi^{k-1}D_{x} - \sum_{\ell=-\frac{(m+1)mk}{2}}^{k-2}\sum_{j=0}^{k}\xi^{\ell}P_{k-2-m;\ell, j}\Big(\big(\partial_{x}^{\beta}\overline{b}_{\alpha}(x)\big)_{0 \le \alpha \le k-2, \beta \ge 0}\Big)D_{x}^{j}. \yesnumber \label{20241008eq5}
\end{IEEEeqnarray*}
Here, the penultimate line comes from the following expansion formula \[
\exp \left ( -\frac{1}{\xi^{m}}a(x) \right )D_{x}^{p}\left ( \exp \left ( \frac{1}{\xi^{m}} a(x) \right )\right )=\sum_{q=1}^{p}\frac{1}{\xi^{mq}}c_{p, q}\Big ( \big ( \partial_{x}^{\beta}a(x)\big )_{\beta \ge 1}\Big ) \quad (p \ge 1)
\] where $c_{p, q} \in \CC\Big[\big\{x^{(\beta)}\big\}_{\beta \ge 1}\Big]$ are polynomials depending only on $p, q$. (In fact, these $c_{p, q}$ are given by Bell polynomials.) Note that since $\partial_{x}\phi_{k-1-m}(x)=\frac{i}{k}P_{k-1-m;k-1-m,0}\Big(\big(\partial_{x}^{\beta}\overline{b}_{\alpha}(x)\big)_{0 \le \alpha \le k-2,\beta \ge 0}\Big )$ and $P_{k-1-m;k-1-m,0} \in \cP_{k-1-m,k-2}$ (inductive hypothesis), $(\ast)$ can be interpreted as a polynomial $Q_{k-1-m, j-p, q} \in\cP_{k-1-m,k-2}$ evaluated at $\big(\partial_{x}^{\beta}\overline{b}_{\alpha}(x)\big)_{0 \le \alpha \le k-2, \beta \ge 0}$. Hence \eqref{20241008eq5} defines the polynomials $P_{k-2-m;\ell, j}$ by considering $\partial_{x}^{\beta}\overline{b}_{\alpha}$'s as independent indeterminates. Now we are left to check (ii), (iii), and (iv).
\begin{itemize}
\item[(ii)] Let $k-m-1 \le \ell \le k-2$. Then from \eqref{20241008eq5}, we have \[
P_{k-2-m;\ell, 0}=\begin{cases}
P_{k-1-m;\ell, 0} & \text{ if }\ell \ge k-m,\\
P_{k-1-m;\ell, 0}-P_{k-1-m;k-1-m,0}=0 & \text{ if }\ell=k-m-1
\end{cases}
\] Then by induction hypothesis (ii) for $m-1$, we have $P_{k-2-m;\ell, 0}\equiv 0$.
\item[(iii)] Let $m \le m' \le k-2$. Then from \eqref{20241008eq5}, \[
P_{k-2-m;k-2-m',0}=P_{k-1-m;k-2-m',0}+\sum_{q=1}^{\lfloor \frac{m'}{m}\rfloor}\sum_{j=q}^{k}P_{k-1-m;k-2-m'+mq, j} Q_{k-1-m,j,q}.
\] By induction hypothesis (iv) for $m-1$, $P_{k-1-m;k-2-m'+mq,j} \in \cP_{\min\{k-2-m'+mq+j, k-m\}, k-2} \subset \cP_{k-m', k-2}$. We also have $Q_{k-1-m,j,q} \in \cP_{k-1-m,k-2}$, and by induction hypothesis (iii) for $m-1$, we have that $P_{k-1-m;k-2-m',0}\Big ( \big \{ x_{\alpha}^{(\beta)}\big \}_{0 \le \alpha \le k-2, \beta \ge 0}\Big )-x_{k-2-m'}^{(0)} \in \cP_{k-1-m',k-2}$. By combining all the information we conclude that \[
P_{k-2-m;k-2-m',0}\Big ( \big \{ x_{\alpha}^{(\beta)}\big \}_{0 \le \alpha \le k-2, \beta \ge 0}\Big )-x_{k-2-m'}^{(0)} \in \cP_{k-1-m',k-2},
\] as desired.
\item[(iv)] Let $0 \le \ell \le k-2$ and $1 \le j \le k$. Then from \eqref{20241008eq5}, one can calculate \[
P_{k-2-m;\ell, j} = \sum_{q=1}^{\min\{k-j,\lfloor \frac{k-\ell-2}{m}\rfloor\}}\sum_{p=j+q}^{k}\binom{p}{j}P_{k-1-m;\ell+mq, p}  Q_{k-1-m,p-j,q}+ P_{k-1-m;\ell, j}
\]
By induction hypothesis (iv) for $m-1$, $P_{k-1-m;\ell+mq+p} \in \cP_{\min\{\ell+mq+p,k-m\},k-2} \subset \cP_{\min\{\ell+j+m+1,k-m\},k-2}$. We also have $Q_{k-1-m,p-j,q} \in \cP_{k-1-m,k-2}$, and by induction hypothesis (iv) for $m-1$, $P_{k-1-m;\ell, j} \in \cP_{\min\{\ell+j,k-m\},k-2}$. Combining all the information allows us to conclude that \[
P_{k-2-m;\ell, j} \in \cP_{\min\{\ell+j, k-1-m\},k-2},
\] as desired.
\end{itemize}
This completes the inductive step and hence proves the Lemma.
\end{proof}

Using Lemma \ref{20241007lem1}, we can now construct approximate solutions to $L^{\ast}v=0$ and use duality testing arguments to obtain the necessary conditions for the $L^{2}$-wellposedness of \eqref{20241006eq1}.

\begin{prop} \label{20241008prop1}
Let $k \ge 2$ be a given integer and let $P_{k-2-m;\ell, j} \in \cP_{0,k-2}$ $(0 \le m \le k-2, \; -\frac{(m+1)mk}{2} \le \ell \le k-2, \; 0 \le j \le k)$ be polynomials as in Lemma \ref{20241007lem1}. Let $b_{j}(0 \le j \le k-2)$ be any given $\overline{C}^{\infty}(\RR)$ functions. Assume that the Cauchy problem \begin{eqnarray}\label{20241007eq5}
\begin{cases}
Lu:=\left ( D_{t}-D_{x}^{k}-\sum_{j=0}^{k-2}b_{j}(x)D_{x}^{j} \right )u =0 & (x, t) \in \RR^{2} \\
u(x, t=0)=u_{0}(x) & x \in \RR
\end{cases}
\end{eqnarray} is $L^{2}$-wellposed. Then, we have \[
\left | \int_{x}^{y}\tIm\bigg(P_{k-q;k-q,0}\Big ( \big ( \partial_{x}^{\beta}\overline{b}_{\alpha}(\widetilde{x})\big )_{0 \le \alpha \le k-2, \beta \ge 0}\Big )\bigg)\dd\widetilde{x} \right | \lesssim_{L} |x-y|^{\frac{q-1}{k-1}} \quad \forall \; 2 \le q \le k-1.
\]
\end{prop}

\begin{proof}
We will prove by induction that for any integer $1 \le m \le k-1$, the following estimates hold: \begin{eqnarray}\label{20241007eq9}
\left | \int_{x}^{y}\tIm\bigg(P_{k-q;k-q,0}\Big ( \big ( \partial_{x}^{\beta}\overline{b}_{\alpha}(\widetilde{x})\big )_{0 \le \alpha \le k-2, \beta \ge 0}\Big )\bigg)\dd\widetilde{x} \right | \lesssim_{L} |x-y|^{\frac{q-1}{m}} \quad \forall \; 2 \le q \le m.
\end{eqnarray}
There is nothing to check for the base case $m=1$. Now assume that \eqref{20241007eq9} holds for some integer $1 \le m \le k-2$.

Define \begin{IEEEeqnarray*}{rCl}
\widetilde{\psi}_{k-2-m}(\xi; x, t)&:=& ix\xi+it\xi^{k}+\psi_{k-2-m}(\xi; x)\\
&:=& ix\xi+it\xi^{k}-\sum_{q=1}^{m}\frac{i}{k\xi^{q}}\int_{0}^{x}P_{k-1-q;k-1-q,0}\Big ( \big (\partial_{x}^{\beta}\overline{b}_{\alpha}(\widetilde{x})\big )_{0 \le \alpha \le k-2, \beta \ge 0} \Big )\dd\widetilde{x}, \\
v(\xi; x, t)&:=&e^{\widetilde{\psi}_{k-2-m}(\xi; x, t)} f\left (\frac{x-x_{1}+k\xi^{k-1}t}{\xi^{m}}\right ), \yesnumber\label{20241007eq6}
\end{IEEEeqnarray*}
 where $x_{1} \in \RR$ is any fixed real number and $f \in C_{c}^{\infty}(\RR)$ is a fixed nontrivial function which is supported in $(-1, 1)$. Then, \begin{eqnarray}\label{20241007eq7}
v_{0}(\xi; x):=v(\xi;x, t=0)=e^{ix\xi}e^{\psi_{k-2-m}(\xi; x)}f\left (\frac{x-x_{1}}{\xi^{m}}\right ).
\end{eqnarray} In addition, by Lemma \ref{20241007lem1} we have
\begin{IEEEeqnarray*}{rCl}
\IEEEeqnarraymulticol{3}{l}{e^{-\widetilde{\psi}_{k-2-m}} L^{\ast}[v]} \\
&=& \left (D_{t}-k\xi^{k-1}D_{x} - \sum_{\ell=-\frac{(m+1)mk}{2}}^{k-2}\sum_{j=0}^{k}\xi^{\ell}P_{k-2-m;\ell, j}\Big(\big(\partial_{x}^{\beta}\overline{b}_{\alpha}(x)\big)_{0 \le \alpha \le k-2, \beta \ge 0}\Big)D_{x}^{j}\right ) \left [ f \left ( \frac{x-x_{1}+k\xi^{k-1}t}{\xi^{m}} \right )\right ] \\
&=& - \sum_{\ell=-\frac{(m+1)mk}{2}}^{k-2}\sum_{j=0}^{k}\xi^{\ell-mj}P_{k-2-m;\ell, j}\Big(\big(\partial_{x}^{\beta}\overline{b}_{\alpha}(x)\big)_{0 \le \alpha \le k-2, \beta \ge 0}\Big)\big(D_{x}^{j}f\big)\left ( \frac{x-x_{1}+k\xi^{k-1}t}{\xi^{m}} \right ) \\
&=& \sum_{\ell=-\frac{(m+3)mk}{2}}^{k-2-m}\sum_{j=0}^{k} \xi^{\ell} e_{k-2-m;\ell, j}(x)\big(\partial_{x}^{j}f\big)\left (\frac{x-x_{1}+k\xi^{k-1}t}{\xi^{m}}\right ),
\end{IEEEeqnarray*}
where $e_{k-2-m;\ell, j}$ are $\overline{C}^{\infty}(\RR)$ functions determined by $L$. (In the last line we have used Lemma \ref{20241007lem1}(ii), which states that $P_{k-2-m;\ell,0}\equiv 0$ for $k-m-1 \le \ell \le k-2$.) From this we can estimate the $L^{2}$-norm of $L^{\ast}v$ as follows.
\begin{IEEEeqnarray*}{rCl}
\|(L^{\ast}v)(\xi;\boldsymbol{\cdot},t)\|_{L^{2}} &=& \left \| e^{\widetilde{\psi}_{k-2-m}}\sum_{\ell=-\frac{(m+3)mk}{2}}^{k-2-m}\sum_{j=0}^{k} \xi^{\ell} e_{k-2-m;\ell, j}(x)\big (\partial_{x}^{j}f\big)\left (\frac{x-x_{1}+k\xi^{k-1}t}{\xi^{m}}\right )\right \|_{L^{2}} \\
& \lesssim &\hspace{-0.1cm}{}_{L}\, \left \| e^{\tRe\big({\psi}_{k-2-m}\big)}\sum_{\ell=-\frac{(m+3)mk}{2}}^{k-2-m}\sum_{j=0}^{k} \xi^{\ell} \big (\partial_{x}^{j}f\big)\left (\frac{x-x_{1}+k\xi^{k-1}t}{\xi^{m}}\right )\right \|_{L^{2}}\\
& \lesssim &\hspace{-0.1cm}{}_{L}\, \left (|\xi|^{-k^{3}}+|\xi|^{k-2-m}\right )\sum_{j=0}^{k} \left \|  e^{\tRe\big({\psi}_{k-2-m}\big)}\big (\partial_{x}^{j}f\big)\left (\frac{x-x_{1}+k\xi^{k-1}t}{\xi^{m}}\right )\right \|_{L^{2}} \yesnumber \label{20241007eq8}
\end{IEEEeqnarray*}

Now let $u(\xi; x, t)$ be a solution to \eqref{20241007eq5} with $u_{0}=v_{0} \in C_{c}^{\infty}(x_{1}-1, x_{1}+1)$. Since \eqref{20241007eq5} is $L^{2}$-wellposed, we have \[
\|u(t)\|_{L^{2}} \le M\|u_{0}\|_{L^{2}} \quad \forall \;|t| \le T,
\] where $M, T$ are positive constants depending only on $L$. Observe that \[
\partial_{t}\langle u, v \rangle = \langle \partial_{t}u, v \rangle + \langle u, \partial_{t}v \rangle = \langle i(D_{t}-L)u, v \rangle + \langle u, iD_{t}v \rangle = \langle u, iL^{\ast}v\rangle
\] Thus for any $|t| \le T$,
\begin{IEEEeqnarray*}{rCl}
\|u_{0}(\xi;\boldsymbol{\cdot})\|_{L^{2}}^{2} &=& \langle u_{0}(\xi;\boldsymbol{\cdot}), v_{0}(\xi;\boldsymbol{\cdot}) \rangle \\
&\le & \|u(\xi;\boldsymbol{\cdot}, t)\|_{L^{2}} \|v(\xi;\boldsymbol{\cdot}, t)\|_{L^{2}} + \left | \int_{0}^{t}\|u(\xi;\boldsymbol{\cdot},\tau)\|_{L^{2}} \|(L^{\ast}v)(\xi;\boldsymbol{\cdot}, \tau)\|_{L^{2}}\dd\tau\right | \\
& \le & M\|u_{0}(\xi;\boldsymbol{\cdot})\|_{L^{2}} \left ( \|v(\xi;\boldsymbol{\cdot}, t)\|_{L^{2}} + \left | \int_{0}^{t}\|(L^{\ast}v)(\xi;\boldsymbol{\cdot}, \tau)\|_{L^{2}}\dd\tau\right |\right )
\end{IEEEeqnarray*} 
Hence, \[
\|u_{0}(\xi;\boldsymbol{\cdot})\|_{L^{2}} \lesssim_{L}\|v(\xi;\boldsymbol{\cdot},t)\|_{L^{2}} + \left | \int_{0}^{t} \|(L^{\ast}v)(\xi;\boldsymbol{\cdot},\tau)\|_{L^{2}}\dd\tau \right | \quad \forall \;|t| \le T.
\] Plugging \eqref{20241007eq6}, \eqref{20241007eq7}, and \eqref{20241007eq8} into this equation gives
\begin{IEEEeqnarray*}{rCl}
\IEEEeqnarraymulticol{3}{l}{\left \| e^{\tRe\big(\psi_{k-2-m}\big)}f\left ( \frac{x-x_{1}}{\xi^m} \right )\right \|_{L^{2}}} \\
&\lesssim&\hspace{-0.1cm}{}_{L} \left \| e^{\tRe\big(\psi_{k-2-m}\big)} f \left ( \frac{x-x_{1}+k\xi^{k-1}t}{\xi^{m}}\right )\right \|_{L^{2}} \\
&& \quad + \left |\int_{0}^{t} \left (|\xi|^{-k^{3}}+|\xi|^{k-2-m}\right ) \sum_{j=0}^{k} \left \|e^{\tRe\big(\psi_{k-2-m}\big)} \big (\partial_{x}^{j}f\big ) \left ( \frac{x-x_{1}+k\xi^{k-1}\tau}{\xi^{m}}\right )\right \|_{L^{2}}\dd\tau \right | \yesnumber \quad \forall \; |t| \le T.\label{20241007eq10}
\end{IEEEeqnarray*}
Meanwhile, \eqref{20241007eq9} implies that
\begin{IEEEeqnarray*}{rCl}
\IEEEeqnarraymulticol{3}{l}{\left | \tRe\big(\psi_{k-2-m}(\xi; x)\big )-\tRe\big (\psi_{k-2-m}(\xi; y)\big ) \right |} \\
&=& \left | \sum_{q=1}^{m}\frac{1}{k\xi^{q}} \int_{x}^{y}\tIm\bigg(P_{k-1-q;k-1-q,0}\Big ( \big (\partial_{x}^{\beta}\overline{b}_{\alpha}(\widetilde{x})\big )_{0 \le \alpha \le k-2, \beta \ge 0} \Big )\bigg)\dd\widetilde{x}\right | \\
& \lesssim&\hspace{-0.1cm}{}_{L}\, \sum_{q=1}^{m-1}\frac{1}{k|\xi|^{q}}|x-y|^{\frac{q}{m}} + \frac{1}{k|\xi|^{m}}|x-y|.
\end{IEEEeqnarray*} Hence if $|x-y|\le |\xi|^{m}$, then \[
\left | \tRe\big(\psi_{k-2-m}(\xi; x)\big )-\tRe\big (\psi_{k-2-m}(\xi; y)\big ) \right | \lesssim_{L} 1.
\] Using this with the fact $f \in C_{c}^{\infty}(-1, 1)$, it is easy to see that the two $L^{2}$-norms \[
\left \| e^{\tRe\big(\psi_{k-2-m}\big)}f\left ( \frac{x-x_{1}}{\xi^m} \right )\right \|_{L^{2}} \sim_{L} \left \| e^{\tRe\big(\psi_{k-2-m}(\xi;x_{1})\big)}f\left ( \frac{x-x_{1}}{\xi^m} \right )\right \|_{L^{2}}
\] are comparable up to a constant depending only on $L$. Using similar argument simplifies \eqref{20241007eq10} into
\begin{IEEEeqnarray*}{rCl}
\IEEEeqnarraymulticol{3}{l}{e^{\tRe\big(\psi_{k-2-m}(\xi; x_{1})\big)} \left \| f\left ( \frac{x-x_{1}}{\xi^m} \right )\right \|_{L^{2}}} \\
&\lesssim&\hspace{-0.1cm}{}_{L}\, e^{\tRe\big(\psi_{k-2-m}(\xi; x_{1}-k\xi^{k-1}t)\big)}\left \|  f \left ( \frac{x-x_{1}+k\xi^{k-1}t}{\xi^{m}}\right )\right \|_{L^{2}} \\
&& \quad + \left |\int_{0}^{t} \left (|\xi|^{-k^{3}}+|\xi|^{k-2-m}\right ) \sum_{j=0}^{k} e^{\tRe\big(\psi_{k-2-m}(\xi; x_{1}-k\xi^{k-1}\tau)\big )} \left \| \big ( \partial_{x}^{j}f \big ) \left ( \frac{x-x_{1}+k\xi^{k-1}\tau}{\xi^{m}}\right )\right \|_{L^{2}}\dd\tau \right |,\\
\IEEEeqnarraymulticol{3}{l}{e^{\tRe\big(\psi_{k-2-m}(\xi; x_{1})\big)} |\xi|^{\frac{m}{2}}\|f\|_{L^{2}}} \\
&\lesssim&\hspace{-0.1cm}{}_{L}\, e^{\tRe\big(\psi_{k-2-m}(\xi; x_{1}-k\xi^{k-1}t)\big)}|\xi|^{\frac{m}{2}}\|f\|_{L^{2}} \\
&& \quad + \left |\int_{0}^{t} \left (|\xi|^{-k^{3}}+|\xi|^{k-2-m}\right ) \sum_{j=0}^{k} e^{\tRe\big(\psi_{k-2-m}(\xi; x_{1}-k\xi^{k-1}\tau)\big )} |\xi|^{\frac{m}{2}}\|\partial_{x}^{j}f\|_{L^{2}}\dd\tau \right | \quad \forall \; |t| \le T,
\end{IEEEeqnarray*}
and since $f \in C_{c}^{\infty}(\RR)$ was a fixed nontrivial function, we obtain \begin{IEEEeqnarray*}{rCl}
\IEEEeqnarraymulticol{3}{l}{e^{\tRe\big(\psi_{k-2-m}(\xi; x_{1})\big)} } \\
&\lesssim&\hspace{-0.1cm}{}_{L}\, e^{\tRe\big(\psi_{k-2-m}(\xi; x_{1}-k\xi^{k-1}t)\big)} +  \big(|\xi|^{-k^{3}}+|\xi|^{k-2-m}\big)\left |\int_{0}^{t} e^{\tRe\big(\psi_{k-2-m}(\xi; x_{1}-k\xi^{k-1}\tau)\big)} \dd\tau \right | \quad \forall \; |t| \le T. \yesnumber\label{20241007eq11}
\end{IEEEeqnarray*}
Now fix arbitrary real numbers $y_{0} \neq y_{1}$. Choose $t$ so that $y_{1}-k\xi^{k-1}t=y_{0}$. Note that the condition $|t| \le T$ is satisfied if $|\xi| \gg_{L} |y_{0}-y_{1}|^{\frac{1}{k-1}}$. Assume that this condition is true. Choose $t_{0} \in [0, t]$ (or $[t, 0]$ if $t<0$) so that $\tRe\big(\psi_{k-2-m}(\xi; y_{1}-k\xi^{k-1}t_{0})\big)$ is maximized. Recall that in \eqref{20241007eq11}, $x_{1}$ was chosen arbitrarily. So by using \eqref{20241007eq11} for $y_{1}-k\xi^{k-1}t_{0}$ instead of $x_{1}$ and $t-t_{0}$ instead of $t$, one gets
\begin{IEEEeqnarray*}{rCl}
\IEEEeqnarraymulticol{3}{l}{e^{\tRe\big(\psi_{k-2-m}(\xi; y_{1}-k\xi^{k-1}t_{0})\big)}}\\
&\lesssim&\hspace{-0.1cm}{}_{L}\, e^{\tRe\big(\psi_{k-2-m}(\xi; y_{0})\big)} + \big(|\xi|^{-k^{3}}+|\xi|^{k-2-m}\big) \left |\int_{0}^{t-t_{0}}  e^{\tRe\big(\psi_{k-2-m}(\xi; y_{1}-k\xi^{k-1}(t_{0}+\tau))\big)} \dd\tau \right | \\
&\le&  e^{\tRe\big(\psi_{k-2-m}(\xi; y_{0})\big)} + |t| \big(|\xi|^{-k^{3}}+|\xi|^{k-2-m}\big)  e^{\tRe\big(\psi_{k-2-m}(\xi; y_{1}-k\xi^{k-1}t_{0})\big)} \\
&=&  e^{\tRe\big(\psi_{k-2-m}(\xi; y_{0})\big)} + \frac{|y_{1}-y_{0}|}{k|\xi|^{k-1}} \big(|\xi|^{-k^{3}}+|\xi|^{k-2-m}\big) e^{\tRe\big(\psi_{k-2-m}(\xi; y_{1}-k\xi^{k-1}t_{0})\big)}
\end{IEEEeqnarray*}
Hence if $|\xi| \gg_{L} \max\left \{|y_{0}-y_{1}|^{\frac{1}{m+1}}, |y_{0}-y_{1}|^{\frac{1}{k-1}}, 1\right \}$, we have \[
e^{\tRe\big(\psi_{k-2-m}(\xi; y_{1})\big)} \le e^{\tRe\big(\psi_{k-2-m}(\xi; y_{1}-k\xi^{k-1}t_{0})\big)}\lesssim_{L}e^{\tRe\big(\psi_{k-2-m}(\xi; y_{0})\big)}.
\] 
Since this holds for arbitrary $y_{0}$ and $y_{1}$, we can switch $y_{0}$ and $y_{1}$ to get the reversed inequality. Combining the two gives \[
\left | \tRe\big(\psi_{k-2-m}(\xi;y_{1})\big)-\tRe\big(\psi_{k-2-m}(\xi;y_{0})\big) \right | \lesssim_{L} 1,
\] that is,\[
\left |\sum_{q=1}^{m}\frac{1}{k\xi^{q}}\int_{y_{0}}^{y_{1}} \tIm\bigg(P_{k-1-q;k-1-q,0}\Big ( \big (\partial_{x}^{\beta}\overline{b}_{\alpha}(\widetilde{x})\big )_{0 \le \alpha \le k-2, \beta \ge 0} \Big )\bigg)\dd\widetilde{x} \right |\lesssim_{L} 1,
\] for any $|\xi| \ge c_{L} \max\left \{|y_{0}-y_{1}|^{\frac{1}{m+1}}, |y_{0}-y_{1}|^{\frac{1}{k-1}}, 1\right \}$ where $c=c_{L}$ is a constant depending only on $L$. Thus if $|y_{0}-y_{1}|>1$, taking $|\xi| = c|y_{0}-y_{1}|^{\frac{1}{m+1}}$ gives \[
\left | \sum_{q=1}^{m} \frac{1}{c^{q}|y_{0}-y_{1}|^{\frac{q}{m+1}}}\int_{y_{0}}^{y_{1}} \tIm\bigg(P_{k-1-q;k-1-q,0}\Big ( \big (\partial_{x}^{\beta}\overline{b}_{\alpha}(\widetilde{x})\big )_{0 \le \alpha \le k-2, \beta \ge 0} \Big )\bigg)\dd\widetilde{x} \right | \lesssim_{L} 1.
\] Now proceeding as in the proof of \cite[Proposition 2.7]{Tarama11} gives \[
\left | \frac{1}{|y_{0}-y_{1}|^{\frac{q}{m+1}}}\int_{y_{0}}^{y_{1}} \tIm\bigg(P_{k-1-q;k-1-q,0}\Big ( \big (\partial_{x}^{\beta}\overline{b}_{\alpha}(\widetilde{x})\big )_{0 \le \alpha \le k-2, \beta \ge 0} \Big )\bigg)\dd\widetilde{x} \right | \lesssim_{L} 1 \quad \forall \; |y_{0}-y_{1}|>1, \; \forall \; 1 \le q \le m.
\] If $|y_{0}-y_{1}|\le 1$, using supremum bound gives \[
\left | \frac{1}{|y_{0}-y_{1}|^{\frac{q}{m+1}}}\int_{y_{0}}^{y_{1}} \tIm\bigg(P_{k-1-q;k-1-q,0}\Big ( \big (\partial_{x}^{\beta}\overline{b}_{\alpha}(\widetilde{x})\big )_{0 \le \alpha \le k-2, \beta \ge 0} \Big )\bigg)\dd\widetilde{x} \right | \lesssim_{L} |y_{0}-y_{1}|^{1-\frac{q}{m+1}} \le 1 \quad \forall \; 1 \le q \le m.
\] Combining these two cases proves \eqref{20241007eq9} with $m+1$ instead of $m$. This completes the inductive step. Using \eqref{20241007eq9} for $m=k-1$ gives the desired conclusion.
\end{proof}

Finally, we can prove Theorem \ref{20241015thm1}.

\begin{proof}[Proof of Theorem \ref{20241015thm1}]
Let $A=\sum_{j=0}^{k-2}c_{j}(x)D_{x}^{j}$. Since $A$ is self-adjoint, $L_{0}u=0$ is $L^{2}$-wellposed, where $L_{0}=D_{t}-D_{x}^{k}-A=D_{t}-D_{x}^{k}-\sum_{j=0}^{k-2}c_{j}(x)D_{x}^{j}$. (See Appendix \ref{appA}.) By assumption, $Lu=0$ with $L=D_{t}-D_{x}^{k}-A-\sum_{j=0}^{m}b_{j}(x)D_{x}^{j} =: D_{t}-D_{x}^{k}-\sum_{j=0}^{k-2}\widetilde{c}_{j}(x)D_{x}^{j}$ is also $L^{2}$-wellposed. By Proposition \ref{20241008prop1} applied for $L_{0}$ and $L$, we have \begin{IEEEeqnarray*}{rCl}
\left | \int_{x}^{y}\tIm \bigg ( P_{m;m,0}\Big(\big(\partial_{x}^{\beta}\overline{c}_{\alpha}(\widetilde{x})\big)_{0 \le \alpha \le k-2, \beta \ge 0} \Big ) \bigg ) \dd\widetilde{x} \right | &\lesssim&\hspace{-0.1cm}{}_{L}\, |x-y|^{\frac{k-1-m}{k-1}}, \yesnumber \label{20241008eq6}\\
\left | \int_{x}^{y}\tIm \bigg ( P_{m;m,0}\Big(\big(\partial_{x}^{\beta}\overline{\widetilde{c}}_{\alpha}(\widetilde{x})\big)_{0 \le \alpha \le k-2, \beta \ge 0} \Big ) \bigg ) \dd\widetilde{x} \right | &\lesssim&\hspace{-0.1cm}{}_{L}\, |x-y|^{\frac{k-1-m}{k-1}}. \yesnumber \label{20241008eq7}
\end{IEEEeqnarray*}
By Lemma \ref{20241007lem1}(iii), we have $P:=P_{m;m,0}\Big ( \big \{ x_{\alpha}^{(\beta)}\big \}_{0 \le \alpha \le k-2, \beta \ge 0}\Big )-x_{m}^{(0)} \in \cP_{m+1, k-2}$. But note that $c_{\alpha}=\widetilde{c}_{\alpha}$ for $m+1 \le \alpha \le k-2$. Hence, the value of $P$ evaluated at $\big (\partial_{x}^{\beta}\overline{c}_{\alpha}(x)\big )_{0 \le \alpha \le k-2, \beta \ge 0}$ and at $\big ( \partial_{x}^{\beta}\overline{\widetilde{c}}_{\alpha}(x)\big )_{0 \le \alpha \le k-2, \beta \ge 0}$ are equal. Thus, \[
P_{m;m,0}\Big(\big(\partial_{x}^{\beta}\overline{c}_{\alpha}(x)\big)_{0 \le \alpha \le k-2, \beta \ge 0} \Big )-P_{m;m,0}\Big(\big(\partial_{x}^{\beta}\overline{\widetilde{c}}_{\alpha}(x)\big)_{0 \le \alpha \le k-2, \beta \ge 0} \Big )=\overline{c}_{m}(x)-\overline{\widetilde{c}}_{m}(x)=-\overline{b}_{m}(x),
\] so taking the difference between \eqref{20241008eq6} and \eqref{20241008eq7} finally gives \[
\left | \int_{x}^{y} \tIm \big ( b_{m}(\widetilde{x}) \big )\dd \widetilde{x}\right | \lesssim_{L} |x-y|^{\frac{k-1-m}{k-1}}.
\]
\end{proof}

\section{Proof of Theorem \ref{20241015thm2} (Sufficiency in $k=5$)}\label{sec3}

In this section, we prove Theorem \ref{20241015thm2}.

\begin{proof}[Proof of Theorem \ref{20241015thm2}]
If the Cauchy problem \eqref{20241015eq2} is $L^{2}$-wellposed, then by Theorem \ref{20241015thm1} it is necessary that
\begin{eqnarray}\label{20241008eq1}
\left | \int_{x}^{y} \tIm \big(b(\widetilde{x})\big)\dd\widetilde{x} \right | \lesssim_{L} |x-y|^{\frac{1}{4}}.
\end{eqnarray}
From now on, assume that this estimate hold. Define two $\Psi$DOs
\begin{IEEEeqnarray*}{rCl}
\Phi_{0, 1}(\ell; \xi, x)&:=&\frac{1}{5}\left ( \frac{1}{\langle \xi\rangle_{\ell}}+\frac{\ell^{2}}{2\langle \xi \rangle_{\ell}^{3}}\right )\int_{-\infty}^{x}\chi\left ( \frac{y-x}{\langle \xi\rangle_{\ell}^{4}} \right )\tIm \big(b(y)\big)\dd y, \\
\Phi_{1}(\ell;\xi, x)&:=&e^{\Phi_{0, 1}}\bigg ( 1+ \frac{1}{5\langle\xi\rangle_{\ell}} \partial_{\xi}\Phi_{0, 1}\tRe\big(b(x)\big) +\frac{1}{5\langle\xi\rangle_{\ell}^{2}}\partial_{\xi}\Phi_{0,1}\left (c(x)+2ib'(x)\right ) \\
&& \qquad \quad -\frac{1}{10\langle\xi\rangle_{\ell}}\left ( \partial_{\xi}^{2}\Phi + \big(\partial_{\xi}\Phi\big)^{2} \right )i \tRe\big(b'(x)\big)\bigg ) \yesnumber\label{20241015eq1}
\end{IEEEeqnarray*}
where $\langle \xi \rangle_{\ell}:=\sqrt{\xi^{2}+\ell^{2}}$ and $\chi \in C_{c}^{\infty}(-2, 2)$ is a fixed function satisfying $\chi(x) =1 \;  \forall \; |x| \le 1$. We will use one lemma in \cite{Mizuhara06}.
\begin{lem}\cite[Lemma 2.1]{Mizuhara06} \label{20241009lem1}
Let $0 \le p < 1 < q$. Assume that $h(x) \in \overline{C}^{\infty}(\RR)$ satisfies \[
\left | \int_{x}^{y}h(\widetilde{x})\dd\widetilde{x} \right | \lesssim |x-y|^{p}.
\] Define $H(x, \xi;\ell)$ by \[
H(x, \xi;\ell) = \int_{-\infty}^{x} \chi \left ( \frac{y-x}{\langle \xi \rangle_{\ell}^{q}}\right )h(y)\dd y,
\] where $\chi \in C_{c}^{\infty}(-2, 2)$ is a fixed function satisfying $\chi(x)=1 \; \forall \; |x| \le 1$. Then, $H \in S_{(\ell)}^{pq}$, and $\partial_{x}H - h \in S_{(\ell)}^{pq-q}$.
\end{lem} Note that $\ell^{n}\langle \xi\rangle_{\ell}^{-m} \in S_{(\ell)}^{n-m}$ for any integers $n, m$. Using this with \eqref{20241008eq1} and Lemma \ref{20241009lem1} shows that \begin{eqnarray} \label{20241015eq4}
\Phi_{0, 1} \in S_{(\ell)}^{0}, \quad \partial_{x}\Phi_{0, 1} - \frac{1}{5}\left ( \frac{1}{\langle \xi \rangle_{\ell}} + \frac{\ell^{2}}{2\langle \xi \rangle_{\ell}^{3}} \right ) \tIm \big ( b(x)\big ) \in S_{(\ell)}^{-4}, \quad \Phi_{1} \in S_{(\ell)}^{0}
\end{eqnarray} Now define \begin{IEEEeqnarray*}{rCl}
L_{(1)}&=&D_{t}-D_{x}^{5} - \tRe\big(b(x)\big)D_{x}^{3} -\left ( c(x) - 2\tIm \big (b'(x) \big ) \right ) D_{x}^{2} \\
&& - \left ( d(x) + 2i\tIm\big(b''(x)\big) - \frac{2}{5}\Big ( \tIm \big (b(x) \big ) \Big )^{2} - \frac{3}{5}ib(x)\tIm \big (b(x) \big ) + \frac{1}{5}i\tRe\big(b(x)\big ) \tIm \big (b(x) \big ) \right ) D_{x} \yesnumber \label{20241015eq3}
\end{IEEEeqnarray*}
Then by using \eqref{20241015eq4}, one can calculate that $L \circ \Phi_{1} \equiv \Phi_{1} \circ L_{(1)} \pmod{S^{0}}$ for any $\ell \ge 1$. The verification of this calculation is done in Appendix \ref{subsubappB.1.1}. Observe from the definition of $\Phi_{1}$ that if we define $\Psi_{1} := e^{-\Phi_{0,1}}$, then one has $\Psi_{1} \in S_{(\ell)}^{0}$, $\Phi_{1} \circ \Psi_{1} - \mathrm{Id} \in S_{(\ell)}^{-1}$, and $\Psi_{1} \circ \Phi_{1} - \mathrm{Id} \in S_{(\ell)}^{-1}$. Thus if $\ell \ge 1$ is sufficiently large, both $\Phi_{1}\circ\Psi_{1}$ and $\Psi_{1}\circ\Phi_{1}$ have their inverses in $S^{0}$ and thus $\Phi_{1}$ has its inverse in $S^{0}$. Fix such $\ell$, so that both $\Phi_{1}$ and $\Phi_{1}^{-1}$ are $\Psi$DO in $S^{0}$. Now define $C \in S^{0}$ by $C=L \circ \Phi_{1} - \Phi_{1} \circ L_{(1)}$. If the Cauchy problem $Lu=0$ is $L^{2}$-wellposed, then for any $T>0$ and $s \in\RR$ the estimate \[
\|u(t)\|_{H^{s}} \lesssim_{T, s} \|u(0)\|_{H^{s}}+\int_{0}^{t}\|Lu(\tau)\|_{H^{s}}\dd\tau \quad \forall \; |t| \le T
\] and the time-reversed estimate for $L^{\ast}$ hold. Plugging $u=\Phi_{1}v$ gives \[
\|v\|_{H^{s}} \lesssim_{T, s} \|v(0)\|_{H^{s}} + \int_{0}^{t}\|L_{(1)}v(\tau)\|_{H^{s}}\dd\tau + \int_{0}^{t}\|v(\tau)\|_{H^{s}}\dd\tau,
\] using that $\Phi, \Phi^{-1}, C$ are $H^{s}$-bounded operator. Now Gr\"onwall's inequality gives \[
|v\|_{H^{s}} \lesssim_{T, s}\|v(0)\|_{H^{s}} + \int_{0}^{t}\|L_{(1)}v(\tau)\|_{H^{s}}\dd\tau.
\] Similar argument for $L^{\ast}$ and $\big(\Phi_{1}^{\ast}\big)^{-1}$ gives time-reversed estimate for $L_{(1)}^{\ast}$. These prove the $L^{2}$-wellposedness of the Cauchy problem $L_{(1)}v=0$. Reversing $L, L_{(1)}$ and replacing $\Phi_{1}$ by $\Phi_{1}^{-1}$ allows us to conclude that the $L^{2}$-wellposedness for $Lu=0$ and $L_{(1)}v=0$ are equivalent. Hence the problem boils down into the $L^{2}$-wellposedness of the Cauchy problem $L_{(1)}u=0$, where 
\begin{IEEEeqnarray*}{rCl}
L_{(1)}&=&D_{t}-D_{x}^{5} - \tRe\big(b(x)\big)D_{x}^{3} -\left ( c(x) - 2\tIm \big (b'(x) \big ) \right ) D_{x}^{2} \\
&& - \left ( d(x) + 2i\tIm\big(b''(x)\big) - \frac{2}{5}\Big ( \tIm \big (b(x) \big ) \Big )^{2} - \frac{3}{5}ib(x)\tIm \big (b(x) \big ) + \frac{1}{5}i\tRe\big(b(x)\big ) \tIm \big (b(x) \big ) \right ) D_{x} \\
&=& D_{t} - D_{x}^{5} - \left ( \tRe\big(b(x)\big)D_{x}^{3}-\frac{3}{2}i\tRe\big(b'(x)\big)D_{x}^{2}\right ) -\left ( c(x) - 2\tIm \big (b'(x) \big ) + \frac{3}{2}i\tRe\big(b'(x)\big) \right ) D_{x}^{2}\\
&& - \left ( d(x) + 2i\tIm\big(b''(x)\big) - \frac{2}{5}\Big ( \tIm \big (b(x) \big ) \Big )^{2} - \frac{3}{5}ib(x)\tIm \big (b(x) \big ) + \frac{1}{5}i\tRe\big(b(x)\big ) \tIm \big (b(x) \big ) \right ) D_{x} \\
&=:& D_{t}-D_{x}^{5}-A_{(1)}-c_{(1)}(x)D_{x}^{2}-d_{(1)}(x)D_{x}. \yesnumber\label{20241015eq5}
\end{IEEEeqnarray*}
Here, $A_{(1)}$ is a self-adjoint differential operator of order $3$. For $L_{(1)}u=0$ to be $L^{2}$-wellposed, it is necessary that 
\begin{eqnarray}\label{20241010eq1}
\left | \int_{x}^{y}\tIm \big (c_{(1)}(\widetilde{x})\big)\dd\widetilde{x} \right | \lesssim_{L} |x-y|^{\frac{1}{2}}.
\end{eqnarray}
by Theorem \ref{20241015thm1}. Assume that this condition holds. Define two $\Psi$DOs
\begin{IEEEeqnarray*}{rCl}
\Phi_{0, 2}(\ell;\xi, x)&:=&\frac{1}{5\langle\xi\rangle_{\ell}^{2}} \int_{-\infty}^{x}\chi\left ( \frac{y-x}{\langle \xi \rangle_{\ell}^{4}} \right )\tIm \big (c_{(1)}(y)\big)\dd y, \\
\Phi_{2}(\ell;\xi, x)&:=&e^{\Phi_{0, 2}}\bigg ( 1 + \frac{1}{5\langle \xi\rangle_{\ell}}\tRe\big(b(x)\big)\partial_{\xi}\Phi_{0,2} + \frac{1}{10\langle\xi\rangle_{\ell}^{2}}\left (2c_{(1)}(x)+i\tRe\big(b'(x)\big)\right )\partial_{\xi}\Phi_{0,2} \\
&& \qquad  \quad - i \partial_{x}\Phi_{0,2}\partial_{\xi}\Phi_{0,2}  - \frac{1}{10\langle\xi\rangle_{\ell}}i\tRe\big(b'(x)\big)\Big ( \partial_{\xi}^{2}\Phi_{0,2}+\big(\partial_{\xi}\Phi_{0,2}\big)^{2}\Big)\bigg ). \yesnumber\label{20241015eq6}
\end{IEEEeqnarray*}
Then by \eqref{20241010eq1} and Lemma \ref{20241009lem1}, we have \begin{eqnarray}\label{20241015eq8}
\Phi_{0, 2} \in S_{(\ell)}^{0}, \quad \partial_{x}\Phi_{0,2}-\frac{1}{5\langle \xi \rangle_{\ell}^{2}}\tIm \big (c_{(1)}(x) \big ) \in S_{(\ell)}^{-4}, \quad \Phi_{2} \in S_{\ell}^{(0)}.
\end{eqnarray} Now define \begin{eqnarray}\label{20241015eq7}
L_{(2)} = D_{t} - D_{x}^{5} - A_{(1)} - \tRe\big(c_{(1)}(x)\big)D_{x}^{2}-\left ( d_{(1)}(x)-2\tIm\big(c_{(1)}'(x)\big ) \right )D_{x}.
\end{eqnarray} Then by using \eqref{20241015eq8}, one can check that $L_{(1)} \circ \Phi_{2} \equiv \Phi_{2} \circ L_{(2)} \pmod{S^{0}}$, for any $\ell \ge 1$. This calculation is postponed in Appendix \ref{subsubappB.1.2}. Arguing as before shows that $L^{2}$-wellposedness for Cauchy problem $L_{(1)}u=0$ and that for $L_{(2)}u=0$ are equivalent. Now we study $L_{(2)}u=0$, where
\begin{IEEEeqnarray*}{rCl}
L_{(2)} &=& D_{t} - D_{x}^{5} - A_{(1)} - \tRe\big(c_{(1)}(x)\big)D_{x}^{2}-\left ( d_{(1)}(x)-2\tIm \big (c_{(1)}'(x)\big)\right )D_{x} \\
&=& D_{t} - D_{x}^{5} - \left ( A_{(1)} + \tRe\big(c_{(1)}(x)\big ) D_{x}^{2} - i\tRe\big(c_{(1)}'(x)\big )D_{x} \right ) - \left ( d_{(1)}(x)-2\tIm \big (c_{(1)}'(x)\big) + i\tRe\big(c_{(1)}'(x) \big )\right )D_{x} \\
&=:& D_{t}-D_{x}^{5}-A_{(2)}-d_{(2)}(x)D_{x}. \yesnumber\label{20241015eq9}
\end{IEEEeqnarray*}
Here, $A_{(2)}$ is a self-adjoint differential operator of order $3$. For $L_{(2)}u=0$ to be $L^{2}$-wellposed, it is necessary that \begin{eqnarray}\label{20241010eq2}
\left | \int_{x}^{y}\tIm\big(d_{(2)}(\widetilde{x})\big )\dd\widetilde{x} \right | \lesssim_{L} |x-y|^{\frac{3}{4}}
\end{eqnarray} by Theorem \ref{20241015thm1}. Assume that this condition holds. Define two $\Psi$DOs
\begin{IEEEeqnarray*}{rCl}
\Phi_{0,3}(\ell;\xi, x) &:=& \frac{1}{5\langle \xi \rangle_{\ell}^{3}}\int_{-\infty}^{x}\chi\left  ( \frac{y-x}{\langle \xi \rangle_{\ell}^{4}} \right )\tIm \big (d_{(2)}(y)\big )\dd y, \\
\Phi_{3}(\ell;\xi, x)&:=& e^{\Phi_{0,3}}\bigg ( 1 + \frac{1}{5\langle \xi \rangle_{\ell}} \tRe\big(b(x)\big)\partial_{\xi}\Phi_{0,3} + \frac{1}{5\langle \xi \rangle_{\ell}^{2}} \left ( \tRe\big(c_{(1)}(x)\big) + \frac{1}{2}i\tRe\big(b'(x)\big) \right )\partial_{\xi}\Phi_{0,3} \\
&& \qquad \quad - \frac{1}{10\langle \xi \rangle_{\ell}}i\tRe\big(b'(x)\big)\Big ( \partial_{\xi}^{2}\Phi_{0,3} + \big ( \partial_{\xi}\Phi_{0,3}\big )^{2} \Big )\bigg ). \yesnumber \label{20241015eq10}
\end{IEEEeqnarray*} Then by \eqref{20241010eq2} and Lemma \ref{20241009lem1} we have \begin{eqnarray}\label{20241015eq12}
\Phi_{0,3} \in S_{(\ell)}^{0}, \quad \partial_{x}\Phi_{0,3} - \frac{1}{5\langle \xi \rangle_{\ell}^{3}}\tIm \big (d_{(2)}(x) \big ) \in S_{(\ell)}^{-4}, \quad \Phi_{3} \in S_{(\ell)}^{0}.
\end{eqnarray} Now define \begin{eqnarray}\label{20241015eq11}
L_{(3)} = D_{t} - D_{x}^{5}-A_{(2)} - \tRe\big(d_{(2)}(x)\big )D_{x}.
\end{eqnarray} Then $L_{(2)} \circ \Phi_{3} \equiv \Phi_{3} \circ L_{(3)} \pmod{S^{0}}$ for any $\ell \ge 1$, which is verified in Appendix \ref{subsubappB.1.3}. Arguing again as before shows that $L^{2}$-wellposedness for Cauchy problem $L_{(2)}u=0$ and that for $L_{(3)}u=0$ are equivalent. But note that $L_{(3)}$ is of the form $D_{t}-D_{x}^{5}-A$, where $A$ is a self-adjoint differential operator of order $3$. Hence, $L_{(3)}u=0$ is $L^{2}$-wellposed. (See Appendix \ref{appA}.)

In conclusion, $L^{2}$-wellposedness of the original problem $Lu=0$ is equivalent to the set of three conditions \eqref{20241008eq1}, \eqref{20241010eq1}, and \eqref{20241010eq2}. Plugging back $c_{(1)}$ and $d_{(1)}$ gives 
\begin{IEEEeqnarray*}{rCl}
\left | \int_{x}^{y} \tIm \left ( c(\widetilde{x})-2\tIm\big(b'(\widetilde{x})\big)+\frac{3}{2}i\tRe\big(b'(\widetilde{x})\big)\right )\dd\widetilde{x} \right |&\lesssim&\hspace{-0.1cm}{}_{L}\, |x-y|^{\frac{1}{2}}, \\
\bigg | \int_{x}^{y} \tIm \bigg ( d(\widetilde{x}) + 2i\tIm\big(b''(\widetilde{x})\big) - \frac{2}{5}\Big ( \tIm \big (b(\widetilde{x}) \big ) \Big )^{2} - \frac{3}{5}ib(\widetilde{x})\tIm \big (b(\widetilde{x}) \big ) \qquad \quad && \\
+ \frac{1}{5}i\tRe\big(b(\widetilde{x})\big ) \tIm \big (b(\widetilde{x}) \big )-2\tIm\big(c_{(1)}'(\widetilde{x})\big) \bigg )\dd\widetilde{x} \bigg |&\lesssim&\hspace{-0.1cm}{}_{L}\, |x-y|^{\frac{3}{4}}.
\end{IEEEeqnarray*}
Observe that for any $p \in \overline{C}^{\infty}(\RR)$, $\left | \int_{x}^{y} p'(\widetilde{x})\dd\widetilde{x} \right | \lesssim \min\{|x-y|, 1\}$, so above two conditions simplify into 
\begin{IEEEeqnarray*}{rCl}
\left | \int_{x}^{y} \tIm \big ( c(\widetilde{x})\big )\dd\widetilde{x} \right |&\lesssim&\hspace{-0.1cm}{}_{L}\, |x-y|^{\frac{1}{2}}, \\
\left | \int_{x}^{y} \tIm \left ( d(\widetilde{x})  - \frac{2}{5}i\tRe\big(b(\widetilde{x})\big ) \tIm \big (b(\widetilde{x}) \big ) \right )\dd\widetilde{x} \right |&\lesssim&\hspace{-0.1cm}{}_{L}\, |x-y|^{\frac{3}{4}}.
\end{IEEEeqnarray*}
These two conditions together with \eqref{20241008eq1} are exactly the three conditions appearing in the statement of Theorem \ref{20241015thm2}.
\end{proof}

\section{Proof of Theorem \ref{20241015thm3} (Sufficiency in $k=6$)}\label{sec4}

In this section, we prove Theorem \ref{20241015thm3}.

\begin{proof}[Proof of Theorem \ref{20241015thm3}]
If the Cauchy problem \eqref{20241015eq13} is $L^{2}$-wellposed, then by Theorem \ref{20241015thm1} it is necessary that 
\begin{IEEEeqnarray*}{rCl}
\left | \int_{x}^{y} \tIm \big (b(\widetilde{x})\big )\dd\widetilde{x} \right | \lesssim_{L} |x-y|^{\frac{1}{5}}. \yesnumber \label{20241010eq3}
\end{IEEEeqnarray*}
Assume that this condition holds. Define two $\Psi$DOs
\begin{IEEEeqnarray*}{rCl}
\Phi_{0,1}(\ell;\xi, x) &:=& \frac{1}{6}\left ( \frac{1}{\langle \xi\rangle_{\ell}}+\frac{\ell^{2}}{2\langle \xi \rangle_{\ell}^{3}}\right )\int_{-\infty}^{x} \chi \left ( \frac{y-x}{\langle \xi \rangle_{\ell}^{5}} \right )\tIm \big (b(y)\big ) \dd y, \\
\Phi_{1}(\ell;\xi, x) &:=& e^{\Phi_{0,1}}\bigg [ 1+ \frac{1}{6}\left ( \frac{1}{\langle \xi\rangle_{\ell}}+\frac{\ell^{2}}{2\langle \xi \rangle_{\ell}^{3}}\right ) \partial_{\xi}\Phi_{0, 1}\tRe\big(b(x)\big) +\frac{1}{6\langle \xi \rangle_{\ell}^{2}}\partial_{\xi}\Phi_{0,1}\left ( c(x)+\frac{5}{2}ib'(x)\right ) \\
&& \qquad \quad - \frac{1}{12\langle \xi \rangle_{\ell}}\Big (\partial_{\xi}^{2}\Phi_{0,1} + \big(\partial_{\xi}\Phi_{0,1}\big)^{2}\Big )i\tRe\big(b'(x)\big) \\
&& \qquad \quad + \frac{1}{\langle \xi \rangle_{\ell}^{3}} \partial_{\xi}\Phi_{0,1}\left ( -\frac{35}{72} b''(x) + \frac{5}{12}ic'(x) + \frac{1}{6} d(x) - \frac{1}{24}b(x)^{2} - \frac{1}{36}\Big ( \tRe \big (b(x) \big ) \Big )^{2}\right )  \\
&& \qquad \quad + \frac{1}{\langle \xi\rangle_{\ell}^{2}}\Big ( \partial_{\xi}^{2}\Phi_{0,1} + \big ( \partial_{\xi}\Phi_{0,1} \big )^{2}\Big )\left ( -\frac{1}{12}i c'(x) +\frac{5}{24}b''(x)
+ \frac{1}{72}\Big ( \tRe\big(b(x)\big)\Big )^{2}\right ) \\
&& \qquad \quad - \frac{1}{\langle \xi \rangle_{\ell}}\Big ( \partial_{\xi}^{3}\Phi_{0,1} + 3\partial_{\xi}^{2}\Phi_{0,1}\partial_{\xi}\Phi_{0,1} + \big ( \partial_{\xi}\Phi_{0,1} \big )^{3} \Big )\frac{1}{36}\tRe\big(b''(x)\big ) \bigg ], \yesnumber \label{20241015eq14}
\end{IEEEeqnarray*}
where $\chi \in C_{c}^{\infty}(-2, 2)$ is a fixed function with $\chi(x)=1 \; \forall \; |x| \le 1$. Then by Lemma \ref{20241009lem1}, one has that \begin{eqnarray}\label{20241015eq16}
\Phi_{0,1} \in S_{(\ell)}^{0}, \quad \partial_{x}\Phi_{0,1} - \frac{1}{6}\left ( \frac{1}{\langle \xi \rangle_{\ell}} + \frac{\ell^{2}}{2\langle \xi \rangle_{\ell}^{3}} \right ) \tIm \big (b(x) \big ) \in S_{(\ell)}^{-5}, \quad \Phi_{1} \in S_{(\ell)}^{0}.
\end{eqnarray} Then by calculations in Appendix \ref{subsubappB.2.1}, we have $L \circ \Phi_{1} \equiv \Phi_{1} \circ L_{(1)} \pmod{S^{0}}$ for any $\ell \ge 1$, where \begin{IEEEeqnarray*}{rCl}
L_{(1)} &=& D_{t}-D_{x}^{6}-\tRe\big(b(x)\big)D_{x}^{4}-\left ( c(x) - \frac{5}{2}\tIm\big(b'(x)\big)\right )D_{x}^{3} - \bigg (d(x)+\frac{1}{4}\Big(\tIm\big(b(x)\big)\Big)^{2}\\
&& - \frac{1}{2}i\tIm\big(b(x)\big)\tRe\big(b(x)\big)+\frac{10}{3}i \tIm \big (b''(x) \big )\bigg )D_{x}^{2} -  \bigg (e(x)-\frac{1}{3}i\tIm\big(b(x)\big)c(x)+ \frac{1}{4}i \tIm\big(b'(x)\big)\tIm\big(b(x)\big) \\
&& + \frac{1}{4}\tRe\big(b'(x)\big)\tIm\big(b(x)\big)- \frac{7}{12}\tRe\big(b(x)\big)\tIm\big(b'(x)\big)+\frac{5}{2} \tIm \big (b'''(x) \big )\bigg )D_{x}. \yesnumber\label{20241015eq15}
\end{IEEEeqnarray*}
Arguing as in the proof of Theorem \ref{20241015thm2} shows that $L^{2}$-wellposedness for Cauchy problem $Lu=0$ and that for $L_{(1)}u=0$ are equivalent. Thus it is now enough to study $L^{2}$-wellposedness of $L_{(1)}u=0$. For simplicity, write 
\begin{IEEEeqnarray*}{rCl}
L_{(1)} &=& D_{t}-D_{x}^{6}-\tRe\big(b(x)\big)D_{x}^{4}-\left ( c(x) - \frac{5}{2}\tIm\big(b'(x)\big)\right )D_{x}^{3}  - \bigg (d(x)+\frac{1}{4}\Big(\tIm\big(b(x)\big)\Big)^{2} \\
&& - \frac{1}{2}i\tIm\big(b(x)\big)\tRe\big(b(x)\big)+\frac{10}{3}i \tIm \big (b''(x) \big )\bigg )D_{x}^{2} -  \bigg (e(x)-\frac{1}{3}i\tIm\big(b(x)\big)c(x)+ \frac{1}{4}i \tIm\big(b'(x)\big)\tIm\big(b(x)\big) \\
&& + \frac{1}{4}\tRe\big(b'(x)\big)\tIm\big(b(x)\big)- \frac{7}{12}\tRe\big(b(x)\big)\tIm\big(b'(x)\big)+\frac{5}{2} \tIm \big (b'''(x) \big )\bigg )D_{x} \\
&=& D_{t}-D_{x}^{6}-\left (\tRe\big(b(x)\big)D_{x}^{4}-2i\tRe\big(b'(x)\big)D_{x}^{3} - i \tRe\big(b'''(x)\big)D_{x} \right )\\
&& -\left ( c(x) - \frac{5}{2}\tIm\big(b'(x)\big)+2i\tRe\big(b'(x)\big)\right )D_{x}^{3} \\
&& - \left (d(x)+\frac{1}{4}\Big(\tIm\big(b(x)\big)\Big)^{2}- \frac{1}{2}i\tIm\big(b(x)\big)\tRe\big(b(x)\big)+\frac{10}{3}i \tIm \big (b''(x) \big )\right )D_{x}^{2} \\
&& -  \bigg (e(x)-\frac{1}{3}i\tIm\big(b(x)\big)c(x)+ \frac{1}{4}i \tIm\big(b'(x)\big)\tIm\big(b(x)\big) + \frac{1}{4}\tRe\big(b'(x)\big)\tIm\big(b(x)\big)\\
&& \quad\;\; - \frac{7}{12}\tRe\big(b(x)\big)\tIm\big(b'(x)\big)+\frac{5}{2} \tIm \big (b'''(x) \big )+i\tRe\big(b'''(x)\big)\bigg )D_{x} \\
&=:& D_{t}-D_{x}^{6}-A_{(1)} - c_{(1)}(x)D_{x}^{3}-d_{(1)}(x)D_{x}^{2}-e_{(1)}(x)D_{x}, \yesnumber \label{20241015eq17}
\end{IEEEeqnarray*}
where $A_{(1)}$ is a self-adjoint differential operator of order $4$. Then by Theorem \ref{20241015thm1}, for $L_{(1)}u=0$ to be $L^{2}$-wellposed, we must have \begin{eqnarray}\label{20241011eq1}
\left | \int_{x}^{y} \tIm \big (c_{(1)}(\widetilde{x})\big)\dd\widetilde{x} \right | \lesssim_{L} |x-y|^{\frac{2}{5}}.
\end{eqnarray}
Assume that this condition holds, and define two $\Psi$DOs
\begin{IEEEeqnarray*}{rCl}
\Phi_{0,2}(\ell;\xi, x)&:=& \frac{1}{6}\left ( \frac{1}{\langle \xi \rangle_{\ell}^{2}} + \frac{\ell^{2}}{2\langle \xi \rangle_{\ell}^{4}}\right ) \int_{-\infty}^{x} \chi \left ( \frac{y-x}{\langle \xi \rangle_{\ell}^{5}} \right )\tIm \big (c_{(1)}(y)\big)\dd y, \\
\Phi_{2}(\ell;\xi, x) &:=& e^{\Phi_{0,2}}\bigg [ 1+ \frac{1}{6}\left ( \frac{1}{\langle \xi\rangle_{\ell}}+\frac{\ell^{2}}{2\langle \xi \rangle_{\ell}^{3}}\right ) \partial_{\xi}\Phi_{0, 2}\tRe\big(b(x)\big) +\frac{1}{6\langle \xi \rangle_{\ell}^{2}}\partial_{\xi}\Phi_{0,2}\left ( \tRe\big(c_{(1)}(x)\big)+\frac{1}{2}i\tRe\big(b'(x)\big)\right ) \\
&& \qquad \quad - \frac{1}{12\langle \xi \rangle_{\ell}}\Big (\partial_{\xi}^{2}\Phi_{0,2} + \big(\partial_{\xi}\Phi_{0,2}\big)^{2}\Big )i\tRe\big(b'(x)\big) \\
&& \qquad \quad + \frac{1}{\langle \xi \rangle_{\ell}^{3}} \partial_{\xi}\Phi_{0,2}\left ( \frac{1}{6} d_{(1)}(x) +\frac{5}{12}ic_{(1)}'(x) + \frac{25}{72}\tRe\big(b''(x)\big) - \frac{5}{72}\Big(\tRe\big(b(x)\big)\Big)^{2} \right )  \\
&& \qquad \quad + \frac{1}{\langle \xi\rangle_{\ell}^{2}}\Big ( \partial_{\xi}^{2}\Phi_{0,2} + \big ( \partial_{\xi}\Phi_{0,2} \big )^{2}\Big )\left ( -\frac{1}{12}i \tRe\big(c_{(1)}'(x)\big) + \frac{1}{24}\tRe\big(b''(x)\big) + \frac{1}{72}\Big ( \tRe\big(b(x)\big)\Big )^{2}\right ) \\
&& \qquad \quad - \frac{1}{\langle \xi \rangle_{\ell}}\Big ( \partial_{\xi}^{3}\Phi_{0,2} + 3\partial_{\xi}^{2}\Phi_{0,2}\partial_{\xi}\Phi_{0,2} + \big ( \partial_{\xi}\Phi_{0,2} \big )^{3} \Big )\frac{1}{36}\tRe\big(b''(x)\big ) \bigg ]. \yesnumber \label{20241015eq18}
\end{IEEEeqnarray*}
By Lemma \ref{20241009lem1}, these $\Psi$DO satisfy \begin{eqnarray} \label{20241015eq20}
\Phi_{0,2} \in S_{(\ell)}^{0}, \quad \partial_{x}\Phi_{0,2} - \frac{1}{6}\left ( \frac{1}{\langle \xi \rangle_{\ell}^{2}} + \frac{\ell^{2}}{2\langle \xi \rangle_{\ell}^{4}} \right ) \tIm \big (c_{(1)}(x) \big ) \in S_{(\ell)}^{-5}, \quad \Phi_{2} \in S_{(\ell)}^{0}.
\end{eqnarray}
Then by Appendix \ref{subsubappB.2.2}, we have $L_{(1)} \circ \Phi_{2} \equiv \Phi_{2} \circ L_{(2)} \pmod{S^{0}}$ for any $\ell \ge 1$, where \begin{IEEEeqnarray*}{rCl}
L_{(2)} &=& D_{t}-D_{x}^{6}-A_{(1)}- \tRe\big(c_{(1)}(x)\big)D_{x}^{3} - \left ( d_{(1)}(x) - \frac{5}{2}\tIm\big(c_{(1)}'(x)\big) \right )D_{x}^{2} \\
&& - \left ( e_{(1)}(x) - \frac{1}{3}i\tRe\big(b(x)\big)\tIm\big(c_{(1)}(x)\big) + \frac{10}{3}i\tIm\big(c_{(1)}''(x)\big) \right )D_{x}. \yesnumber \label{20241015eq19}
\end{IEEEeqnarray*}
So, the $L^{2}$-wellposedness of the Cauchy problem $L_{(1)}u=0$ is equivalent to that of $L_{(2)}u=0$. To use Theorem \ref{20241015thm1}, we rewrite $L_{(2)}$ as 
\begin{IEEEeqnarray*}{rCl}
L_{(2)} &=& D_{t}-D_{x}^{6}-A_{(1)}- \tRe\big(c_{(1)}(x)\big)D_{x}^{3} - \left ( d_{(1)}(x) - \frac{5}{2}\tIm\big(c_{(1)}'(x)\big) \right )D_{x}^{2} \\
&& - \left ( e_{(1)}(x) - \frac{1}{3}i\tRe\big(b(x)\big)\tIm\big(c_{(1)}(x)\big) + \frac{10}{3}i\tIm\big(c_{(1)}''(x)\big) \right )D_{x} \\
&=& D_{t} - D_{x}^{6} - \left (A_{(1)} + \tRe\big(c_{(1)}(x)\big)D_{x}^{3} - \frac{3}{2}i\tRe\big(c_{(1)}'(x)\big)D_{x}^{2} \right )\\
&& - \left ( d_{(1)}(x) - \frac{5}{2}\tIm\big(c_{(1)}'(x)\big) + \frac{3}{2}i\tRe\big(c_{(1)}'(x)\big)\right )D_{x}^{2} \\
&& - \left ( e_{(1)}(x) - \frac{1}{3}i\tRe\big(b(x)\big)\tIm\big(c_{(1)}(x)\big) + \frac{10}{3}i\tIm\big(c_{(1)}''(x)\big) \right )D_{x} \\
&=:& D_{t}-D_{x}^{6}-A_{(2)} - d_{(2)}(x)D_{x}^{2}-e_{(2)}(x)D_{x}. \yesnumber\label{20241015eq21}
\end{IEEEeqnarray*}
Then $A_{(2)}$ is a self-adjoint differential operator of order $4$. So we can apply Theorem \ref{20241015thm1}, which gives that 
\begin{eqnarray}\label{20241014eq1}
\left | \int_{x}^{y}\tIm\big (d_{(2)}(\widetilde{x})\big )\dd\widetilde{x} \right | \lesssim_{L} |x-y|^{\frac{3}{5}}
\end{eqnarray} is necessary for the $L^{2}$-wellposedness of $L_{(2)}u=0$. Assume this condition from now on. Define two $\Psi$DOs 
\begin{IEEEeqnarray*}{rCl}
\Phi_{0,3}(\ell;\xi, x)&:=& \frac{1}{6\langle \xi \rangle_{\ell}^{3}} \int_{-\infty}^{x} \chi \left ( \frac{y-x}{\langle \xi \rangle_{\ell}^{5}} \right )\tIm \big (d_{(2)}(y)\big)\dd y, \\
\Phi_{3}(\ell;\xi, x) &:=& e^{\Phi_{0,3}}\bigg [ 1+ \frac{1}{6}\left ( \frac{1}{\langle \xi\rangle_{\ell}}+\frac{\ell^{2}}{2\langle \xi \rangle_{\ell}^{3}}\right ) \partial_{\xi}\Phi_{0, 3}\tRe\big(b(x)\big) +\frac{1}{6\langle \xi \rangle_{\ell}^{2}}\partial_{\xi}\Phi_{0,3}\left ( \tRe\big(c_{(1)}(x)\big)+\frac{1}{2}i\tRe\big(b'(x)\big)\right ) \\
&& \qquad \quad - \frac{1}{12\langle \xi \rangle_{\ell}}\Big (\partial_{\xi}^{2}\Phi_{0,3} + \big(\partial_{\xi}\Phi_{0,3}\big)^{2}\Big )i\tRe\big(b'(x)\big) \\
&& \qquad \quad + \frac{1}{\langle \xi \rangle_{\ell}^{3}} \partial_{\xi}\Phi_{0,3}\left ( \frac{1}{6} \tRe\big (d_{(2)}(x)\big ) +\frac{1}{6}i\tRe\big (c_{(1)}'(x)\big) + \frac{25}{72}\tRe\big(b''(x)\big) - \frac{5}{72}\Big(\tRe\big(b(x)\big)\Big)^{2} \right )  \\
&& \qquad \quad + \frac{1}{\langle \xi\rangle_{\ell}^{2}}\Big ( \partial_{\xi}^{2}\Phi_{0,3} + \big ( \partial_{\xi}\Phi_{0,3} \big )^{2}\Big )\left ( -\frac{1}{12}i \tRe\big(c_{(1)}'(x)\big) + \frac{1}{24}\tRe\big(b''(x)\big) + \frac{1}{72}\Big ( \tRe\big(b(x)\big)\Big )^{2}\right ) \\
&& \qquad \quad - \frac{1}{\langle \xi \rangle_{\ell}}\Big ( \partial_{\xi}^{3}\Phi_{0,3} + 3\partial_{\xi}^{2}\Phi_{0,3}\partial_{\xi}\Phi_{0,3} + \big ( \partial_{\xi}\Phi_{0,3} \big )^{3} \Big )\frac{1}{36}\tRe\big(b''(x)\big ) \bigg ]. \yesnumber\label{20241015eq22}
\end{IEEEeqnarray*}
By Lemma \ref{20241009lem1}, we have \begin{eqnarray}\label{20241015eq24}
\Phi_{0,3} \in S_{(\ell)}^{0}, \quad \partial_{x}\Phi_{0,3} - \frac{1}{6\langle \xi \rangle_{\ell}^{3}} \tIm \big (d_{(2)}(x) \big ) \in S_{(\ell)}^{-5}, \quad \Phi_{3} \in S_{(\ell)}^{0}.
\end{eqnarray} Now define
\begin{IEEEeqnarray*}{rCl}
L_{(3)}&=& D_{t}-D_{x}^{6}-A_{(2)}-\tRe\big(d_{(2)}(x)\big)D_{x}^{2} - \left ( e_{(2)}(x) -\frac{5}{2}\tIm\big (d_{(2)}'(x)\big) \right )D_{x}. \yesnumber\label{20241015eq23}
\end{IEEEeqnarray*} Then by Appendix \ref{subsubappB.2.3}, we have $L_{(2)} \circ \Phi_{3} \equiv \Phi_{3} \circ L_{(3)} \pmod{S^{0}}$ for any $\ell \ge 1$. Hence, the $L^{2}$-wellposedness for the Cauchy problem $L_{(2)}u=0$ is equivalent to that for $L_{(3)}u=0$. Rewrite $L_{(3)}$ as
\begin{IEEEeqnarray*}{rCl}
L_{(3)} &=& D_{t}-D_{x}^{6}-A_{(2)}-\tRe\big(d_{(2)}(x)\big)D_{x}^{2} - \left ( e_{(2)}(x) -\frac{5}{2}\tIm\big (d_{(2)}'(x)\big) \right )D_{x} \\
&=& D_{t} - D_{x}^{6} - \left ( A_{(2)} + \tRe\big(d_{(2)}(x)\big)D_{x}^{2}-i\tRe\big(d_{(2)}'(x)\big)D_{x}\right ) -\left ( e_{(2)}(x) -\frac{5}{2}\tIm\big (d_{(2)}'(x)\big)+ i \tRe\big(d_{(2)}'(x)\big) \right )D_{x} \\
&=:& D_{t}-D_{x}^{6}-A_{(3)}-e_{(3)}(x)D_{x}, \yesnumber \label{20241015eq25}
\end{IEEEeqnarray*}
where $A_{(3)}$ is a $4$th order self-adjoint differential operator. Then by Theorem \ref{20241015thm1}, for $L_{(3)}u=0$ to $L^{2}$-wellposed, the following estimate on $e_{(3)}$ must be satisfied: 
\begin{eqnarray} \label{20241014eq2}
\left | \int_{x}^{y}\tIm\big(e_{(3)}(\widetilde{x})\big)\dd\widetilde{x} \right | \lesssim_{L} |x-y|^{\frac{4}{5}}.
\end{eqnarray}
Assume that this estimate holds. Define two $\Psi$DOs
\begin{IEEEeqnarray*}{rCl}
\Phi_{0,4}(\ell;\xi, x)&:=& \frac{1}{6\langle \xi \rangle_{\ell}^{4}} \int_{-\infty}^{x}\chi\left ( \frac{y-x}{\langle \xi \rangle_{\ell}^{5}}\right )\tIm \big (e_{(3)}(y)\big)\dd y, \\
\Phi_{4}(\ell;\xi, x)&:=& e^{\Phi_{0,4}}\bigg [ 1+ \frac{1}{6}\left ( \frac{1}{\langle \xi\rangle_{\ell}}+\frac{\ell^{2}}{2\langle \xi \rangle_{\ell}^{3}}\right ) \partial_{\xi}\Phi_{0,4}\tRe\big(b(x)\big) +\frac{1}{6\langle \xi \rangle_{\ell}^{2}}\partial_{\xi}\Phi_{0,4}\left ( \tRe\big(c_{(1)}(x)\big)+\frac{1}{2}i\tRe\big(b'(x)\big)\right ) \\
&& \qquad \quad - \frac{1}{12\langle \xi \rangle_{\ell}}\Big (\partial_{\xi}^{2}\Phi_{0,4} + \big(\partial_{\xi}\Phi_{0,4}\big)^{2}\Big )i\tRe\big(b'(x)\big) \\
&& \qquad \quad + \frac{1}{\langle \xi \rangle_{\ell}^{3}} \partial_{\xi}\Phi_{0,4}\left ( \frac{1}{6} \tRe\big (d_{(2)}(x)\big ) +\frac{1}{6}i\tRe\big (c_{(1)}'(x)\big) + \frac{25}{72}\tRe\big(b''(x)\big) - \frac{5}{72}\Big(\tRe\big(b(x)\big)\Big)^{2} \right )  \\
&& \qquad \quad + \frac{1}{\langle \xi\rangle_{\ell}^{2}}\Big ( \partial_{\xi}^{2}\Phi_{0,4} + \big ( \partial_{\xi}\Phi_{0,4} \big )^{2}\Big )\left ( -\frac{1}{12}i \tRe\big(c_{(1)}'(x)\big) + \frac{1}{24}\tRe\big(b''(x)\big) + \frac{1}{72}\Big ( \tRe\big(b(x)\big)\Big )^{2}\right ) \\
&& \qquad \quad - \frac{1}{\langle \xi \rangle_{\ell}}\Big ( \partial_{\xi}^{3}\Phi_{0,4} + 3\partial_{\xi}^{2}\Phi_{0,4}\partial_{\xi}\Phi_{0,4} + \big ( \partial_{\xi}\Phi_{0,4} \big )^{3} \Big )\frac{1}{36}\tRe\big(b''(x)\big ) \bigg ]. \yesnumber \label{20241015eq26}
\end{IEEEeqnarray*}
By Lemma \ref{20241009lem1}, we have \begin{eqnarray}\label{20241015eq28}
\Phi_{0, 4} \in S_{(\ell)}^{0}, \quad \partial_{x}\Phi_{0,4} - \frac{1}{6\langle \xi \rangle_{\ell}^{4}}\tIm\big(e_{(3)}(x)\big) \in S_{(\ell)}^{-5}, \quad \Phi_{4} \in S_{(\ell)}^{0}.
\end{eqnarray}
Then we have $L_{(3)} \circ \Phi_{4} \equiv \Phi_{4} \circ L_{(4)}\pmod{S^{0}}$ for any $\ell \ge 1$, where $L_{(4)}$ is defined by \begin{eqnarray} \label{20241015eq27}
L_{(4)} = D_{t}-D_{x}^{6}-A_{(3)}-\tRe\big(e_{(3)}(x)\big)D_{x}.
\end{eqnarray} The calculations are done in Appendix \ref{subsubappB.2.4}. Thus, the $L^{2}$-wellposedness of $L_{(3)}u=0$ is equivalent to $L_{(4)}u=0$, which is indeed $L^{2}$-wellposed since $A_{(3)}+\tRe\big(e_{(3)}(x)\big)D_{x}$ is self-adjoint. (See Appendix \ref{appA}.)

To sum up, the $L^{2}$-wellposedness of the original Cauchy problem $Lu=0$ is equivalent to the set of four conditions \eqref{20241010eq3}, \eqref{20241011eq1}, \eqref{20241014eq1}, and \eqref{20241014eq2}. Writing $c_{(1)}, d_{(2)},$ and $e_{(3)}$ in terms of $b, c, d, e$ gives the set of four conditions
\begin{IEEEeqnarray*}{rCl}
\left | \int_{x}^{y}\tIm\big(b(\widetilde{x})\big)\dd\widetilde{x} \right | &\lesssim&\hspace{-0.1cm}{}_{L}\,|x-y|^{\frac{1}{5}}, \\
\left | \int_{x}^{y}\tIm\left (c(\widetilde{x})-\frac{5}{2}\tIm\big(b'(\widetilde{x})\big)+2i\tRe\big(b'(\widetilde{x})\big)\right )\dd\widetilde{x} \right | &\lesssim&\hspace{-0.1cm}{}_{L}\,|x-y|^{\frac{2}{5}}, \\
\bigg | \int_{x}^{y}\tIm\bigg ( d(\widetilde{x})+\frac{1}{4}\Big ( \tIm\big(b(\widetilde{x})\big)\Big )^{2}-\frac{1}{2}i\tIm\big(b(\widetilde{x})\big)\tRe\big(b(\widetilde{x})\big )\qquad\qquad\; && \\
+ \frac{10}{3}i\tIm\big(b''(\widetilde{x})\big)-\frac{5}{2}\tIm\big(c_{(1)}'(\widetilde{x})\big)+\frac{3}{2}i\tRe\big(c_{(1)}'(\widetilde{x})\big)\bigg )\dd\widetilde{x} \bigg | &\lesssim&\hspace{-0.1cm}{}_{L}\,|x-y|^{\frac{3}{5}}, \\
\bigg | \int_{x}^{y}\tIm\bigg ( e(\widetilde{x})-\frac{1}{3}i\tIm\big(b(\widetilde{x})\big)c(\widetilde{x}) + \frac{1}{4}i \tIm\big(b'(\widetilde{x})\big)\tIm\big(b(\widetilde{x})\big) + \frac{1}{4}\tRe\big(b'(\widetilde{x})\big)\tIm\big(b(\widetilde{x})\big) \qquad \qquad\quad\;\; && \\
- \frac{7}{12}\tRe\big(b(\widetilde{x})\big)\tIm\big(b'(\widetilde{x})\big) +\frac{5}{2} \tIm \big (b'''(\widetilde{x}) \big )+i\tRe\big(b'''(\widetilde{x})\big)-\frac{1}{3}i\tRe\big(b(\widetilde{x})\big)\tIm\big(c(\widetilde{x})\big) \qquad \;
&& \\ -\frac{2}{3}i\tRe\big(b(\widetilde{x})\big)\tRe\big(b'(\widetilde{x})\big)+\frac{10}{3}i\tIm\big(c_{(1)}''(\widetilde{x})\big)-\frac{5}{2}\tIm\big(d_{(2)}'(\widetilde{x})\big )+i\tRe\big(d_{(2)}'(\widetilde{x})\big)\bigg )\dd\widetilde{x} \bigg | &\lesssim&\hspace{-0.1cm}{}_{L}\,|x-y|^{\frac{4}{5}}.
\end{IEEEeqnarray*}
Note that for any $p \in \overline{C}^{\infty}(\RR)$, one has $\left | \int_{x}^{y}p'(\widetilde{x})\dd\widetilde{x} \right | \lesssim \min\{|x-y|, 1\}$, so that $\left | \int_{x}^{y}p'(\widetilde{x})\dd\widetilde{x} \right | \lesssim |x-y|^{\alpha}$ for any $0 \le \alpha \le 1$. Using this fact simplifies the four conditions into
\begin{IEEEeqnarray*}{rCl} 
\left | \int_{x}^{y}\tIm\big(b(\widetilde{x})\big)\dd\widetilde{x} \right | &\lesssim&\hspace{-0.1cm}{}_{L}\,|x-y|^{\frac{1}{5}}, \\
\left | \int_{x}^{y}\tIm\big (c(\widetilde{x})\big )\dd\widetilde{x} \right | &\lesssim&\hspace{-0.1cm}{}_{L}\,|x-y|^{\frac{2}{5}}, \\
\left | \int_{x}^{y}\tIm\left ( d(\widetilde{x})-\frac{1}{4}b(\widetilde{x})^{2} \right )\dd\widetilde{x} \right | &\lesssim&\hspace{-0.1cm}{}_{L}\,|x-y|^{\frac{3}{5}}, \\
\bigg | \int_{x}^{y}\tIm\bigg ( e(\widetilde{x})-\frac{1}{3}b(\widetilde{x})c(\widetilde{x})\bigg )\dd\widetilde{x} \bigg | &\lesssim&\hspace{-0.1cm}{}_{L}\,|x-y|^{\frac{4}{5}}.
\end{IEEEeqnarray*}
These are exactly the four conditions appearing in the statement of Theorem \ref{20241015thm3}. 
\end{proof}

\appendix

\section{$L^{2}$-wellposedness for the self-adjoint case}\label{appA}

In this appendix, we prove the $L^{2}$-wellposedness for the Cauchy problem \[
\begin{cases}
Lu:=(D_{t}-D_{x}^{k}-A)u=0 & (x, t) \in \RR^{2} \\
u(x, t=0)=u_{0}(x) & x \in \RR
\end{cases},
\] where $k \in \{5, 6\}$ and $A$ is a self-adjoint differential operator of order $\le k-2$, using pseudodifferential operators.

For $k=5$, $A$ is of the form \[
\alpha(x)D_{x}^{3}+\left ( \beta(x) - \frac{3}{2}i\alpha'(x)\right )D_{x}^{2} + \left ( \gamma(x) - i\beta'(x) \right )D_{x},
\] where $\alpha, \beta, \gamma \in \overline{C}^{\infty}(\RR)$ are real-valued smooth functions whose derivatives of every order are bounded. Now fix $s \in \RR$ arbitrarily. Define $\langle D \rangle := \langle D_{x}\rangle := \mathrm{Op}\big(\sqrt{1+\xi^{2}}\big)$. Then, one can calculate that
\begin{IEEEeqnarray*}{rCl}
B_{s}&:=&\langle D \rangle^{s} \circ (D_{x}^{5}+A) \circ \langle D\rangle^{-s} \\
&\equiv& D_{x}^{5}+\alpha(x)D_{x}^{3}+\left ( \beta(x)-\frac{3}{2}i\alpha'(x)-is\alpha'(x)\right )D_{x}^{2}\\
&& +\left ( \gamma(x)-i\beta'(x)-is\left ( \beta'(x)-\frac{3}{2}i\alpha''(x)\right ) - \frac{1}{2}s(s-1)\alpha''(x) \right )D_{x} \pmod{S^{0}}.
\end{IEEEeqnarray*}
Now define $\Phi \in S_{(\ell)}^{0}$ by \[
\Phi(\ell;\xi, x) := 1 - \frac{s}{5\langle \xi \rangle_{\ell}^{2}}\alpha(x) - \frac{s}{5\langle \xi \rangle_{\ell}^{3}}\left ( \beta(x)+i\alpha'(x) - \frac{1}{2}is\alpha'(x)\right ).
\]
Then observe that
\begin{IEEEeqnarray*}{rCl}
B_{s} \circ \Phi - \Phi \circ B_{0} &\equiv& \left ( -is\alpha'(x)\xi^{2}-is\left(\beta'(x)-\frac{3}{2}i\alpha''(x)\right)\xi-\frac{1}{2}s(s-1)\alpha''(x)\xi \right ) \\
&& - i\left ( 5\xi^{4}\right )\left ( -\frac{s}{5\langle \xi \rangle_{\ell}^{2}}\alpha'(x) - \frac{s}{5\langle \xi \rangle_{\ell}^{3}}\left ( \beta'(x)+i\alpha''(x)-\frac{1}{2}is\alpha''(x)\right )\right ) \\
&& - \frac{1}{2}\left(20\xi^{3}\right ) \left ( -\frac{s}{5\langle \xi \rangle_{\ell}^{2}}\alpha''(x)\right ) \equiv  0 \pmod{S^{0}}.
\end{IEEEeqnarray*}
Since $\Phi-\mathrm{Id} \in S_{(\ell)}^{-1}$, $\Phi$ has its inverse in $S^{0}$ if $\ell \ge 1$ is sufficiently large. Fix such $\ell$ so that $\Phi$ and $\Phi^{-1}$ are in $S^{0}$. Define $E_{1}:=B_{s}\circ \Phi - \Phi \circ B_{0} \in S^{0}$, and $E_{2}:=B_{0}-B_{0}^{\ast} \in S^{0}$. Then,
\begin{IEEEeqnarray*}{rCl}
\IEEEeqnarraymulticol{3}{l}{\partial_{t}\left \|\Phi^{-1}\langle D \rangle^{s}u\right \|_{L^{2}}^{2}}
\\ &=& 2\tRe \left \langle \partial_{t} \Phi^{-1}\langle D \rangle^{s}u, \Phi^{-1}\langle D \rangle^{s}u \right \rangle \\
& = & 2\tRe\left \langle i(D_{t}-B_{0})\Phi^{-1}\langle D \rangle^{s}u, \Phi^{-1}\langle D \rangle^{s}u\right \rangle + 2\tRe\left \langle iB_{0}\Phi^{-1}\langle D \rangle^{s}u, \Phi^{-1}\langle D \rangle^{s}u\right \rangle \\
& = & 2\tRe\left \langle i\Phi^{-1}(D_{t}-B_{s}+E_{1}\Phi^{-1})\langle D \rangle^{s}u, \Phi^{-1}\langle D \rangle^{s}u\right \rangle + \tRe\left \langle i(B_{0}-B_{0}^{\ast})\Phi^{-1}\langle D \rangle^{s}u, \Phi^{-1}\langle D \rangle^{s}u\right \rangle \\
& = & 2\tRe\left \langle i\Phi^{-1}\langle D \rangle^{s}(D_{t}-B_{0})u, \Phi^{-1}\langle D \rangle^{s}u\right \rangle + 2\tRe\left\langle i\Phi^{-1}E_{1}\Phi^{-1}\langle D\rangle^{s}u, \Phi^{-1}\langle D\rangle^{s}u\right \rangle \\
&& + \tRe\left \langle iE_{2}\Phi^{-1}\langle D \rangle^{s}u, \Phi^{-1}\langle D \rangle^{s}u\right \rangle \\
& \lesssim & \|(D_{t}-B_{0})u\|_{H^{s}}\|u\|_{H^{s}}+\|u\|_{H^{s}}^{2} 
\end{IEEEeqnarray*}
where the constants depend on the $L^{2}$ norm of $E_{1}, E_{2}$, and $\Phi^{-1}$. These depends only on $L$ and $s$. Hence by integrating above inequality, one gets \[
\|u(t)\|_{H^{s}}^{2} \lesssim_{L,s} \|u(t=0)\|_{H^{s}}^{2} + \left | \int_{0}^{t} \left ( \|(D_{t}-B_{0})u(\tau)\|_{H^{s}}\|u(\tau)\|_{H^{s}}+\|u(\tau)\|_{H^{s}}^{2}\right )\dd\tau \right |,
\] and by Gr\"onwall inequality we conclude that \[
\|u(t)\|_{H^{s}} \lesssim_{L,s,T} \|u(t=0)\|_{H^{s}} + \left | \int_{0}^{t}\|(D_{t}-B_{0})u(\tau)\|_{H^{s}}\dd\tau \right | \quad \forall \; |t| \le T.
\] Since $B_{0}^{\ast}$ is self-adjoint, similar calculation proves the time-reversed estimate for $D_{t}-B_{0}^{\ast}$ too. Since $s$ was arbitrary, we conclude that $(D_{t}-B_{0})u=0$ is $L^{2}$-wellposed, as desired.

For $k=6$, $A$ is of the form \[
\alpha(x)D_{x}^{4} + \left ( \beta(x)-2i\alpha'(x)\right )D_{x}^{3} + \left ( \gamma(x)-\frac{3}{2}i\beta'(x)\right )D_{x}^{2}+\left ( \delta(x)-i\gamma'(x)-i\alpha'''(x)\right )D_{x},
\] where $\alpha, \beta, \gamma, \delta \in \overline{C}^{\infty}(\RR)$ are real-valued. Now fix arbitrary $s \in \RR$. Then one can calculate that
\begin{IEEEeqnarray*}{rCl}
B_{s}&:=&\langle D \rangle^{s} \circ (D_{x}^{6}+A) \circ \langle D\rangle^{-s} \\
&\equiv& D_{x}^{6}+\alpha(x)D_{x}^{4} + \left ( \beta(x)-2i\alpha'(x)-is\alpha'(x)\right )D_{x}^{3} \\
&& + \left ( \gamma(x)-\frac{3}{2}i\beta'(x)-is\left (\beta'(x)-2i\alpha''(x)\right) - \frac{1}{2}s(s-1)\alpha''(x)\right )D_{x}^{2} \\
&& + \bigg ( \delta(x) - i\gamma'(x)-i\alpha'''(x)+is\alpha'(x)-is\left (\gamma'(x)-\frac{3}{2}i\beta''(x)\right )\\
&& \quad\;\;\; -\frac{1}{2}s(s-1)\left (\beta''(x)-2i\alpha'''(x)\right )+\frac{1}{6}is(s-1)(s-2)\alpha'''(x) \bigg )D_{x} \\
&\equiv & D_{x}^{6}+\alpha(x)D_{x}^{4} + \left ( \beta(x)-i(s+2)\alpha'(x)\right )D_{x}^{3} + \left (\gamma(x)-i\left(s+\frac{3}{2}\right )\beta'(x) - \frac{1}{2}s(s+3)\alpha''(x)\right )D_{x}^{2} \\
&& + \left ( \delta(x) - i(s+1)\gamma'(x)-\frac{1}{2}s(s+2)\beta''(x)+\frac{1}{6}is(s-1)(s+4)\alpha'''(x)-i\alpha'''(x)+is\alpha'(x) \right )D_{x} 
\end{IEEEeqnarray*}
Define $\Phi \in S_{(\ell)}^{0}$ by \begin{IEEEeqnarray*}{rCl}
\Phi(\ell;\xi,x)&:=&1 - \frac{s}{6\langle \xi\rangle_{\ell}^{2}}\alpha(x)-\frac{s}{6\langle \xi\rangle_{\ell}^{3}}\left ( \beta(x)+i\alpha'(x)-\frac{1}{2}is\alpha'(x)\right )\\
&& -\frac{s}{6\langle \xi\rangle_{\ell}^{4}}\left ( \gamma(x)+\frac{3}{2}i\beta'(x)-\frac{1}{2}is\beta'(x)+\left (\frac{3}{2}+\frac{3}{4}s-\frac{1}{6}s^{2}\right )\alpha''(x) - \frac{1}{2}\alpha(x)^{2}-\frac{s}{12}\alpha(x)^{2} \right )
\end{IEEEeqnarray*}
Then, we have
\begin{IEEEeqnarray*}{rCl}
\IEEEeqnarraymulticol{3}{l}{B_{s}\circ\Phi - \Phi \circ B_{0}} \\
&\equiv& -is\alpha'(x)\xi^{3}-\left(is\beta'(x)+\frac{1}{2}s(s+3)\alpha''(x)\right )\xi^{2} \\
&& -\left ( is\gamma'(x)+\frac{1}{2}s(s+2)\beta''(x)-\frac{1}{6}is(s-1)(s+4)\alpha'''(x)-is\alpha'(x)\right )\xi \\
&& + \left (-is\alpha'(x)\xi^{3}\right)\left(-\frac{s}{6\langle \xi\rangle_{\ell}^{2}}\alpha(x)\right ) - i\left (6\xi^{5} + 4\alpha(x)\xi^{3}\right )\left (-\frac{s}{6\langle \xi\rangle_{\ell}^{2}}\alpha'(x)\right ) \\
&& - i\left (6\xi^{5}\right )\bigg ( -\frac{s}{6\langle \xi\rangle_{\ell}^{3}}\left ( \beta'(x)+i\alpha''(x)-\frac{1}{2}is\alpha''(x)\right ) \\
&& -\frac{s}{6\langle \xi\rangle_{\ell}^{4}}\bigg ( \gamma(x)'+\frac{3}{2}i\beta''(x)-\frac{1}{2}is\beta''(x) +\left (\frac{3}{2}+\frac{3}{4}s-\frac{1}{6}s^{2}\right )\alpha'''(x) - \alpha(x)\alpha'(x) -\frac{s}{6}\alpha(x)\alpha'(x)\bigg )\bigg ) \\
&& - \frac{1}{2}\left (30\xi^{4}\right )\left (-\frac{s}{6\langle \xi\rangle_{\ell}^{2}}\alpha''(x)-\frac{s}{6\langle \xi\rangle_{\ell}^{3}}\left (\beta''(x)+i\alpha'''(x)-\frac{1}{2}is\alpha'''(x)\right )\right ) + \frac{1}{6}i\left (120\xi^{3}\right)\left (-\frac{s}{6\langle \xi\rangle_{\ell}^{2}}\alpha'''(x)\right) \\
&& +i\left ( \frac{s\xi}{3\langle \xi\rangle_{\ell}^{4}}\alpha(x)\right )\left ( \alpha'(x)\xi^{4}\right ) \\
&\equiv& -is\alpha'(x)\xi^{3}-is\beta'(x)\xi^{2}-\frac{1}{2}s(s+3)\alpha''(x)\xi^{2}- is\gamma'(x)\xi -\frac{1}{2}s(s+2)\beta''(x)\xi+\frac{1}{6}is(s-1)(s+4)\alpha'''(x)\xi\\
&& +is\alpha'(x)\xi+\frac{1}{6}is^{2}\alpha(x)\alpha'(x)\xi +is\alpha'(x)\xi^{3}-is\alpha'(x)\xi + \frac{2}{3}is\alpha(x)\alpha'(x)\xi +\bigg ( is\xi^{2}\left ( \beta'(x)+i\alpha''(x)-\frac{1}{2}is\alpha''(x)\right ) \\
&& +is\xi\bigg ( \gamma(x)'+\frac{3}{2}i\beta''(x)-\frac{1}{2}is\beta''(x) +\left (\frac{3}{2}+\frac{3}{4}s-\frac{1}{6}s^{2}\right )\alpha'''(x) - \alpha(x)\alpha'(x) -\frac{s}{6}\alpha(x)\alpha'(x) \bigg )\bigg ) \\
&& +\left (\frac{5}{2}s\xi^{2}\alpha''(x)+\frac{5}{2}s\xi\left (\beta''(x)+i\alpha'''(x)-\frac{1}{2}is\alpha'''(x)\right )\right ) -\frac{10}{3}is\xi\alpha'''(x) + \frac{1}{3}is\alpha(x)\alpha'(x)\xi \equiv 0\pmod{S^{0}}.
\end{IEEEeqnarray*}
Now arguing as before gives that $Lu=0$ is $L^{2}$-wellposed.

\section{Deferred calculations}\label{appB}

We collect here the calculations skipped in the proof of Theorem \ref{20241015thm2} and \ref{20241015thm3}.

\subsection{Calculations in the proof of Theorem \ref{20241015thm2} ($k=5$)}\label{subappB.1}

\subsubsection{$L \circ \Phi_{1} \equiv \Phi_{1} \circ L_{(1)}$} \label{subsubappB.1.1}

We will check that $L \circ \Phi_{1} \equiv \Phi_{1} \circ L_{(1)} \equiv 0 \pmod{S^{0}}$ for any $\ell \ge 1$, where $L, \Phi_{1},$ and $L_{(1)}$ are defined in \eqref{20241015eq2}, \eqref{20241015eq1}, and \eqref{20241015eq3}, respectively. Recall that we had \eqref{20241015eq4}, that is, 
\[
\Phi_{0, 1} \in S_{(\ell)}^{0}, \quad \partial_{x}\Phi_{0, 1} - \frac{1}{5}\left ( \frac{1}{\langle \xi \rangle_{\ell}} + \frac{\ell^{2}}{2\langle \xi \rangle_{\ell}^{3}} \right ) \tIm \big ( b(x)\big ) \in S_{(\ell)}^{-4}, \quad \Phi_{1} \in S_{(\ell)}^{0}
\] Having this in mind, one can calculate that
\begin{IEEEeqnarray*}{rCl}
L \circ \Phi_{1}&\equiv& (D_{x}^{5} + b(x)D_{x}^{3} + c(x)D_{x}^{2}+d(x)D_{x}+e(x)) \circ \Phi_{1} \\
&\equiv& \sum_{j=0}^{4} \frac{i^{-j}}{j!}\partial_{\xi}^{j}\left ( \xi^{5}+b(x)\xi^{3}+c(x)\xi^{2}+d(x)\xi +e(x)\right)\partial_{x}^{j}\Phi_{1} \\
& \equiv & e^{\Phi_{0, 1}}\left ( \xi^{5}+b(x)\xi^{3}+c(x)\xi^{2}+d(x)\xi+e(x)\right ) \\
&& + e^{\Phi_{0, 1}}\left ( \xi^{5}+b(x)\xi^{3}\right)\left ( \frac{1}{5\langle\xi\rangle_{\ell}} \partial_{\xi}\Phi_{0, 1}\tRe\big(b(x)\big) \right ) \\
&& + e^{\Phi_{0, 1}}\left ( \xi^{5}\right) \left (\frac{1}{5\langle\xi\rangle_{\ell}^{2}}\partial_{\xi}\Phi_{0,1}\left (c(x)+2ib'(x)\right ) -\frac{1}{10\langle\xi\rangle_{\ell}}\left ( \partial_{\xi}^{2}\Phi + \big(\partial_{\xi}\Phi\big)^{2} \right )i \tRe\big(b'(x)\big) \right ) \\
&& -ie^{\Phi_{0, 1}}\left ( 5\xi^{4}+3b(x)\xi^{2}\right )\partial_{x}\Phi_{0,1}\left ( 1+ \frac{1}{5\langle\xi\rangle_{\ell}} \partial_{\xi}\Phi_{0, 1}\tRe\big(b(x)\big) \right ) \\
&& -ie^{\Phi_{0, 1}}\left( 5\xi^{4}\right )\bigg ( \frac{1}{5\langle\xi\rangle_{\ell}} \partial_{x}\partial_{\xi}\Phi_{0, 1}\tRe\big(b(x)\big)+\frac{1}{5\langle\xi\rangle_{\ell}} \partial_{\xi}\Phi_{0, 1}\tRe\big(b'(x)\big)  \\
&& \qquad \qquad \quad \quad \;\; +\frac{1}{5\langle\xi\rangle_{\ell}^{2}}\partial_{\xi}\Phi_{0,1}\left (c'(x)+2ib''(x)\right ) -\frac{1}{10\langle\xi\rangle_{\ell}}\left ( \partial_{\xi}^{2}\Phi + \big(\partial_{\xi}\Phi\big)^{2} \right )i \tRe\big(b''(x)\big)\bigg) \\
&& - \frac{1}{2}e^{\Phi_{0,1}}\left ( 20\xi^{3} \right )\left ( \big ( \partial_{x}\Phi_{0,1}\big)^{2} + \partial_{xx}\Phi_{0,1}+\frac{1}{5\langle \xi\rangle_{\ell}}\partial_{\xi}\Phi_{0,1}\tRe\big(b''(x)\big)\right ) \\
&& + \frac{1}{6}ie^{\Phi_{0,1}}\left ( 60\xi^{2} \right ) \left (\partial_{xxx}\Phi_{0,1} \right ) \\
& \equiv & e^{\Phi_{0, 1}}\left ( \xi^{5}+b(x)\xi^{3}+c(x)\xi^{2}+d(x)\xi\right ) \\
&& + e^{\Phi_{0, 1}}\left ( \xi^{4}-\frac{\ell^{2}}{2}\xi^{2}+b(x)\xi^{2}\right)\left ( \frac{1}{5} \partial_{\xi}\Phi_{0, 1}\tRe\big(b(x)\big) \right ) \\
&& + e^{\Phi_{0, 1}} \left (\frac{1}{5}\xi^{3}\partial_{\xi}\Phi_{0,1}\left (c(x)+2ib'(x)\right ) -\frac{1}{10}\xi^{4}\left ( \partial_{\xi}^{2}\Phi + \big(\partial_{\xi}\Phi\big)^{2} \right )i \tRe\big(b'(x)\big) \right ) \\
&& -ie^{\Phi_{0, 1}}\left (  \tIm \big ( b(x)\big )\xi^{3} +\frac{1}{5}\xi^{2}\partial_{\xi}\Phi_{0,1}\tRe\big(b(x)\big)\tIm\big(b(x)\big)+ \frac{3}{5}b(x)\tIm\big(b(x)\big)\xi\right ) \\
&& -ie^{\Phi_{0, 1}}\bigg ( -\frac{1}{5}\xi \tIm \big ( b(x)\big ) \tRe\big(b(x)\big)+\xi^{3}\partial_{\xi}\Phi_{0, 1}\tRe\big(b'(x)\big) +\xi^{2}\partial_{\xi}\Phi_{0,1}\left (c'(x)+2ib''(x)\right ) \\
&& \qquad \qquad \, -\frac{1}{2}\xi^{3}\left ( \partial_{\xi}^{2}\Phi + \big(\partial_{\xi}\Phi\big)^{2} \right )i \tRe\big(b''(x)\big)\bigg) \\
&& - e^{\Phi_{0,1}}\left ( \frac{2}{5}\xi \left (  \tIm \big ( b(x)\big )\right)^{2} + 2 \tIm \big ( b'(x)\big )\xi^{2}+2\xi^{2}\partial_{\xi}\Phi_{0,1}\tRe\big(b'(x)\big)\right ) \\
&& + e^{\Phi_{0,1}}\left ( 2i \tIm \big ( b''(x)\big )\xi \right ) \\
&\equiv & e^{\Phi_{0,1}} \Bigg [ \xi^{5} + \left ( \tRe\big(b(x)\big) + \frac{1}{5}\tRe\big(b(x)\big)\xi\partial_{\xi}\Phi_{0,1}\right )\xi^{3} \\
&& \qquad \quad+ \bigg ( c(x) + \frac{1}{5}c(x)\xi\partial_{\xi}\Phi_{0,1}+\frac{2}{5}ib'(x)\xi\partial_{\xi}\Phi_{0,1}-\frac{1}{10}i\tRe\big(b'(x)\big)\xi^{2}\Big ( \partial_{\xi}^{2}\Phi_{0,1} + \big ( \partial_{\xi}\Phi_{0,1} \big )^{2} \Big ) \\
&& \qquad \quad\qquad - i\tRe\big(b'(x)\big)\xi\partial_{\xi}\Phi_{0,1} - 2\tIm\big(b'(x)\big)\bigg ) \xi^{2} \\
&& \qquad \quad + \bigg ( d(x) - \frac{\ell^{2}}{10}\tRe\big(b(x)\big)\xi\partial_{\xi}\Phi_{0,1} + \frac{1}{5}\Big ( \tRe\big(b(x)\big)\Big)^{2}\xi\partial_{\xi}\Phi_{0,1} - \frac{3}{5}ib(x)\tIm\big(b(x)\big)  \\
&&\qquad \quad \qquad + \frac{1}{5}i\tRe\big(b(x)\big)\tIm\big(b(x)\big)-ic'(x)\xi\partial_{\xi}\Phi_{0,1} +2b''(x)\xi \partial_{\xi}\Phi_{0,1}  \\
&& \qquad \quad \qquad - \frac{1}{2}\tRe\big(b''(x)\big)\xi^{2}\Big ( \partial_{\xi}^{2}\Phi_{0,1}+\big(\partial_{\xi}\Phi_{0,1}\big)^{2}\Big) - \frac{2}{5}\Big ( \tIm \big (b(x)\big)\Big)^{2} \\
&& \qquad \quad \qquad -2\tRe\big(b''(x)\big)\xi \partial_{\xi}\Phi_{0,1} + 2i\tIm\big(b''(x)\big) \bigg )\xi\Bigg ] \pmod{S^{0}}, \\
\Phi_{1} \circ L_{(1)} &\equiv &\Phi_{1} \circ \bigg ( D_{x}^{5} + \tRe\big(b(x)\big)D_{x}^{3} +\left ( c(x) - 2\tIm \big (b'(x) \big ) \right ) D_{x}^{2} + \bigg ( d(x) + 2i\tIm\big(b''(x)\big) \\
&& \qquad \quad- \frac{2}{5}\Big ( \tIm \big (b(x) \big ) \Big )^{2} - \frac{3}{5}ib(x)\tIm \big (b(x) \big ) + \frac{1}{5}i\tRe\big(b(x)\big ) \tIm \big (b(x) \big ) \bigg ) D_{x}\bigg ) \\
& \equiv & e^{\Phi_{0,1}} \bigg ( \xi^{5} + \tRe\big(b(x)\big)\xi^{3} + \left ( c(x)-2\tIm\big(b'(x)\big)\right )\xi^{2} + \bigg (  d(x) + 2i\tIm\big(b''(x)\big)  \\
&& \qquad \quad- \frac{2}{5}\Big ( \tIm \big (b(x) \big ) \Big )^{2}- \frac{3}{5}ib(x)\tIm \big (b(x) \big ) + \frac{1}{5}i\tRe\big(b(x)\big ) \tIm \big (b(x) \big )\bigg )\xi \bigg ) \\
&& + e^{\Phi_{0,1}}\left ( \frac{1}{5\langle \xi \rangle_{\ell}} \partial_{\xi}\Phi_{0,1}\tRe\big(b(x)\big)\right ) \left ( \xi^{5} + \tRe\big(b(x)\big)\xi^{3}\right ) \\
&& + e^{\Phi_{0,1}} \left ( \frac{1}{5\langle\xi\rangle_{\ell}^{2}}\partial_{\xi}\Phi_{0,1}\left (c(x)+2ib'(x)\right ) -\frac{1}{10\langle\xi\rangle_{\ell}}\left ( \partial_{\xi}^{2}\Phi + \big(\partial_{\xi}\Phi\big)^{2} \right )i \tRe\big(b'(x)\big) \right ) \left ( \xi^{5} \right ) \\
&& - i e^{\Phi_{0,1}}\partial_{\xi}\Phi_{0,1}\left ( \tRe\big(b'(x)\big)\xi^{3} + \left ( c'(x)-2\tIm\big(b''(x)\big) \right ) \xi^{2}\right ) \\
&& - \frac{1}{2}e^{\Phi_{0,1}}\Big ( \partial_{\xi}^{2}\Phi_{0,1} + \big ( \partial_{\xi}\Phi_{0,1}\big )^{2} \Big ) \left ( \tRe \big (b''(x)\big)\xi^{3} \right ) \\
& \equiv & e^{\Phi_{0,1}}\Bigg [ \xi^{5} + \left ( \tRe\big(b(x)\big) + \frac{1}{5}\tRe\big(b(x)\big)\xi\partial_{\xi}\Phi_{0,1}\right )\xi^{3}  + \bigg ( c(x)-2\tIm\big(b'(x)\big)+\frac{1}{5}c(x)\xi\partial_{\xi}\Phi_{0,1} \\
&&\qquad \quad+\frac{2}{5}ib'(x)\xi\partial_{\xi}\Phi_{0,1}-\frac{1}{10}i\tRe\big(b'(x)\big)\xi^{2}\Big(\partial_{\xi}^{2}\Phi_{0,1} + \big ( \partial_{\xi}\Phi_{0,1}\big)^{2}\Big)-i\tRe\big(b'(x)\big)\xi\partial_{\xi}\Phi_{0,1}\bigg )\xi^{2} \\
&& \qquad \quad+ \bigg ( d(x)+2i\tIm\big(b''(x)\big)-\frac{2}{5}\Big(\tIm\big(b(x)\big)\Big)^{2}-\frac{3}{5}ib(x)\tIm\big(b(x)\big)+\frac{1}{5}i\tRe\big(b(x)\big)\tIm\big(b(x)\big) \\
&& \qquad \quad \qquad-\frac{\ell^{2}}{10}\tRe\big(b(x)\big)\xi\partial_{\xi}\Phi_{0,1}+ \frac{1}{5}\Big(\tRe\big(b(x)\big)\Big)^{2}\xi\partial_{\xi}\Phi_{0,1}-ic'(x)\xi\partial_{\xi}\Phi_{0,1}\\
&& \qquad \quad \qquad+2i\tIm\big(b''(x)\big)\xi\partial_{\xi}\Phi_{0,1}-\frac{1}{2}\tRe\big(b''(x)\big)\xi^{2}\Big(\partial_{\xi}^{2}\Phi_{0,1}+\big(\partial_{\xi}\Phi_{0,1}\big)^{2}\Big)\bigg )\xi\Bigg ] \pmod{S^{0}}.
\end{IEEEeqnarray*}
Hence $L \circ \Phi_{1} \equiv \Phi_{1} \circ L_{(1)} \pmod{S^{0}}$.

\subsubsection{$L_{(1)} \circ \Phi_{2} \equiv \Phi_{2} \circ L_{(2)}$} \label{subsubappB.1.2}

We check that $L_{(1)} \circ \Phi_{2} \equiv \Phi_{2} \circ L_{(2)} \pmod{S^{0}}$ for any $\ell \ge 1$, where $L_{(1)}, \Phi_{2}$, and $L_{(2)}$ are defined in \eqref{20241015eq5}, \eqref{20241015eq6}, and \eqref{20241015eq7} respectively. Note that by \eqref{20241015eq8} we have \[
\Phi_{0, 2} \in S_{(\ell)}^{0}, \quad \partial_{x}\Phi_{0,2}-\frac{1}{5\langle \xi \rangle_{\ell}^{2}}\tIm \big (c_{(1)}(x) \big ) \in S_{(\ell)}^{-4}, \quad \Phi_{2} \in S_{\ell}^{(0)}.
\]
Then we can calculate that
\begin{IEEEeqnarray*}{rCl}
L_{(1)} \circ \Phi_{2} & \equiv & \left (D_{x}^{5}+A_{(1)} +c_{(1)}(x)D_{x}^{2}+d_{(1)}(x)D_{x}\right ) \circ \Phi_{2} \\
& \equiv & e^{\Phi_{0,2}}\bigg [ \left ((\xi^{5} + \tRe\big(b(x)\big)\xi^{3} + \left ( c_{(1)}(x)-\frac{3}{2}i\tRe\big(b'(x)\big)\right )\xi^{2} + d_{(1)}(x)\xi\right ) \\
&& \qquad \quad + \left ( \xi^{5} + \tRe\big(b(x)\big)\xi^{3} \right ) \frac{1}{5\langle\xi\rangle_{\ell}}\tRe\big(b(x)\big)\partial_{\xi}\Phi_{0,2} \\
&& \qquad \quad+ \left ( \xi^{5} \right ) \bigg ( \frac{1}{10\langle\xi\rangle_{\ell}^{2}}\left (2c_{(1)}(x)+i\tRe\big(b'(x)\big)\right )\partial_{\xi}\Phi_{0,2} - i \partial_{x}\Phi_{0,2}\partial_{\xi}\Phi_{0,2}  \\
&& \qquad \quad \qquad\qquad- \frac{1}{10\langle\xi\rangle_{\ell}}i\tRe\big(b'(x)\big)\Big ( \partial_{\xi}^{2}\Phi_{0,2}+\big(\partial_{\xi}\Phi_{0,2}\big)^{2}\Big)\bigg ) \\
&& \qquad \quad - i \left ( 5\xi^{4}\right )\bigg ( \partial_{x}\Phi_{0,2} + \frac{1}{5\langle \xi\rangle_{\ell}}\tRe\big(b'(x)\big)\partial_{\xi}\Phi_{0,2} + \frac{1}{10\langle\xi\rangle_{\ell}^{2}}\left (2c_{(1)}'(x)+i\tRe\big(b''(x)\big)\right )\partial_{\xi}\Phi_{0,2} \\
&& \qquad \quad \qquad \qquad \;\; - i \partial_{x}^{2}\Phi_{0,2}\partial_{\xi}\Phi_{0,2}  - \frac{1}{10\langle\xi\rangle_{\ell}}i\tRe\big(b''(x)\big)\Big ( \partial_{\xi}^{2}\Phi_{0,2}+\big(\partial_{\xi}\Phi_{0,2}\big)^{2}\Big)\bigg ) \\
&& \qquad \quad - \frac{1}{2}\left ( 20\xi^{3}\right)\left ( \partial_{x}^{2}\Phi_{0,2}+\frac{1}{5\langle \xi \rangle_{\ell}}\tRe\big(b''(x)\big)\partial_{\xi}\Phi_{0,2}\right )\bigg ] \\
& \equiv & e^{\Phi_{0,2}} \bigg [ \left ((\xi^{5} + \tRe\big(b(x)\big)\xi^{3} -\frac{3}{2}i\tRe\big(b'(x)\big)\xi^{2} \right ) + \left ( \xi^{5} + \tRe\big(b(x)\big)\xi^{3} \right ) \frac{1}{5\langle\xi\rangle_{\ell}}\tRe\big(b(x)\big)\partial_{\xi}\Phi_{0,2} \\
&& \qquad \quad+ \left ( \xi^{5} \right ) \bigg ( \frac{1}{10\langle\xi\rangle_{\ell}^{2}}\left (2c_{(1)}(x)+i\tRe\big(b'(x)\big)\right )\partial_{\xi}\Phi_{0,2} - i \partial_{x}\Phi_{0,2}\partial_{\xi}\Phi_{0,2} \\
&& \qquad \quad \qquad \qquad - \frac{1}{10\langle\xi\rangle_{\ell}}i\tRe\big(b'(x)\big)\Big ( \partial_{\xi}^{2}\Phi_{0,2}+\big(\partial_{\xi}\Phi_{0,2}\big)^{2}\Big)\bigg ) + \tRe\big(c_{(1)}(x)\big)\xi^{2} \\
&& \qquad \quad +d_{(1)}(x)\xi  - i\tRe\big(b'(x)\big)\xi^{3}\partial_{\xi}\Phi_{0,2} - ic_{(1)}'(x)\xi^{2}\partial_{\xi}\Phi_{0,2}- \tIm\big(c_{(1)}'(x)\big)\xi^{2}\partial_{\xi}\Phi_{0,2} \\
&& \qquad \quad  - \frac{1}{2}\tRe\big(b''(x)\big)\xi^{3}\Big ( \partial_{\xi}^{2}\Phi_{0,2}+\big(\partial_{\xi}\Phi_{0,2}\big)^{2}\Big ) - 2\tIm\big(c_{(1)}'(x)\big)\xi - \frac{3}{2}\tRe\big(b''(x)\big)\xi^{2}\partial_{\xi}\Phi_{0,2} \bigg ] \pmod{S^{0}}, \\
\Phi_{2} \circ L_{(2)} &\equiv & \Phi_{2} \circ \left ( D_{x}^{5} + A_{(1)} + \tRe\big(c_{(1)}(x)\big)D_{x}^{2}+\left ( d_{(1)}(x)-2\tIm\big(c_{(1)}'(x)\big)\right )D_{x} \right ) \\
&\equiv& e^{\Phi_{0,2}} \bigg [ \left ( \xi^{5} + \tRe\big(b(x)\big)\xi^{3} + \left ( \tRe\big(c_{(1)}(x)\big ) - \frac{3}{2}i\tRe\big(b'(x)\big ) \right )\xi^{2} + \left (d_{(1)}(x)-2\tIm\big(c_{(1)}'(x)\big)\right )\xi \right ) \\
&& \qquad \quad + \left ( \xi^{5} + \tRe\big(b(x)\big)\xi^{3} \right ) \frac{1}{5\langle\xi\rangle_{\ell}}\tRe\big(b(x)\big)\partial_{\xi}\Phi_{0,2} \\
&& \qquad \quad+ \left ( \xi^{5} \right ) \bigg ( \frac{1}{10\langle\xi\rangle_{\ell}^{2}}\left (2c_{(1)}(x)+i\tRe\big(b'(x)\big)\right )\partial_{\xi}\Phi_{0,2} - i \partial_{x}\Phi_{0,2}\partial_{\xi}\Phi_{0,2}  \\
&& \qquad \quad \qquad \qquad- \frac{1}{10\langle\xi\rangle_{\ell}}i\tRe\big(b'(x)\big)\Big ( \partial_{\xi}^{2}\Phi_{0,2}+\big(\partial_{\xi}\Phi_{0,2}\big)^{2}\Big)\bigg ) \\
&& \qquad \quad - i \left ( \partial_{\xi}\Phi_{0,2}\right )\left ( \tRe\big(b'(x)\big)\xi^{3} + \left ( \tRe\big(c_{(1)}'(x)\big ) - \frac{3}{2}i\tRe\big(b''(x)\big ) \right )\xi^{2} \right ) \\
&& \qquad \quad - \frac{1}{2} \Big( \partial_{\xi}^{2}\Phi_{0,2} + \big (\partial_{\xi}\Phi_{0,2}\big )^{2} \Big )\left ( \tRe\big(b''(x)\big)\xi^{3} \right )\bigg] \\
& \equiv & e^{\Phi_{0,2}} \bigg [ \left ((\xi^{5} + \tRe\big(b(x)\big)\xi^{3} -\frac{3}{2}i\tRe\big(b'(x)\big)\xi^{2} \right ) + \left ( \xi^{5} + \tRe\big(b(x)\big)\xi^{3} \right ) \frac{1}{5\langle\xi\rangle_{\ell}}\tRe\big(b(x)\big)\partial_{\xi}\Phi_{0,2} \\
&& \qquad \quad+ \left ( \xi^{5} \right ) \bigg ( \frac{1}{10\langle\xi\rangle_{\ell}^{2}}\left (2c_{(1)}(x)+i\tRe\big(b'(x)\big)\right )\partial_{\xi}\Phi_{0,2} - i \partial_{x}\Phi_{0,2}\partial_{\xi}\Phi_{0,2} \\
&& \qquad \quad \qquad \qquad - \frac{1}{10\langle\xi\rangle_{\ell}}i\tRe\big(b'(x)\big)\Big ( \partial_{\xi}^{2}\Phi_{0,2}+\big(\partial_{\xi}\Phi_{0,2}\big)^{2}\Big)\bigg ) + \tRe\big(c_{(1)}(x)\big)\xi^{2}\\ 
&& \qquad \quad  + \left ( d_{(1)}(x)-2\tIm\big(c_{(1)}'(x)\big)\right )\xi -i\tRe\big(b'(x)\big)\xi^{3}\partial_{\xi}\Phi_{0,2} - i \tRe\big(c_{(1)}'(x)\big)\xi^{2}\partial_{\xi}\Phi_{0,2} \\
&& \qquad \quad-\frac{3}{2}\tRe\big(b''(x)\big)\xi^{2}\partial_{\xi}\Phi_{0,2} - \frac{1}{2}\tRe\big(b''(x)\big)\xi^{3}\Big ( \partial_{\xi}^{2}\Phi_{0,2}+\big(\partial_{\xi}\Phi_{0,2}\big)^{2}\Big)\bigg ] \pmod{S^{0}}.
\end{IEEEeqnarray*}
Hence $L_{(1)} \circ \Phi_{2} \equiv \Phi_{2} \circ L_{(2)} \pmod{S^{0}}$.

\subsubsection{$L_{(2)} \circ \Phi_{3} \equiv \Phi_{3} \circ L_{(3)}$}\label{subsubappB.1.3}

We check that $L_{(2)} \circ\Phi_{3} \equiv \Phi_{3} \circ L_{(3)} \pmod{S^{0}}$ for all $\ell \ge 1$, where $L_{(2)}, \Phi_{3},$ and $L_{(3)}$ are defined in \eqref{20241015eq9}, \eqref{20241015eq10}, and \eqref{20241015eq11} respectively. Note that we have \eqref{20241015eq12}, that is, \[
\Phi_{0,3} \in S_{(\ell)}^{0}, \quad \partial_{x}\Phi_{0,3} - \frac{1}{5\langle \xi \rangle_{\ell}^{3}}\tIm \big (d_{(2)}(x) \big ) \in S_{(\ell)}^{-4}, \quad \Phi_{3} \in S_{(\ell)}^{0}.
\] Then,
\begin{IEEEeqnarray*}{rCl}
\IEEEeqnarraymulticol{3}{l}{L_{(2)} \circ \Phi_{3} - \Phi_{3} \circ L_{(3)}}\\
 &\equiv & \left (D_{x}^{5}+A_{(2)}+d_{(2)}(x)D_{x}\right ) \circ \Phi_{3}- \Phi_{3} \circ  \left (D_{x}^{5}+A_{(2)}+\tRe\big(d_{(2)}(x)\big)D_{x}\right ) \\
& \equiv & e^{\Phi_{0,3}} \bigg [\tIm\big(d_{(2)}(x)\big)\xi - i\left (5\xi^{4}\right )\bigg ( \partial_{x}\Phi_{0,3} + \frac{1}{5\langle \xi \rangle_{\ell}} \tRe\big(b'(x)\big)\partial_{\xi}\Phi_{0,3}\\
&& \qquad \quad  + \frac{1}{5\langle \xi \rangle_{\ell}^{2}} \left ( \tRe\big(c_{(1)}'(x)\big) + \frac{1}{2}i\tRe\big(b''(x)\big) \right )\partial_{\xi}\Phi_{0,3}  - \frac{1}{10\langle \xi \rangle_{\ell}}i\tRe\big(b''(x)\big)\Big ( \partial_{\xi}^{2}\Phi_{0,3} + \big ( \partial_{\xi}\Phi_{0,3}\big )^{2} \Big )\bigg )  \\
&& \qquad \quad - \frac{1}{2}\left ( 20\xi^{3} \right ) \left ( \frac{1}{5\langle \xi \rangle_{\ell}}\tRe\big(b''(x)\big)\partial_{\xi}\Phi_{0,3} \right )\\
&& \qquad \quad+ i \left ( \partial_{\xi}\Phi_{0,3} \right ) \left ( \tRe\big(b'(x)\big)\xi^{3} - \frac{3}{2}i\tRe\big(b''(x)\big)\xi^{2} + \tRe\big(c_{(1)}'(x)\big)\xi^{2} \right ) \\
&& \qquad \quad + \frac{1}{2} \Big ( \partial_{\xi}^{2}\Phi_{0,3}+\big(\partial_{\xi}\Phi_{0,3}\big)^{2}\Big ) \left ( \tRe\big(b''(x)\big)\xi^{3} \right )\bigg ] \\
&\equiv & e^{\Phi_{0,3}} \bigg [ -i\tRe\big(b'(x)\big)\xi^{3}\partial_{\xi}\Phi_{0,3} - i \left ( \tRe\big(c_{(1)}'(x)\big) + \frac{1}{2}i\tRe\big(b''(x)\big) \right )\xi^{2} \partial_{\xi}\Phi_{0,3} \\
&& \qquad \quad - \frac{1}{2}\tRe\big(b''(x)\big)\xi^{3}\Big ( \partial_{\xi}^{2}\Phi_{0,3}+\big(\partial_{\xi}\Phi_{0,3}\big)^{2}\Big )- 2\tRe\big(b''(x)\big)\xi^{2}\partial_{\xi}\Phi_{0,3} + i\tRe\big(b'(x)\big)\xi^{3}\partial_{\xi}\Phi_{0,3}  \\
&& \qquad \quad + \frac{3}{2}\tRe\big(b''(x)\big)\xi^{2}\partial_{\xi}\Phi_{0,3} + i\tRe\big(c_{(1)}'(x)\big)\xi^{2}\partial_{\xi}\Phi_{0,3}\\
&& \qquad \quad + \frac{1}{2} \Big ( \partial_{\xi}^{2}\Phi_{0,3}+\big(\partial_{\xi}\Phi_{0,3}\big)^{2}\Big ) \left ( \tRe\big(b''(x)\big)\xi^{3} \right )  \bigg ] \equiv 0 \pmod{S^{0}}.
\end{IEEEeqnarray*}
Hence the claim.

\subsection{Calculations in the proof of Theorem \ref{20241015thm3} ($k=6$)}\label{subappB.2}

\subsubsection{$L \circ \Phi_{1} \equiv \Phi_{1} \circ L_{(1)}$}\label{subsubappB.2.1}

We will check that $L \circ \Phi_{1} \equiv \Phi_{1} \circ L_{(1)} \equiv 0 \pmod{S^{0}}$ for any $\ell \ge 1$, where $L, \Phi_{1},$ and $L_{(1)}$ are defined in \eqref{20241015eq13}, \eqref{20241015eq14}, and \eqref{20241015eq15} respectively. Note that we have \eqref{20241015eq16}, that is, 
\[
\Phi_{0,1} \in S_{(\ell)}^{0}, \quad \partial_{x}\Phi_{0,1} - \frac{1}{6}\left ( \frac{1}{\langle \xi \rangle_{\ell}} + \frac{\ell^{2}}{2\langle \xi \rangle_{\ell}^{3}} \right ) \tIm \big (b(x) \big ) \in S_{(\ell)}^{-5}, \quad \Phi_{1} \in S_{(\ell)}^{0}.
\]
Then, one can just calculate
\begin{IEEEeqnarray*}{rCl}
\IEEEeqnarraymulticol{3}{l}{L \circ \Phi_{1}} \\
& \equiv &\left ( D_{x}^{6}+b(x)D_{x}^{4}+c(x)D_{x}^{3}+d(x)D_{x}^{2}+e(x)D_{x}+f(x) \right ) \circ \Phi_{1} \\
& \equiv &  \left ( \xi^{6} + b(x) \xi^{4} + c(x) \xi^{3} + d(x) \xi^{2} + e(x)\xi + f(x) \right ) \Phi_{1} - i \left ( 6\xi^{5} + 4b(x)\xi^{3}+3c(x)\xi^{2} \right )\partial_{x}\Phi_{1}  \\
&& - \frac{1}{2} \left ( 30\xi^{4} + 12b(x)\xi^{2} \right )\partial_{x}^{2}\Phi_{1}+ \frac{1}{6} i \left ( 120\xi^{3} \right ) \partial_{x}^{3}\Phi_{1}+ \frac{1}{24} \left ( 360\xi^{2} \right ) \partial_{x}^{4}\Phi_{1} \\
& \equiv & e^{\Phi_{0,1}} \Bigg [ \left ( \xi^{6} + b(x) \xi^{4} + c(x) \xi^{3} + d(x) \xi^{2} + e(x)\xi + f(x) \right ) \\
&& \qquad \quad+ \left ( \xi^{6}+b(x)\xi^{4}+c(x)\xi^{3} \right ) \frac{1}{6\langle \xi \rangle_{\ell}}\tRe\big(b(x)\big)\partial_{\xi}\Phi_{0,1} \\
&& \qquad \quad+ \left ( \xi^{6} + b(x)\xi^{4} \right )\left ( \frac{1}{6\langle \xi \rangle_{\ell}^{2}} \left ( c(x) + \frac{5}{2}ib'(x) \right )\partial_{\xi}\Phi_{0,1}- \frac{1}{12\langle \xi \rangle_{\ell}}i\tRe\big(b'(x)\big)\Big (\partial_{\xi}^{2}\Phi_{0,1} + \big(\partial_{\xi}\Phi_{0,1}\big)^{2}\Big )\right ) \\
&& \qquad \quad+ \left (\xi^{6}\right ) \bigg ( \frac{1}{6}\frac{\ell^{2}}{2\langle \xi \rangle_{\ell}^{3}} \tRe\big(b(x)\big)\partial_{\xi}\Phi_{0, 1} \\
&& \qquad \quad \qquad \qquad+ \frac{1}{\langle \xi \rangle_{\ell}^{3}} \left ( -\frac{35}{72} b''(x) + \frac{5}{12}ic'(x) + \frac{1}{6} d(x) - \frac{1}{24}b(x)^{2} - \frac{1}{36}\Big ( \tRe \big (b(x) \big ) \Big )^{2}\right ) \partial_{\xi}\Phi_{0,1} \\
&& \qquad \quad \qquad \qquad+ \frac{1}{\langle \xi\rangle_{\ell}^{2}}\left ( -\frac{1}{12}i c'(x) +\frac{5}{24}b''(x)
+ \frac{1}{72}\Big ( \tRe\big(b(x)\big)\Big )^{2}\right )\Big ( \partial_{\xi}^{2}\Phi_{0,1} + \big ( \partial_{\xi}\Phi_{0,1} \big )^{2}\Big ) \\
&& \qquad \quad \qquad \qquad- \frac{1}{\langle \xi \rangle_{\ell}}\frac{1}{36}\tRe\big(b''(x)\big )\Big ( \partial_{\xi}^{3}\Phi_{0,1} + 3\partial_{\xi}^{2}\Phi_{0,1}\partial_{\xi}\Phi_{0,1} + \big ( \partial_{\xi}\Phi_{0,1} \big )^{3} \Big )\bigg ) \\
&& \qquad \quad - i\left (6\xi^{5}+4b(x)\xi^{3}+3c(x)\xi^{2}\right ) \partial_{x}\Phi_{0,1} \\
&&  \qquad \quad- i \left (6\xi^{5}+4b(x)\xi^{3} \right ) \frac{1}{6\langle \xi \rangle_{\ell}}\tRe\big(b'(x)\big ) \partial_{\xi}\Phi_{0,1} \\
&&  \qquad \quad- i \left ( 6\xi^{5} \right ) \bigg (\frac{1}{6\langle \xi \rangle_{\ell}}\partial_{\xi}\partial_{x}\Phi_{0,1}\tRe\big(b(x)\big) + \frac{1}{6\langle \xi \rangle_{\ell}^{2}}\left ( c'(x) + \frac{5}{2}ib''(x) \right )\partial_{\xi}\Phi_{0,1} \\
&& \qquad \quad \;\;\qquad\qquad- \frac{1}{12\langle \xi \rangle_{\ell}}i\tRe\big(b''(x)\big ) \Big (\partial_{\xi}^{2}\Phi_{0,1} + \big ( \partial_{\xi}\Phi_{0,1} \big )^{2} \Big ) \\
&& \qquad \quad \;\;\qquad \qquad + \frac{\ell^{2}}{12\langle \xi \rangle_{\ell}^{3}}\tRe\big(b'(x)\big)\partial_{\xi}\Phi_{0,1}+ \frac{1}{6\langle \xi \rangle_{\ell}^{2}}\partial_{\xi}\partial_{x}\Phi_{0,1} \left ( c(x) + \frac{5}{2}ib'(x) \right ) \\
&& \qquad \quad \; \; \qquad \qquad - \frac{1}{12\langle \xi \rangle_{\ell}} i \tRe\big(b'(x) \big ) \Big ( \partial_{\xi}^{2}\partial_{x}\Phi_{0,1} + 2\partial_{\xi}\partial_{x}\Phi_{0,1} \partial_{\xi}\Phi_{0,1}\Big ) \\
&&  \qquad \quad\qquad \qquad \;\;+ \frac{1}{\langle \xi \rangle_{\ell}^{3}} \left ( -\frac{35}{72} b'''(x) + \frac{5}{12}ic''(x) + \frac{1}{6} d'(x) - \frac{1}{12}b(x)b'(x) - \frac{1}{18}\tRe \big (b(x) \big ) \tRe\big(b'(x)\big ) \right ) \partial_{\xi}\Phi_{0,1} \\
&& \qquad \quad \qquad \qquad \;\; + \frac{1}{\langle \xi \rangle_{\ell}^{2}}\left ( -\frac{1}{12}i c''(x) +\frac{5}{24}b'''(x)
+ \frac{1}{36}\tRe\big(b(x)\big)\tRe\big(b'(x)\big)\right ) \Big (\partial_{\xi}^{2}\Phi_{0,1} + \big ( \partial_{\xi}\Phi_{0,1} \big )^{2} \Big ) \\
&& \qquad \quad \qquad \qquad \;\;- \frac{1}{\langle \xi \rangle_{\ell}}\frac{1}{36}\tRe\big(b'''(x)\big) \Big ( \partial_{\xi}^{3}\Phi_{0,1} + 3\partial_{\xi}^{2}\Phi_{0,1}\partial_{\xi}\Phi_{0,1} + \big ( \partial_{\xi}\Phi_{0,1} \big )^{3} \Big ) +\frac{1}{6\langle \xi\rangle_{\ell}}\partial_{x}\Phi_{0,1}\tRe\big(b(x)\big)\partial_{\xi}\Phi_{0,1} \\
&& \qquad \quad \qquad \qquad \;\;+ \frac{1}{6\langle \xi \rangle_{\ell}^{2}}\partial_{x}\Phi_{0,1} \left ( c(x)+ \frac{5}{2}ib'(x) \right ) \partial_{\xi}\Phi_{0,1}- \frac{1}{12\langle \xi \rangle_{\ell}} i \tRe\big(b'(x) \big ) \partial_{x}\Phi_{0,1} \Big ( \partial_{\xi}^{2}\Phi_{0,1} + \big ( \partial_{\xi}\Phi_{0,1} \big )^{2} \Big )\bigg ) \\
&& \qquad \quad - \frac{1}{2} \left ( 30\xi^{4} + 12b(x)\xi^{2} \right ) \partial_{x}^{2}\Phi_{0,1} \\
&& \qquad \quad - \frac{1}{2} \left ( 30\xi^{4} \right ) \bigg ( \big ( \partial_{x}\Phi_{0,1} \big )^{2} + \frac{1}{6\langle \xi \rangle_{\ell}}\tRe\big(b''(x)\big)\partial_{\xi}\Phi_{0,1} + \frac{1}{3\langle \xi \rangle_{\ell}} \tRe\big(b'(x)\big ) \partial_{\xi}\partial_{x}\Phi_{0,1} + \frac{1}{6\langle \xi \rangle_{\ell}} \tRe\big(b(x)\big)\partial_{\xi}\partial_{x}^{2}\Phi_{0,1} \\
&& \qquad \quad \qquad \qquad \quad\; + \frac{1}{6\langle \xi \rangle_{\ell}^{2}}\left ( c''(x) + \frac{5}{2}ib'''(x) \right ) \partial_{\xi}\Phi_{0,1} - \frac{1}{12\langle \xi \rangle_{\ell}} i \tRe\big(b'''(x) \big ) \Big ( \partial_{\xi}^{2}\Phi_{0,1} + \big ( \partial_{\xi}\Phi_{0,1} \big )^{2} \Big )\\
&& \qquad \quad \qquad \qquad \quad \; + \frac{1}{6\langle \xi \rangle_{\ell}} \tRe\big(b(x)\big ) \partial_{x}^{2}\Phi_{0,1} \partial_{\xi}\Phi_{0,1} + \frac{1}{3\langle \xi \rangle_{\ell}} \tRe\big(b'(x) \big ) \partial_{x}\Phi_{0,1} \partial_{\xi}\Phi_{0,1} \bigg ) \\
&& \qquad \quad + \frac{1}{6}i \left ( 120\xi^{3} \right )\left ( \partial_{x}^{3}\Phi_{0,1} + 3\partial_{x}\Phi_{0,1}\partial_{x}^{2}\Phi_{0,1} + \frac{1}{6\langle \xi \rangle_{\ell}}\tRe\big(b'''(x)\big)\partial_{\xi}\Phi_{0,1} \right ) \\
&& \qquad \quad + \frac{1}{24} \left ( 360\xi^{2} \right )\left ( \partial_{x}^{4}\Phi_{0,1} \right ) \Bigg ] \\
& \equiv & e^{\Phi_{0,1}} \Bigg [ \left ( \xi^{6} + b(x) \xi^{4} + c(x) \xi^{3} + d(x) \xi^{2} + e(x)\xi + f(x) \right ) \\
&& \qquad \quad+ \left ( \xi^{6}+b(x)\xi^{4}+c(x)\xi^{3} \right ) \frac{1}{6\langle \xi \rangle_{\ell}}\tRe\big(b(x)\big)\partial_{\xi}\Phi_{0,1} \\
&& \qquad \quad+ \left ( \xi^{6} + b(x)\xi^{4} \right )\left ( \frac{1}{6\langle \xi \rangle_{\ell}^{2}} \left ( c(x) + \frac{5}{2}ib'(x) \right )\partial_{\xi}\Phi_{0,1}- \frac{1}{12\langle \xi \rangle_{\ell}}i\tRe\big(b'(x)\big)\Big (\partial_{\xi}^{2}\Phi_{0,1} + \big(\partial_{\xi}\Phi_{0,1}\big)^{2}\Big )\right ) \\
&& \qquad \quad+ \left (\xi^{6}\right ) \bigg ( \frac{1}{6}\frac{\ell^{2}}{2\langle \xi \rangle_{\ell}^{3}} \tRe\big(b(x)\big)\partial_{\xi}\Phi_{0, 1} \\
&& \qquad \quad \qquad \qquad+ \frac{1}{\langle \xi \rangle_{\ell}^{3}} \left ( -\frac{35}{72} b''(x) + \frac{5}{12}ic'(x) + \frac{1}{6} d(x) - \frac{1}{24}b(x)^{2} - \frac{1}{36}\Big ( \tRe \big (b(x) \big ) \Big )^{2}\right ) \partial_{\xi}\Phi_{0,1} \\
&& \qquad \quad \qquad \qquad+ \frac{1}{\langle \xi\rangle_{\ell}^{2}}\left ( -\frac{1}{12}i c'(x) +\frac{5}{24}b''(x)
+ \frac{1}{72}\Big ( \tRe\big(b(x)\big)\Big )^{2}\right )\Big ( \partial_{\xi}^{2}\Phi_{0,1} + \big ( \partial_{\xi}\Phi_{0,1} \big )^{2}\Big ) \\
&& \qquad \quad \qquad \qquad- \frac{1}{\langle \xi \rangle_{\ell}}\frac{1}{36}\tRe\big(b''(x)\big )\Big ( \partial_{\xi}^{3}\Phi_{0,1} + 3\partial_{\xi}^{2}\Phi_{0,1}\partial_{\xi}\Phi_{0,1} + \big ( \partial_{\xi}\Phi_{0,1} \big )^{3} \Big )\bigg ) \\
&& \qquad \quad - i\left (6\xi^{5}+4b(x)\xi^{3}+3c(x)\xi^{2}\right ) \frac{1}{6}\left ( \frac{1}{\langle \xi \rangle_{\ell}} + \frac{\ell^{2}}{2\langle \xi \rangle_{\ell}^{3}} \right ) \tIm \big (b(x) \big ) \\
&&  \qquad \quad- i \left (6\xi^{5}+4b(x)\xi^{3} \right ) \frac{1}{6\langle \xi \rangle_{\ell}}\tRe\big(b'(x)\big ) \partial_{\xi}\Phi_{0,1} \\
&&  \qquad \quad- i \left ( 6\xi^{5} \right ) \bigg (-\frac{1}{36\langle \xi \rangle_{\ell}}\frac{\xi}{\langle \xi \rangle_{\ell}^{3}}\tIm\big(b(x)\big)\tRe\big(b(x)\big) + \frac{1}{6\langle \xi \rangle_{\ell}^{2}}\left ( c'(x) + \frac{5}{2}ib''(x) \right )\partial_{\xi}\Phi_{0,1} \\
&& \qquad \quad \;\;\qquad\qquad- \frac{1}{12\langle \xi \rangle_{\ell}}i\tRe\big(b''(x)\big ) \Big (\partial_{\xi}^{2}\Phi_{0,1} + \big ( \partial_{\xi}\Phi_{0,1} \big )^{2} \Big ) \\
&& \qquad \quad \;\;\qquad \qquad + \frac{\ell^{2}}{12\langle \xi \rangle_{\ell}^{3}}\tRe\big(b'(x)\big)\partial_{\xi}\Phi_{0,1}- \frac{\xi}{36\langle \xi \rangle_{\ell}^{5}}\tIm\big(b(x)\big) \left ( c(x) + \frac{5}{2}ib'(x) \right ) \\
&& \qquad \quad \; \; \qquad \qquad - \frac{1}{12\langle \xi \rangle_{\ell}} i \tRe\big(b'(x) \big ) \Big ( -\frac{1}{6}\tIm\big(b(x)\big)\left ( \frac{1}{\langle \xi \rangle_{\ell}^{3}} - \frac{3\xi^{2}}{\langle \xi \rangle_{\ell}^{5}}\right ) -\frac{1}{3}\frac{\xi}{\langle \xi \rangle_{\ell}^{3}}\tIm\big(b(x)\big) \partial_{\xi}\Phi_{0,1}\Big ) \\
&&  \qquad \quad\qquad \qquad \;\;+ \frac{1}{\langle \xi \rangle_{\ell}^{3}} \left ( -\frac{35}{72} b'''(x) + \frac{5}{12}ic''(x) + \frac{1}{6} d'(x) - \frac{1}{12}b(x)b'(x) - \frac{1}{18}\tRe \big (b(x) \big ) \tRe\big(b'(x)\big ) \right ) \partial_{\xi}\Phi_{0,1} \\
&& \qquad \quad \qquad \qquad \;\; + \frac{1}{\langle \xi \rangle_{\ell}^{2}}\left ( -\frac{1}{12}i c''(x) +\frac{5}{24}b'''(x)
+ \frac{1}{36}\tRe\big(b(x)\big)\tRe\big(b'(x)\big)\right ) \Big (\partial_{\xi}^{2}\Phi_{0,1} + \big ( \partial_{\xi}\Phi_{0,1} \big )^{2} \Big ) \\
&& \qquad \quad \qquad \qquad \;\;- \frac{1}{\langle \xi \rangle_{\ell}}\frac{1}{36}\tRe\big(b'''(x)\big) \Big ( \partial_{\xi}^{3}\Phi_{0,1} + 3\partial_{\xi}^{2}\Phi_{0,1}\partial_{\xi}\Phi_{0,1} + \big ( \partial_{\xi}\Phi_{0,1} \big )^{3} \Big ) \\
&& \qquad \quad \qquad \qquad \;\;+\frac{1}{36\langle \xi\rangle_{\ell}^{2}}\tIm \big ( b(x) \big )\tRe\big(b(x)\big)\partial_{\xi}\Phi_{0,1} \\
&& \qquad \quad \qquad \qquad \;\;+ \frac{1}{36\langle \xi \rangle_{\ell}^{2}}\left ( \frac{1}{\langle \xi \rangle_{\ell}} + \frac{\ell^{2}}{2\langle \xi \rangle_{\ell}^{3}} \right ) \tIm \big (b(x) \big ) \left ( c(x)+ \frac{5}{2}ib'(x) \right ) \partial_{\xi}\Phi_{0,1} \\
&&  \qquad \quad\qquad \qquad \;\;- \frac{1}{12\langle \xi \rangle_{\ell}} i \tRe\big(b'(x) \big ) \frac{1}{6}\left ( \frac{1}{\langle \xi \rangle_{\ell}} + \frac{\ell^{2}}{2\langle \xi \rangle_{\ell}^{3}} \right ) \tIm \big (b(x) \big ) \Big ( \partial_{\xi}^{2}\Phi_{0,1} + \big ( \partial_{\xi}\Phi_{0,1} \big )^{2} \Big )\bigg ) \\
&& \qquad \quad - \frac{1}{2} \left ( 30\xi^{4} + 12b(x)\xi^{2} \right ) \frac{1}{6}\left ( \frac{1}{\langle \xi \rangle_{\ell}} + \frac{\ell^{2}}{2\langle \xi \rangle_{\ell}^{3}} \right ) \tIm \big (b'(x) \big ) \\
&& \qquad \quad - \frac{1}{2} \left ( 30\xi^{4} \right ) \bigg ( \left ( \frac{1}{6\langle \xi \rangle_{\ell}}  \tIm \big (b(x) \big ) \right )^{2} + \frac{1}{6\langle \xi \rangle_{\ell}}\tRe\big(b''(x)\big)\partial_{\xi}\Phi_{0,1} - \frac{\xi}{18\langle \xi \rangle_{\ell}^{4}} \tRe\big(b'(x)\big )\tIm\big(b(x)\big) \\
&& \qquad \quad \qquad \qquad \quad \;- \frac{\xi}{36\langle \xi \rangle_{\ell}^{4}} \tRe\big(b(x)\big)\tIm\big(b'(x)\big) \\
&& \qquad \quad \qquad \qquad \quad\; + \frac{1}{6\langle \xi \rangle_{\ell}^{2}}\left ( c''(x) + \frac{5}{2}ib'''(x) \right ) \partial_{\xi}\Phi_{0,1} - \frac{1}{12\langle \xi \rangle_{\ell}} i \tRe\big(b'''(x) \big ) \Big ( \partial_{\xi}^{2}\Phi_{0,1} + \big ( \partial_{\xi}\Phi_{0,1} \big )^{2} \Big )\\
&& \qquad \quad \qquad \qquad \quad \; + \frac{1}{36\langle \xi \rangle_{\ell}^{2}} \tRe\big(b(x)\big ) \tIm \big (b'(x) \big ) \partial_{\xi}\Phi_{0,1} + \frac{1}{18\langle \xi \rangle_{\ell}^{2}} \tRe\big(b'(x) \big )  \tIm \big (b(x) \big ) \partial_{\xi}\Phi_{0,1} \bigg ) \\
&& \qquad \quad + \frac{1}{6}i \left ( 120\xi^{3} \right )\left ( \frac{1}{6\langle \xi \rangle_{\ell}}  \tIm \big (b''(x) \big ) + \frac{1}{12\langle \xi \rangle_{\ell}^{2}} \tIm \big (b(x) \big ) \tIm \big (b'(x) \big ) + \frac{1}{6\langle \xi \rangle_{\ell}}\tRe\big(b'''(x)\big)\partial_{\xi}\Phi_{0,1} \right ) \\
&& \qquad \quad + \frac{1}{24} \left ( 360\xi^{2} \right )\left ( \frac{1}{6\langle \xi \rangle_{\ell}} \tIm \big (b'''(x) \big ) \right ) \Bigg ] \\
& \equiv & e^{\Phi_{0,1}} \Bigg [ \left ( \xi^{6} + b(x) \xi^{4} + c(x) \xi^{3} + d(x) \xi^{2} + e(x)\xi + f(x) \right ) \\
&& \qquad \quad+ \left ( \xi^{5} - \frac{\ell^{2}}{2}\xi^{3} +b(x)\xi^{3} +c(x)\xi^{2}  \right ) \frac{1}{6}\tRe\big(b(x)\big)\partial_{\xi}\Phi_{0,1} \\
&& \qquad \quad+ \left ( \xi^{4} - \frac{\ell^{2}}{2}\xi^{2} + b(x)\xi^{2}\right ) \frac{1}{6} \left ( c(x) + \frac{5}{2}ib'(x) \right )\partial_{\xi}\Phi_{0,1}\\
&& \quad \quad- \left ( \xi^{5} - \frac{\ell^{2}}{2}\xi^{3} + b(x)\xi^{3} \right )\frac{1}{12}i\tRe\big(b'(x)\big)\Big (\partial_{\xi}^{2}\Phi_{0,1} + \big(\partial_{\xi}\Phi_{0,1}\big)^{2}\Big ) \\
&& \qquad \quad+ \bigg ( \frac{\ell^{2}}{12}\tRe\big(b(x)\big)\xi^{3}\partial_{\xi}\Phi_{0, 1} + \left ( -\frac{35}{72} b''(x) + \frac{5}{12}ic'(x) + \frac{1}{6} d(x) - \frac{1}{24}b(x)^{2} - \frac{1}{36}\Big ( \tRe \big (b(x) \big ) \Big )^{2}\right )\xi^{3} \partial_{\xi}\Phi_{0,1} \\
&& \qquad \quad \qquad +\left ( -\frac{1}{12}i c'(x) +\frac{5}{24}b''(x)
+ \frac{1}{72}\Big ( \tRe\big(b(x)\big)\Big )^{2}\right )\xi^{4}\Big ( \partial_{\xi}^{2}\Phi_{0,1} + \big ( \partial_{\xi}\Phi_{0,1} \big )^{2}\Big ) \\
&& \qquad \quad \qquad - \frac{1}{36}\tRe\big(b''(x)\big )\xi^{5}\Big ( \partial_{\xi}^{3}\Phi_{0,1} + 3\partial_{\xi}^{2}\Phi_{0,1}\partial_{\xi}\Phi_{0,1} + \big ( \partial_{\xi}\Phi_{0,1} \big )^{3} \Big )\bigg ) \\
&& \qquad \quad - i\left (6\xi^{4}+4b(x)\xi^{2}+3c(x)\xi\right ) \frac{1}{6} \tIm \big (b(x) \big )- i \left (6\xi^{4}-3\ell^{2}\xi^{2} + 4b(x)\xi^{2} \right ) \frac{1}{6}\tRe\big(b'(x)\big ) \partial_{\xi}\Phi_{0,1} \\
&&  \qquad \quad+ i \bigg (\frac{1}{6}\xi^{2}\tIm\big(b(x)\big)\tRe\big(b(x)\big) - \xi^{3}\left ( c'(x) + \frac{5}{2}ib''(x) \right )\partial_{\xi}\Phi_{0,1} + \frac{1}{2}\xi^{4}i\tRe\big(b''(x)\big ) \Big (\partial_{\xi}^{2}\Phi_{0,1} + \big ( \partial_{\xi}\Phi_{0,1} \big )^{2} \Big ) \\
&& \qquad \quad \qquad - \frac{\ell^{2}}{2}\xi^{2}\tRe\big(b'(x)\big)\partial_{\xi}\Phi_{0,1}+ \frac{\xi}{6}\tIm\big(b(x)\big) \left ( c(x) + \frac{5}{2}ib'(x) \right ) \\
&& \qquad \quad \qquad + \frac{1}{2} i \tRe\big(b'(x) \big ) \Big ( -\frac{1}{6}\tIm\big(b(x)\big)\left ( -2\xi\right ) -\frac{1}{3}\xi^{2}\tIm\big(b(x)\big) \partial_{\xi}\Phi_{0,1}\Big ) \\
&&  \qquad \quad\qquad - \left ( -\frac{35}{12} b'''(x) + \frac{5}{2}ic''(x) + d'(x) - \frac{1}{2}b(x)b'(x) - \frac{1}{3}\tRe \big (b(x) \big ) \tRe\big(b'(x)\big ) \right ) \xi^{2}\partial_{\xi}\Phi_{0,1} \\
&& \qquad \quad \qquad - \left ( -\frac{1}{2}i c''(x) +\frac{5}{4}b'''(x) + \frac{1}{6}\tRe\big(b(x)\big)\tRe\big(b'(x)\big)\right ) \xi^{3} \Big (\partial_{\xi}^{2}\Phi_{0,1} + \big ( \partial_{\xi}\Phi_{0,1} \big )^{2} \Big ) \\
&& \qquad \quad \qquad +\frac{1}{6}\tRe\big(b'''(x)\big) \xi^{4}\Big ( \partial_{\xi}^{3}\Phi_{0,1} + 3\partial_{\xi}^{2}\Phi_{0,1}\partial_{\xi}\Phi_{0,1} + \big ( \partial_{\xi}\Phi_{0,1} \big )^{3} \Big ) -\frac{1}{6}\tRe\big(b(x)\big)\tIm\big(b(x)\big)\xi^{3}\partial_{\xi}\Phi_{0,1} \\
&& \qquad \quad \qquad- \frac{1}{6} \tIm \big (b(x) \big ) \left ( c(x)+ \frac{5}{2}ib'(x) \right ) \xi^{2} \partial_{\xi}\Phi_{0,1}+ \frac{1}{12} i \tRe\big(b'(x) \big )  \tIm \big (b(x) \big ) \xi^{3} \Big ( \partial_{\xi}^{2}\Phi_{0,1} + \big ( \partial_{\xi}\Phi_{0,1} \big )^{2} \Big )\bigg ) \\
&& \qquad \quad - \frac{1}{2} \left ( 5\xi^{3} + 2b(x)\xi \right )  \tIm \big (b'(x) \big ) \\
&& \qquad \quad -  \bigg ( \frac{5}{12}\Big ( \tIm \big (b(x) \big ) \Big )^{2} \xi^{2} + \frac{5}{2} \tRe\big(b''(x)\big)\xi^{3}\partial_{\xi}\Phi_{0,1} - \frac{5}{6} \tRe\big(b'(x)\big )\tIm\big(b(x)\big)\xi - \frac{5}{12} \tRe\big(b(x)\big)\tIm\big(b'(x)\big)\xi \\
&& \qquad \quad \qquad  + \frac{5}{2}\left ( c''(x) + \frac{5}{2}ib'''(x) \right )\xi^{2} \partial_{\xi}\Phi_{0,1} - \frac{5}{4} i \tRe\big(b'''(x) \big ) \xi^{3}\Big ( \partial_{\xi}^{2}\Phi_{0,1} + \big ( \partial_{\xi}\Phi_{0,1} \big )^{2} \Big )\\
&& \qquad \quad \qquad  + \frac{5}{12} \tRe\big(b(x)\big ) \tIm \big (b'(x) \big ) \xi^{2}\partial_{\xi}\Phi_{0,1} + \frac{5}{6} \tRe\big(b'(x) \big )  \tIm \big (b(x) \big )\xi^{2} \partial_{\xi}\Phi_{0,1} \bigg ) \\
&& \qquad \quad + \left ( \frac{10}{3}i \tIm \big (b''(x) \big )\xi^{2} + \frac{5}{3}i \tIm \big (b(x) \big ) \tIm \big (b'(x) \big )\xi + \frac{10}{3}i\tRe\big(b'''(x)\big)\xi^{2}\partial_{\xi}\Phi_{0,1} \right ) + \left ( \frac{5}{2} \tIm \big (b'''(x) \big )\xi \right ) \Bigg ] \\
& \equiv & e^{\Phi_{0,1}} \Bigg [ \bigg ( \xi^{6} + \tRe\big(b(x)\big) \xi^{4} + \left (c(x) -\frac{5}{2}\tIm\big(b'(x)\big) \right ) \xi^{3} \\
&& \qquad \quad+ \left (d(x)+\frac{1}{4}\Big (\tIm\big(b(x)\big)\Big )^{2}-\frac{1}{2}i\tRe\big(b(x)\big)\tIm\big(b(x)\big)+\frac{10}{3}i \tIm \big (b''(x) \big )\right ) \xi^{2} \\
&& \qquad \quad+ \bigg (e(x)-\frac{1}{3}i\tIm\big(b(x)\big)c(x)+ \frac{1}{4}i \tIm\big(b'(x)\big)\tIm\big(b(x)\big) + \frac{1}{4}\tRe\big(b'(x)\big)\tIm\big(b(x)\big) \\
&& \qquad \quad \qquad - \frac{7}{12}\tRe\big(b(x)\big)\tIm\big(b'(x)\big)+\frac{5}{2} \tIm \big (b'''(x) \big )\bigg )\xi  \bigg ) \\
&& \qquad \quad+ \bigg (\frac{1}{6}\tRe\big(b(x)\big)\xi^{5} + \left (\frac{1}{6}c(x) + \frac{5}{12}ib'(x) -i\tRe\big(b'(x)\big)\right )\xi^{4} \\
&& \qquad \quad \quad \; + \left ( -\frac{35}{72} b''(x) - \frac{7}{12}ic'(x) + \frac{1}{6} d(x) - \frac{1}{24}b(x)^{2} + \frac{5}{36}\Big ( \tRe \big (b(x) \big ) \Big )^{2} + \frac{5}{2}i\tIm\big(b''(x)\big)\right )\xi^{3} \\
&& \qquad \quad \quad \; +\bigg (\frac{1}{3}\tRe\big(b(x)\big)c(x)-\frac{\ell^{2}}{12}c(x) - \frac{5\ell^{2}}{24}ib'(x)  -i d'(x)  + \frac{7}{12}i\tRe \big (b(x) \big ) \tRe\big(b'(x)\big ) \\
&& \qquad \quad \qquad \quad - \frac{1}{2}i\tIm\big(b'(x)\big)\tIm\big(b(x)\big)+ \frac{10}{3}\tIm\big(b'''(x)\big)- \frac{4}{3} \tRe\big(b(x)\big ) \tIm \big (b'(x) \big )  \\
&& \qquad \quad \qquad \quad- \frac{1}{2} \tRe\big(b'(x) \big )  \tIm \big (b(x) \big )  \bigg )\xi^{2}\bigg )\partial_{\xi}\Phi_{0,1}  \\
&& \qquad \quad + \bigg (-\frac{1}{12}i\tRe\big(b'(x)\big)\xi^{5} +\left (- \frac{1}{2}\tRe\big(b''(x)\big)-\frac{1}{12}i c'(x) +\frac{5}{24}b''(x) + \frac{1}{72}\Big ( \tRe\big(b(x)\big)\Big )^{2}\right )\xi^{4} \\
&& \qquad \quad \quad+ \bigg (\frac{5}{4} \tIm\big(b'''(x) \big ) -\frac{1}{2} c''(x) - \frac{1}{4}i\tRe\big(b(x)\big)\tRe\big(b'(x)\big) + \frac{\ell^{2}}{24}i\tRe\big(b'(x)\big)\bigg ) \xi^{3} \bigg )\Big ( \partial_{\xi}^{2}\Phi_{0,1} + \big ( \partial_{\xi}\Phi_{0,1} \big )^{2} \Big ) \\
&& \qquad \quad +\left (\frac{1}{6}i\tRe\big(b'''(x)\big) \xi^{4}-\frac{1}{36}\tRe\big(b''(x)\big)\xi^{5} \right )\Big ( \partial_{\xi}^{3}\Phi_{0,1} + 3\partial_{\xi}^{2}\Phi_{0,1}\partial_{\xi}\Phi_{0,1} + \big ( \partial_{\xi}\Phi_{0,1} \big )^{3} \Big ) \Bigg ] \pmod{S^{0}}, \\
\IEEEeqnarraymulticol{3}{l}{\Phi_{1} \circ L_{(1)}} \\
& \equiv & \Phi_{1} \circ \bigg [ \xi^{6} + \tRe\big(b(x)\big) \xi^{4} + \left (c(x)-\frac{5}{2}\tIm\big(b'(x)\big)\right )\xi^{3} \\
&& \qquad \quad+ \left (d(x)+\frac{1}{4}\Big(\tIm\big(b(x)\big)\Big)^{2}- \frac{1}{2}i\tIm\big(b(x)\big)\tRe\big(b(x)\big)+\frac{10}{3}i \tIm \big (b''(x) \big )\right ) \xi^{2} \\
&& \qquad \quad +  \bigg (e(x)-\frac{1}{3}i\tIm\big(b(x)\big)c(x)+ \frac{1}{4}i \tIm\big(b'(x)\big)\tIm\big(b(x)\big) + \frac{1}{4}\tRe\big(b'(x)\big)\tIm\big(b(x)\big)\\
&& \qquad \quad \qquad- \frac{7}{12}\tRe\big(b(x)\big)\tIm\big(b'(x)\big)+\frac{5}{2} \tIm \big (b'''(x) \big )\bigg )\xi\bigg ]  \\
& \equiv & e^{\Phi_{0,1}} \Bigg [ \bigg ( \xi^{6} + \tRe\big(b(x)\big) \xi^{4} + \left (c(x)-\frac{5}{2}\tIm\big(b'(x)\big)\right )\xi^{3} \\
&& \qquad \quad + \left (d(x)+\frac{1}{4}\Big(\tIm\big(b(x)\big)\Big)^{2}- \frac{1}{2}i\tIm\big(b(x)\big)\tRe\big(b(x)\big)+\frac{10}{3}i \tIm \big (b''(x) \big )\right ) \xi^{2} \\
&&\qquad \quad +  \bigg (e(x)-\frac{1}{3}i\tIm\big(b(x)\big)c(x)+ \frac{1}{4}i \tIm\big(b'(x)\big)\tIm\big(b(x)\big) + \frac{1}{4}\tRe\big(b'(x)\big)\tIm\big(b(x)\big)\\
&& \qquad \quad \qquad- \frac{7}{12}\tRe\big(b(x)\big)\tIm\big(b'(x)\big)+\frac{5}{2} \tIm \big (b'''(x) \big )\bigg )\xi\bigg ) \\
&& \qquad \quad + \left ( \xi^{6} + \tRe\big(b(x)\big)\xi^{4} + \left ( c(x) - \frac{5}{2}\tIm\big(b'(x)\big) \right )\xi^{3} \right )\left ( \frac{1}{6\langle \xi \rangle_{\ell}}\tRe\big(b(x)\big)\partial_{\xi}\Phi_{0,1}  \right ) \\
&& \qquad \quad + \left ( \xi^{6} + \tRe\big(b(x)\big)\xi^{4} \right ) \left ( \frac{1}{6\langle \xi \rangle_{\ell}^{2}}\left ( c(x)+\frac{5}{2}ib'(x)\right )\partial_{\xi}\Phi_{0,1}-\frac{1}{12\langle \xi \rangle_{\ell}}i\tRe\big(b'(x)\big)\Big ( \partial_{\xi}^{2}\Phi_{0,1} + \big ( \partial_{\xi}\Phi_{0,1} \big )^{2} \Big )\right ) \\
&&\qquad \quad + \left ( \xi^{6} \right ) \bigg ( \frac{1}{6}\frac{\ell^{2}}{2\langle \xi \rangle_{\ell}^{3}}\partial_{\xi}\Phi_{0, 1}\tRe\big(b(x)\big) \\
&& \qquad \qquad \qquad \quad + \frac{1}{\langle \xi \rangle_{\ell}^{3}} \partial_{\xi}\Phi_{0,1}\left ( -\frac{35}{72} b''(x) + \frac{5}{12}ic'(x) + \frac{1}{6} d(x) - \frac{1}{24}b(x)^{2} - \frac{1}{36}\Big ( \tRe \big (b(x) \big ) \Big )^{2}\right )  \\
&& \qquad \qquad\qquad \quad + \frac{1}{\langle \xi\rangle_{\ell}^{2}}\Big ( \partial_{\xi}^{2}\Phi_{0,1} + \big ( \partial_{\xi}\Phi_{0,1} \big )^{2}\Big )\left ( -\frac{1}{12}i c'(x) +\frac{5}{24}b''(x)
+ \frac{1}{72}\Big ( \tRe\big(b(x)\big)\Big )^{2}\right ) \\
&&\qquad \qquad \qquad \quad - \frac{1}{\langle \xi \rangle_{\ell}}\Big ( \partial_{\xi}^{3}\Phi_{0,1} + 3\partial_{\xi}^{2}\Phi_{0,1}\partial_{\xi}\Phi_{0,1} + \big ( \partial_{\xi}\Phi_{0,1} \big )^{3} \Big )\frac{1}{36}\tRe\big(b''(x)\big ) \bigg ) \\
&& \qquad \quad- i \left ( \partial_{\xi}\Phi_{0,1} + \frac{1}{6\langle \xi \rangle_{\ell}} \tRe\big(b(x)\big)\Big ( \partial_{\xi}^{2}\Phi_{0,1}+\big (\partial_{\xi}\Phi_{0,1}\big )^{2}\Big )- \frac{\xi}{6\langle \xi \rangle_{\ell}^{3}} \tRe\big(b(x)\big)\partial_{\xi}\Phi_{0,1} \right ) \\
&& \qquad \quad \qquad\bigg ( \tRe\big(b'(x)\big)\xi^{4} + \left ( c'(x) - \frac{5}{2}\tIm\big(b''(x)\big)\right )\xi^{3} + \bigg ( d'(x) + \frac{1}{2}\tIm \big (b(x)\big)\tIm \big (b'(x)\big) \\
&& \qquad \quad \qquad\quad\;- \frac{1}{2}i\tRe\big(b'(x)\big)\tIm\big(b(x)\big)-\frac{1}{2}i\tRe\big(b(x)\big)\tIm\big(b'(x)\big)+\frac{10}{3}i\tIm\big(b'''(x)\big)\bigg )\xi^{2} \bigg ) \\
&& \qquad \quad- \frac{1}{2} \Big ( \partial_{\xi}^{2}\Phi_{0,1} + \big ( \partial_{\xi}\Phi_{0,1}\big )^{2} \Big ) \left ( \tRe\big(b''(x)\big)\xi^{4} + \left ( c''(x)-\frac{5}{2}\tIm\big(b'''(x)\big) \right )\xi^{3} \right ) \\
&& \qquad \quad+ \frac{1}{6}i \Big ( \partial_{\xi}^{3}\Phi_{0,1} + 3\partial_{\xi}^{2}\Phi_{0,1}\partial_{\xi}\Phi_{0,1} + \big ( \partial_{\xi}\Phi_{0,1} \big )^{3} \Big ) \left ( \tRe\big(b'''(x)\big)\xi^{4} \right )\Bigg ] \\
& \equiv & e^{\Phi_{0,1}} \Bigg [ \bigg ( \xi^{6} + \tRe\big(b(x)\big) \xi^{4} + \left (c(x)-\frac{5}{2}\tIm\big(b'(x)\big)\right )\xi^{3} \\
&& \qquad \quad+ \left (d(x)+\frac{1}{4}\Big(\tIm\big(b(x)\big)\Big)^{2}- \frac{1}{2}i\tIm\big(b(x)\big)\tRe\big(b(x)\big)+\frac{10}{3}i \tIm \big (b''(x) \big )\right ) \xi^{2} \\
&&\qquad \quad +  \bigg (e(x)-\frac{1}{3}i\tIm\big(b(x)\big)c(x)+ \frac{1}{4}i \tIm\big(b'(x)\big)\tIm\big(b(x)\big) + \frac{1}{4}\tRe\big(b'(x)\big)\tIm\big(b(x)\big)\\
&& \qquad \quad \qquad - \frac{7}{12}\tRe\big(b(x)\big)\tIm\big(b'(x)\big)+\frac{5}{2} \tIm \big (b'''(x) \big )\bigg )\xi\bigg ) \\
&& \qquad \quad + \partial_{\xi}\Phi_{0, 1} \bigg (  \frac{1}{6}\tRe\big(b(x)\big)\xi^{5}+\left ( \frac{1}{6}c(x) + \frac{5}{12}ib'(x)-i\tRe\big(b'(x)\big) \right ) \xi^{4} \\
&& \qquad \quad \qquad \qquad \;\;+\left ( \frac{5}{36} \Big (\tRe\big(b(x)\big)\Big )^{2} -\frac{35}{72} b''(x) - \frac{7}{12}ic'(x) + \frac{1}{6} d(x) - \frac{1}{24}b(x)^{2} + \frac{5}{2}i\tIm\big(b''(x)\big)\right )\xi^{3} \\
&& \qquad \quad \qquad \qquad \;\; + \bigg (\frac{1}{3}\tRe\big(b(x)\big) c(x) - \frac{4}{3}\tRe\big(b(x)\big)\tIm\big(b'(x)\big)-\frac{\ell^{2}}{12}c(x) - \frac{5}{24}\ell^{2}ib'(x)  \\
&& \qquad \quad \qquad \qquad \;\; \qquad+ \frac{7}{12}i\tRe\big(b(x)\big)\tRe\big(b'(x)\big)-id'(x)- \frac{1}{2}i\tIm \big (b(x)\big)\tIm \big (b'(x)\big) \\
&& \qquad \quad \qquad \qquad \;\;\qquad - \frac{1}{2}\tRe\big(b'(x)\big)\tIm\big(b(x)\big)+\frac{10}{3}\tIm\big(b'''(x)\big)\bigg )\xi^{2}   \bigg ) \\
&& \qquad \quad+ \Big ( \partial_{\xi}^{2}\Phi_{0,1} + \big ( \partial_{\xi}\Phi_{0,1}\big )^{2} \Big ) \bigg ( - \frac{1}{12}i\tRe\big(b'(x)\big)\xi^{5} +\bigg (-\frac{1}{2}\tRe\big(b''(x)\big) - \frac{1}{12}ic'(x) \\
&& \qquad \quad \qquad \qquad \qquad \qquad \qquad \quad\;+ \frac{5}{24}b''(x)+ \frac{1}{72}\Big ( \tRe\big(b(x)\big) \Big )^{2} \bigg )\xi^{4} + \bigg ( -\frac{1}{2}c''(x)+\frac{5}{4}\tIm\big(b'''(x)\big)\\
&& \qquad \quad \qquad \qquad \qquad \qquad \qquad \quad\;  - \frac{1}{4}i\tRe\big(b(x)\big)\tRe\big(b'(x)\big) + \frac{\ell^{2}}{24}i\tRe\big(b'(x)\big) \bigg )\xi^{3} \bigg ) \\
&& \qquad \quad+\Big ( \partial_{\xi}^{3}\Phi_{0,1} + 3\partial_{\xi}^{2}\Phi_{0,1}\partial_{\xi}\Phi_{0,1} + \big ( \partial_{\xi}\Phi_{0,1} \big )^{3} \Big ) \left ( \frac{1}{6}i\tRe\big(b'''(x)\big)\xi^{4}-\frac{1}{36}\tRe\big(b''(x)\big)\xi^{5} \right )\Bigg ]  \pmod{S^{0}}.
\end{IEEEeqnarray*}
Hence, $L \circ \Phi_{1} \equiv \Phi_{1} \circ L_{(1)}\pmod{S^{0}}$.

\subsubsection{$L_{(1)} \circ \Phi_{2} \equiv \Phi_{2} \circ L_{(2)}$}\label{subsubappB.2.2}

We check that $L_{(1)} \circ \Phi_{2} \equiv \Phi_{2} \circ L_{(2)} \pmod{S^{0}}$ for any $\ell \ge 1$, where $L_{(1)}, \Phi_{2}$, and $L_{(2)}$ are defined in \eqref{20241015eq17}, \eqref{20241015eq18}, and \eqref{20241015eq19} respectively. Note that by \eqref{20241015eq20} we have \[
\Phi_{0,2} \in S_{(\ell)}^{0}, \quad \partial_{x}\Phi_{0,2} - \frac{1}{6}\left ( \frac{1}{\langle \xi \rangle_{\ell}^{2}} + \frac{\ell^{2}}{2\langle \xi \rangle_{\ell}^{4}} \right ) \tIm \big (c_{(1)}(x) \big ) \in S_{(\ell)}^{-5}, \quad \Phi_{2} \in S_{(\ell)}^{0}.
\]
Then, one can calculate that
\begin{IEEEeqnarray*}{rCl}
\IEEEeqnarraymulticol{3}{l}{L_{(1)} \circ \Phi_{2} - \Phi_{2} \circ L_{(2)}} \\
& \equiv & \left (D_{x}^{6}+A_{(1)}+c_{(1)}(x)D_{x}^{3}+d_{(1)}(x)D_{x}^{2}+e_{(1)}(x)D_{x} \right ) \circ \Phi_{2} \\
&& - \Phi_{2} \circ \bigg ( D_{x}^{6} + A_{(1)} + \tRe\big(c_{(1)}(x)\big)D_{x}^{3} + \left ( d_{(1)}(x) - \frac{5}{2}\tIm\big(c_{(1)}'(x)\big) \right )D_{x}^{2} \\
&& \qquad \qquad + \left ( e_{(1)}(x) - \frac{1}{3}i\tRe\big(b(x)\big)\tIm\big(c_{(1)}(x)\big) + \frac{10}{3}i\tIm\big(c_{(1)}''(x)\big) \right )D_{x}\bigg ) \\
& \equiv & e^{\Phi_{0,2}}\bigg [ \left ( i\tIm\big(c_{(1)}(x)\big)\xi^{3} + \frac{5}{2}\tIm\big(c_{(1)}'(x)\big)\xi^{2} + \left (\frac{1}{3}i\tRe\big(b(x)\big)\tIm\big(c_{(1)}(x)\big)-\frac{10}{3}i\tIm\big(c_{(1)}''(x)\big)\right )\xi\right ) \\
&& \qquad \quad + \left (  i\tIm\big(c_{(1)}(x)\big)\xi^{3} \right ) \frac{1}{6}\left ( \frac{1}{\langle \xi\rangle_{\ell}}+\frac{\ell^{2}}{2\langle \xi \rangle_{\ell}^{3}}\right ) \partial_{\xi}\Phi_{0, 2}\tRe\big(b(x)\big) \\
&& \qquad \quad - i \left ( 6\xi^{5} + 4\tRe\big(b(x)\big)\xi^{3}  \right ) \left ( \partial_{x}\Phi_{0,2}+ \frac{1}{6}\left ( \frac{1}{\langle \xi\rangle_{\ell}}+\frac{\ell^{2}}{2\langle \xi \rangle_{\ell}^{3}}\right ) \partial_{\xi}\Phi_{0, 2}\tRe\big(b'(x)\big) \right )\\
&& \qquad \quad - i \left ( 6\xi^{5} \right ) \bigg (\frac{1}{6\langle \xi \rangle_{\ell}^{2}}\partial_{\xi}\Phi_{0,2}\left ( \tRe\big(c_{(1)}'(x)\big)+\frac{1}{2}i\tRe\big(b''(x)\big)\right ) - \frac{1}{12\langle \xi \rangle_{\ell}}\Big (\partial_{\xi}^{2}\Phi_{0,2} + \big(\partial_{\xi}\Phi_{0,2}\big)^{2}\Big )i\tRe\big(b''(x)\big) \\
&& \qquad \quad \qquad \qquad\;\; + \frac{1}{\langle \xi \rangle_{\ell}^{3}} \partial_{\xi}\Phi_{0,2}\left ( \frac{1}{6} d_{(1)}'(x) +\frac{5}{12}ic_{(1)}''(x) + \frac{25}{72}\tRe\big(b'''(x)\big) - \frac{5}{36}\tRe\big(b(x)\big)\tRe\big(b'(x)\big) \right )  \\
&& \qquad \quad \qquad \qquad\;\;+ \frac{1}{\langle \xi\rangle_{\ell}^{2}}\Big ( \partial_{\xi}^{2}\Phi_{0,2} + \big ( \partial_{\xi}\Phi_{0,2} \big )^{2}\Big )\left ( -\frac{1}{12}i \tRe\big(c_{(1)}''(x)\big) + \frac{1}{24}\tRe\big(b'''(x)\big) + \frac{1}{36}\tRe\big(b(x)\big) \tRe\big(b'(x)\big)\right ) \\
&& \qquad \quad \qquad \qquad\;\; - \frac{1}{\langle \xi \rangle_{\ell}}\Big ( \partial_{\xi}^{3}\Phi_{0,2} + 3\partial_{\xi}^{2}\Phi_{0,2}\partial_{\xi}\Phi_{0,2} + \big ( \partial_{\xi}\Phi_{0,2} \big )^{3} \Big )\frac{1}{36}\tRe\big(b'''(x)\big ) \\
&& \qquad \quad \qquad \qquad \;\; + \frac{1}{6\langle \xi \rangle_{\ell}}\partial_{x}\Phi_{0,2} \partial_{\xi}\Phi_{0, 2}\tRe\big(b(x)\big) + \frac{1}{6\langle \xi \rangle_{\ell}}\partial_{\xi}\partial_{x}\Phi_{0,2}\tRe\big(b(x)\big)\bigg ) \\
&& \qquad \quad -\frac{1}{2}\left ( 30\xi^{4} \right ) \bigg ( \partial_{x}^{2}\Phi_{0,2} + \frac{1}{6\langle \xi \rangle_{\ell}}\partial_{\xi}\Phi_{0,2}\tRe\big(b''(x)\big)+\frac{1}{6\langle \xi \rangle_{\ell}^{2}}\partial_{\xi}\Phi_{0,2}\left ( \tRe\big(c_{(1)}''(x)\big)+\frac{1}{2}i\tRe\big(b'''(x)\big)\right ) \\
&& \qquad \quad \qquad \qquad \quad \;\;- \frac{1}{12\langle \xi \rangle_{\ell}}\Big (\partial_{\xi}^{2}\Phi_{0,2} + \big(\partial_{\xi}\Phi_{0,2}\big)^{2}\Big )i\tRe\big(b'''(x)\big)\bigg ) \\
&& \qquad \quad + \frac{1}{6}i\left(120\xi^{3}\right ) \left ( \partial_{x}^{3}\Phi_{0,2} + \frac{1}{6\langle \xi \rangle_{\ell}}\partial_{\xi}\Phi_{0,2}\tRe\big(b'''(x)\big)\right ) \\
&& \qquad \quad + i \left ( \partial_{\xi}\Phi_{0,2}\right ) \left ( \tRe\big(b'(x)\big)\xi^{4} + \left ( \tRe\big(c_{(1)}'(x)\big)- 2i\tRe\big(b''(x)\big)\right )\xi^{3} + \left ( d_{(1)}'(x)-\frac{5}{2}\tIm\big(c_{(1)}''(x)\big) \right )\xi^{2} \right ) \\
&& \qquad \quad + i \left ( \frac{1}{6\langle \xi \rangle_{\ell}}\tRe\big(b(x)\big)\Big ( \partial_{\xi}^{2}\Phi_{0,2}+\big (\partial_{\xi}\Phi_{0,2}\big)^{2}\Big )-\frac{\xi}{6\langle \xi \rangle_{\ell}^{3}}\tRe\big(b(x)\big)\partial_{\xi}\Phi_{0,2}\right ) \left ( \tRe\big(b'(x)\big)\xi^{4} \right ) \\
&& \qquad \quad + \frac{1}{2} \Big ( \partial_{\xi}^{2}\Phi_{0,2}+\big (\partial_{\xi}\Phi_{0,2}\big)^{2}\Big )\left ( \tRe\big(b''(x)\big)\xi^{4} + \left ( \tRe\big(c_{(1)}''(x)\big)- 2i\tRe\big(b'''(x)\big)\right )\xi^{3} \right ) \\
&& \qquad \quad -\frac{1}{6}i \Big ( \partial_{\xi}^{3}\Phi_{0,2} + 3\partial_{\xi}^{2}\Phi_{0,2}\partial_{\xi}\Phi_{0,2}+\big(\partial_{\xi}\Phi_{0,2}\big)^{3}\Big )\left ( \tRe\big(b'''(x)\big)\xi^{4} \right )\bigg ] \\
& \equiv & e^{\Phi_{0,2}}\bigg [ i\tIm\big(c_{(1)}(x)\big)\xi^{3} + \frac{5}{2}\tIm\big(c_{(1)}'(x)\big)\xi^{2} + \left (\frac{1}{3}i\tRe\big(b(x)\big)\tIm\big(c_{(1)}(x)\big)-\frac{10}{3}i\tIm\big(c_{(1)}''(x)\big)\right )\xi \\
&& \qquad \quad + \frac{1}{6}  i\tRe\big(b(x)\big)\tIm\big(c_{(1)}(x)\big)\xi^{2}\partial_{\xi}\Phi_{0, 2} - i\xi^{3}\tIm\big(c_{(1)}(x)\big) - i\tRe\big(b'(x)\big)\xi^{4}\partial_{\xi}\Phi_{0,2} \\
&& \qquad \quad-\frac{2}{3}i\tRe\big(b(x)\big)\tIm\big(c_{(1)}(x)\big)\xi - \frac{2}{3}i\tRe\big(b(x)\big)\tRe\big(b'(x)\big)\xi^{2}\partial_{\xi}\Phi_{0,2} \\
&& \qquad \quad - i \bigg (\xi^{3}\partial_{\xi}\Phi_{0,2}\left ( \tRe\big(c_{(1)}'(x)\big)+\frac{1}{2}i\tRe\big(b''(x)\big)\right ) - \frac{1}{2}\xi^{4}\Big (\partial_{\xi}^{2}\Phi_{0,2} + \big(\partial_{\xi}\Phi_{0,2}\big)^{2}\Big )i\tRe\big(b''(x)\big) \\
&& \qquad \quad \qquad  + \xi^{2}\partial_{\xi}\Phi_{0,2}\left (  d_{(1)}'(x) +\frac{5}{2}ic_{(1)}''(x) + \frac{25}{12}\tRe\big(b'''(x)\big) - \frac{5}{6}\tRe\big(b(x)\big)\tRe\big(b'(x)\big) \right )  \\
&& \qquad \quad \qquad + \xi^{3}\Big ( \partial_{\xi}^{2}\Phi_{0,2} + \big ( \partial_{\xi}\Phi_{0,2} \big )^{2}\Big )\left ( -\frac{1}{2}i \tRe\big(c_{(1)}''(x)\big) + \frac{1}{4}\tRe\big(b'''(x)\big) + \frac{1}{6}\tRe\big(b(x)\big) \tRe\big(b'(x)\big)\right ) \\
&& \qquad \quad \qquad  - \xi^{4}\Big ( \partial_{\xi}^{3}\Phi_{0,2} + 3\partial_{\xi}^{2}\Phi_{0,2}\partial_{\xi}\Phi_{0,2} + \big ( \partial_{\xi}\Phi_{0,2} \big )^{3} \Big )\frac{1}{6}\tRe\big(b'''(x)\big ) \\
&& \qquad \quad \qquad + \frac{1}{6} \tRe\big(b(x)\big) \tIm\big(c_{(1)}(x)\big) \xi^{2}\partial_{\xi}\Phi_{0, 2} -\frac{1}{3}\tRe\big(b(x)\big)\tIm\big(c_{(1)}(x)\big)\xi\bigg ) \\
&& \qquad \quad - \bigg ( \frac{5}{2}\tIm\big(c_{(1)}'(x)\big)\xi^{2} + \frac{5}{2}\tRe\big(b''(x)\big)\xi^{3}\partial_{\xi}\Phi_{0,2}+\frac{5}{2}\left ( \tRe\big(c_{(1)}''(x)\big)+\frac{1}{2}i\tRe\big(b'''(x)\big)\right )\xi^{2}\partial_{\xi}\Phi_{0,2} \\
&& \qquad \quad \quad\;\; - \frac{5}{4}i\tRe\big(b'''(x)\big)\xi^{3}\Big (\partial_{\xi}^{2}\Phi_{0,2} + \big(\partial_{\xi}\Phi_{0,2}\big)^{2}\Big )\bigg ) \\
&& \qquad \quad + \left ( \frac{10}{3}i\tIm\big(c_{(1)}''(x)\big)\xi + \frac{10}{3}i\tRe\big(b'''(x)\big)\xi^{2}\partial_{\xi}\Phi_{0,2}\right ) \\
&& \qquad \quad + i \left ( \partial_{\xi}\Phi_{0,2}\right ) \left ( \tRe\big(b'(x)\big)\xi^{4} + \left ( \tRe\big(c_{(1)}'(x)\big)- 2i\tRe\big(b''(x)\big)\right )\xi^{3} + \left ( d_{(1)}'(x)-\frac{5}{2}\tIm\big(c_{(1)}''(x)\big) \right )\xi^{2} \right ) \\
&& \qquad \quad + i \left ( \frac{1}{6}\tRe\big(b(x)\big)\tRe\big(b'(x)\big)\xi^{3}\Big ( \partial_{\xi}^{2}\Phi_{0,2}+\big (\partial_{\xi}\Phi_{0,2}\big)^{2}\Big )-\frac{1}{6}\tRe\big(b(x)\big)\tRe\big(b'(x)\big)\xi^{2}\partial_{\xi}\Phi_{0,2}\right )  \\
&& \qquad \quad + \frac{1}{2} \Big ( \partial_{\xi}^{2}\Phi_{0,2}+\big (\partial_{\xi}\Phi_{0,2}\big)^{2}\Big )\left ( \tRe\big(b''(x)\big)\xi^{4} + \left ( \tRe\big(c_{(1)}''(x)\big)- 2i\tRe\big(b'''(x)\big)\right )\xi^{3} \right ) \\
&& \qquad \quad -\frac{1}{6}i \Big ( \partial_{\xi}^{3}\Phi_{0,2} + 3\partial_{\xi}^{2}\Phi_{0,2}\partial_{\xi}\Phi_{0,2}+\big(\partial_{\xi}\Phi_{0,2}\big)^{3}\Big )\left ( \tRe\big(b'''(x)\big)\xi^{4} \right )\bigg ]  \equiv  0 \pmod{S^{0}}.
\end{IEEEeqnarray*}
Hence the claim.

\subsubsection{$L_{(2)} \circ \Phi_{3} \equiv \Phi_{3} \circ L_{(3)}$}\label{subsubappB.2.3}

We check that $L_{(2)} \circ\Phi_{3} \equiv \Phi_{3} \circ L_{(3)} \pmod{S^{0}}$ for all $\ell \ge 1$, where $L_{(2)}, \Phi_{3},$ and $L_{(3)}$ are defined in \eqref{20241015eq21}, \eqref{20241015eq22}, and \eqref{20241015eq23} respectively. Note that by \eqref{20241015eq24}, we have that \[
\Phi_{0,3} \in S_{(\ell)}^{0}, \quad \partial_{x}\Phi_{0,3} - \frac{1}{6\langle \xi \rangle_{\ell}^{3}} \tIm \big (d_{(2)}(x) \big ) \in S_{(\ell)}^{-5}, \quad \Phi_{3} \in S_{(\ell)}^{0}.
\]
Then,
\begin{IEEEeqnarray*}{rCl}
\IEEEeqnarraymulticol{3}{l}{L_{(2)} \circ \Phi_{3} - \Phi_{3} \circ L_{(3)}} \\
& \equiv & \left ( D_{x}^{6}+A_{(2)}+d_{(2)}(x)D_{x}^{2}+e_{(2)}(x)D_{x}\right ) \circ \Phi_{3} \\
&& - \Phi_{3} \circ \left ( D_{x}^{6} + A_{(2)} + \tRe\big(d_{(2)}(x)\big)D_{x}^{2} + \left ( e_{(2)}(x) -\frac{5}{2}\tIm\big (d_{(2)}'(x)\big) \right )D_{x} \right ) \\
& \equiv & e^{\Phi_{0,3}} \bigg [ \left ( i\tIm\big(d_{(2)}(x)\big)\xi^{2} + \frac{5}{2}\tIm\big(d_{(2)}'(x)\big)\xi \right ) \\
&& \qquad \quad - i \left ( 6\xi^{5} + 4\tRe\big(b(x)\big)\xi^{3} \right ) \bigg ( \frac{1}{6}\left ( \frac{1}{\langle \xi\rangle_{\ell}}+\frac{\ell^{2}}{2\langle \xi \rangle_{\ell}^{3}}\right ) \partial_{\xi}\Phi_{0, 3}\tRe\big(b'(x)\big) \bigg ) \\
&& \qquad \quad - i\left ( 6\xi^{5} \right ) \bigg (\partial_{x}\Phi_{0,3}+\frac{1}{6\langle \xi \rangle_{\ell}^{2}}\partial_{\xi}\Phi_{0,3}\left ( \tRe\big(c_{(1)}'(x)\big)+\frac{1}{2}i\tRe\big(b''(x)\big)\right ) \\
&& \qquad \quad \qquad \qquad \;\; - \frac{1}{12\langle \xi \rangle_{\ell}}\Big (\partial_{\xi}^{2}\Phi_{0,3} + \big(\partial_{\xi}\Phi_{0,3}\big)^{2}\Big )i\tRe\big(b''(x)\big) \\
&& \qquad \quad \qquad \qquad \;\;+ \frac{1}{\langle \xi \rangle_{\ell}^{3}} \partial_{\xi}\Phi_{0,3}\left ( \frac{1}{6} \tRe\big (d_{(2)}'(x)\big ) +\frac{1}{6}i\tRe\big (c_{(1)}''(x)\big) + \frac{25}{72}\tRe\big(b'''(x)\big) - \frac{5}{36}\tRe\big(b(x)\big)\tRe\big(b'(x)\big) \right )  \\
&& \qquad \quad \qquad \qquad \;\; + \frac{1}{\langle \xi\rangle_{\ell}^{2}}\Big ( \partial_{\xi}^{2}\Phi_{0,3} + \big ( \partial_{\xi}\Phi_{0,3} \big )^{2}\Big )\left ( -\frac{1}{12}i \tRe\big(c_{(1)}''(x)\big) + \frac{1}{24}\tRe\big(b'''(x)\big) + \frac{1}{36}\tRe\big(b(x)\big)\tRe\big(b'(x)\big)\right ) \\
&& \qquad \quad \qquad \qquad \;\; - \frac{1}{\langle \xi \rangle_{\ell}}\Big ( \partial_{\xi}^{3}\Phi_{0,3} + 3\partial_{\xi}^{2}\Phi_{0,3}\partial_{\xi}\Phi_{0,3} + \big ( \partial_{\xi}\Phi_{0,3} \big )^{3} \Big )\frac{1}{36}\tRe\big(b'''(x)\big ) \bigg ) \\
&& \qquad \quad - \frac{1}{2}\left ( 30\xi^{4}\right )\bigg (\frac{1}{6\langle \xi \rangle_{\ell}} \partial_{\xi}\Phi_{0, 3}\tRe\big(b''(x)\big)+\partial_{x}^{2}\Phi_{0,3}+\frac{1}{6\langle \xi \rangle_{\ell}^{2}}\partial_{\xi}\Phi_{0,3}\left ( \tRe\big(c_{(1)}''(x)\big)+\frac{1}{2}i\tRe\big(b'''(x)\big)\right ) \\
&& \qquad \quad \qquad \qquad \quad \;\; - \frac{1}{12\langle \xi \rangle_{\ell}}\Big (\partial_{\xi}^{2}\Phi_{0,3} + \big(\partial_{\xi}\Phi_{0,3}\big)^{2}\Big )i\tRe\big(b'''(x)\big)\bigg ) \\
&& \qquad \quad + \frac{1}{6}i\left ( 120\xi^{3}\right ) \left ( \frac{1}{6\langle \xi \rangle_{\ell}}\partial_{\xi}\Phi_{0,3}\tRe\big(b'''(x)\big)\right) \\
&& \qquad \quad + i \left ( \partial_{\xi}\Phi_{0,3}\right ) \left ( \tRe\big(b'(x)\big)\xi^{4} + \left ( \tRe\big(c_{(1)}'(x)\big)- 2i\tRe\big(b''(x)\big)\right )\xi^{3} + \left ( \tRe\big(d_{(2)}'(x)\big) -\frac{3}{2}i\tRe\big(c_{(1)}''(x)\big)  \right )\xi^{2} \right ) \\
&& \qquad \quad + i \left ( \frac{1}{6\langle \xi \rangle_{\ell}}\tRe\big(b(x)\big)\Big ( \partial_{\xi}^{2}\Phi_{0,3}+\big (\partial_{\xi}\Phi_{0,3}\big)^{2}\Big )-\frac{\xi}{6\langle \xi \rangle_{\ell}^{3}}\tRe\big(b(x)\big)\partial_{\xi}\Phi_{0,3}\right ) \left ( \tRe\big(b'(x)\big)\xi^{4} \right ) \\
&& \qquad \quad +\frac{1}{2} \Big ( \partial_{\xi}^{2}\Phi_{0,3}+\big (\partial_{\xi}\Phi_{0,3}\big)^{2}\Big )\left ( \tRe\big(b''(x)\big)\xi^{4} + \left ( \tRe\big(c_{(1)}''(x)\big)- 2i\tRe\big(b'''(x)\big)\right )\xi^{3} \right ) \\
&& \qquad \quad - \frac{1}{6}i\Big ( \partial_{\xi}^{3}\Phi_{0,3} + 3\partial_{\xi}^{2}\Phi_{0,3}\partial_{\xi}\Phi_{0,3}+\big(\partial_{\xi}\Phi_{0,3}\big)^{3}\Big )\left ( \tRe\big(b'''(x)\big)\xi^{4} \right )  \bigg ] \\
& \equiv & e^{\Phi_{0,3}} \bigg [  i\tIm\big(d_{(2)}(x)\big)\xi^{2} + \frac{5}{2}\tIm\big(d_{(2)}'(x)\big)\xi - i \tRe\big(b'(x)\big)\xi^{4}\partial_{\xi}\Phi_{0,3} - \frac{2}{3}i\tRe\big(b(x)\big)\tRe\big(b'(x)\big)\xi^{2}\partial_{\xi}\Phi_{0,3} \\
&& \qquad \quad - i\tIm\big(d_{(2)}(x)\big)\xi^{2} - i\left ( \tRe\big(c_{(1)}'(x)\big)+\frac{1}{2}i\tRe\big(b''(x)\big)\right )\xi^{3}\partial_{\xi}\Phi_{0,3} - \frac{1}{2}\tRe\big(b''(x)\big)\xi^{4}\Big (\partial_{\xi}^{2}\Phi_{0,3} + \big(\partial_{\xi}\Phi_{0,3}\big)^{2}\Big )\\
&& \qquad \quad -i\xi^{2}  \partial_{\xi}\Phi_{0,3}\left (  \tRe\big (d_{(2)}'(x)\big ) +i\tRe\big (c_{(1)}''(x)\big) + \frac{25}{12}\tRe\big(b'''(x)\big) - \frac{5}{6}\tRe\big(b(x)\big)\tRe\big(b'(x)\big) \right )  \\
&& \qquad \quad  -i\xi^{3}\Big ( \partial_{\xi}^{2}\Phi_{0,3} + \big ( \partial_{\xi}\Phi_{0,3} \big )^{2}\Big )\left ( -\frac{1}{2}i \tRe\big(c_{(1)}''(x)\big) + \frac{1}{4}\tRe\big(b'''(x)\big) + \frac{1}{6}\tRe\big(b(x)\big)\tRe\big(b'(x)\big)\right ) \\
&& \qquad \quad  +i\xi^{4}\Big ( \partial_{\xi}^{3}\Phi_{0,3} + 3\partial_{\xi}^{2}\Phi_{0,3}\partial_{\xi}\Phi_{0,3} + \big ( \partial_{\xi}\Phi_{0,3} \big )^{3} \Big )\frac{1}{6}\tRe\big(b'''(x)\big )  \\
&& \qquad \quad - \bigg (\frac{5}{2}\xi^{3}\partial_{\xi}\Phi_{0, 3}\tRe\big(b''(x)\big)+\frac{5}{2}\tIm\big(d_{(2)}'(x)\big)\xi+\frac{5}{2}\xi^{2}\partial_{\xi}\Phi_{0,3}\left ( \tRe\big(c_{(1)}''(x)\big)+\frac{1}{2}i\tRe\big(b'''(x)\big)\right ) \\
&& \qquad \quad \quad \;\; -\frac{5}{4}\xi^{3}\Big (\partial_{\xi}^{2}\Phi_{0,3} + \big(\partial_{\xi}\Phi_{0,3}\big)^{2}\Big )i\tRe\big(b'''(x)\big)\bigg )  + \frac{10}{3}i \tRe\big(b'''(x)\big)\xi^{2}\partial_{\xi}\Phi_{0,3} \\
&& \qquad \quad + i \left ( \partial_{\xi}\Phi_{0,3}\right ) \left ( \tRe\big(b'(x)\big)\xi^{4} + \left ( \tRe\big(c_{(1)}'(x)\big)- 2i\tRe\big(b''(x)\big)\right )\xi^{3} + \left ( \tRe\big(d_{(2)}'(x)\big) -\frac{3}{2}i\tRe\big(c_{(1)}''(x)\big)  \right )\xi^{2} \right ) \\
&& \qquad \quad + i \left ( \frac{1}{6}\xi^{3}\tRe\big(b(x)\big)\tRe\big(b'(x)\big)\Big ( \partial_{\xi}^{2}\Phi_{0,3}+\big (\partial_{\xi}\Phi_{0,3}\big)^{2}\Big )-\frac{1}{6}\tRe\big(b(x)\big)\tRe\big(b'(x)\big)\xi^{2}\partial_{\xi}\Phi_{0,3}\right ) \\
&& \qquad \quad +\frac{1}{2} \Big ( \partial_{\xi}^{2}\Phi_{0,3}+\big (\partial_{\xi}\Phi_{0,3}\big)^{2}\Big )\left ( \tRe\big(b''(x)\big)\xi^{4} + \left ( \tRe\big(c_{(1)}''(x)\big)- 2i\tRe\big(b'''(x)\big)\right )\xi^{3} \right ) \\
&& \qquad \quad - \frac{1}{6}i\tRe\big(b'''(x)\big)\xi^{4}\Big ( \partial_{\xi}^{3}\Phi_{0,3} + 3\partial_{\xi}^{2}\Phi_{0,3}\partial_{\xi}\Phi_{0,3}+\big(\partial_{\xi}\Phi_{0,3}\big)^{3}\Big )\bigg ] \equiv 0 \pmod{S^{0}}.
\end{IEEEeqnarray*}
Hence the claim.

\subsubsection{$L_{(3)} \circ \Phi_{4} \equiv \Phi_{4} \circ L_{(4)}$}\label{subsubappB.2.4}

We check that $L_{(3)} \circ\Phi_{4} \equiv \Phi_{4} \circ L_{(4)} \pmod{S^{0}}$ for all $\ell \ge 1$, where $L_{(3)}, \Phi_{4},$ and $L_{(4)}$ are defined in \eqref{20241015eq25}, \eqref{20241015eq26}, and \eqref{20241015eq27} respectively. Note that by \eqref{20241015eq28}, we have that
\[
\Phi_{0, 4} \in S_{(\ell)}^{0}, \quad \partial_{x}\Phi_{0,4} - \frac{1}{6\langle \xi \rangle_{\ell}^{4}}\tIm\big(e_{(3)}(x)\big) \in S_{(\ell)}^{-5}, \quad \Phi_{4} \in S_{(\ell)}^{0}.
\]
Then, one can check the claim by calculating
\begin{IEEEeqnarray*}{rCl}
\IEEEeqnarraymulticol{3}{l}{L_{(3)} \circ \Phi_{4} - \Phi_{4} \circ L_{(4)}} \\
&\equiv &\left ( D_{x}^{6}+A_{(3)}+e_{(3)}(x)D_{x} \right ) \circ \Phi_{4} - \Phi_{4} \circ \left ( D_{x}^{6}+A_{(3)} + \tRe\big(e_{(3)}(x)\big)D_{x} \right ) \\
&\equiv & e^{\Phi_{0,4}} \bigg [ i\tIm\big(e_{(3)}(x)\big) \xi - i\left ( 6\xi^{5} + 4\tRe\big(b(x)\big)\xi^{3} \right )\left ( \frac{1}{6}\left ( \frac{1}{\langle \xi \rangle_{\ell}} + \frac{\ell^{2}}{2\langle \xi\rangle_{\ell}^{3}} \right ) \partial_{\xi}\Phi_{0,4}\tRe\big(b'(x)\big) \right )\\
&& \qquad \quad - i \left ( 6\xi^{5} \right ) \bigg ( \partial_{x}\Phi_{0,4} +\frac{1}{6\langle \xi \rangle_{\ell}^{2}}\partial_{\xi}\Phi_{0,4}\left ( \tRe\big(c_{(1)}'(x)\big)+\frac{1}{2}i\tRe\big(b''(x)\big)\right ) \\
&& \qquad \quad \qquad \qquad \;\;- \frac{1}{12\langle \xi \rangle_{\ell}}\Big (\partial_{\xi}^{2}\Phi_{0,4} + \big(\partial_{\xi}\Phi_{0,4}\big)^{2}\Big )i\tRe\big(b''(x)\big) \\
&& \qquad \quad \qquad \qquad \;\;+ \frac{1}{\langle \xi \rangle_{\ell}^{3}} \partial_{\xi}\Phi_{0,4}\left ( \frac{1}{6} \tRe\big (d_{(2)}'(x)\big ) +\frac{1}{6}i\tRe\big (c_{(1)}''(x)\big) + \frac{25}{72}\tRe\big(b'''(x)\big) - \frac{5}{36}\tRe\big(b(x)\big)\tRe\big(b'(x)\big) \right )  \\
&& \qquad \quad \qquad \qquad \;\;+ \frac{1}{\langle \xi\rangle_{\ell}^{2}}\Big ( \partial_{\xi}^{2}\Phi_{0,4} + \big ( \partial_{\xi}\Phi_{0,4} \big )^{2}\Big )\left ( -\frac{1}{12}i \tRe\big(c_{(1)}''(x)\big) + \frac{1}{24}\tRe\big(b'''(x)\big) + \frac{1}{36} \tRe\big(b(x)\big)\tRe\big(b'(x)\big)\right ) \\
&& \qquad \quad \qquad \qquad \;\;- \frac{1}{\langle \xi \rangle_{\ell}}\Big ( \partial_{\xi}^{3}\Phi_{0,4} + 3\partial_{\xi}^{2}\Phi_{0,4}\partial_{\xi}\Phi_{0,4} + \big ( \partial_{\xi}\Phi_{0,4} \big )^{3} \Big )\frac{1}{36}\tRe\big(b'''(x)\big )\bigg )\\
&&\qquad \quad  - \frac{1}{2} \left ( 30\xi^{4}\right ) \bigg ( \frac{1}{6\langle \xi \rangle_{\ell}}\partial_{\xi}\Phi_{0,4}\tRe\big(b''(x)\big) +\frac{1}{6\langle \xi \rangle_{\ell}^{2}}\partial_{\xi}\Phi_{0,4}\left ( \tRe\big(c_{(1)}''(x)\big)+\frac{1}{2}i\tRe\big(b'''(x)\big)\right )\\
&& \qquad \quad \qquad \qquad \quad \;\; - \frac{1}{12\langle \xi \rangle_{\ell}}\Big (\partial_{\xi}^{2}\Phi_{0,4} + \big(\partial_{\xi}\Phi_{0,4}\big)^{2}\Big )i\tRe\big(b'''(x)\big)\bigg ) \\
&& \qquad \quad + \frac{1}{6}i\left(120\xi^{3}\right ) \left ( \frac{1}{6\langle\xi\rangle_{\ell}}\partial_{\xi}\Phi_{0,4}\tRe\big(b'''(x)\big)\right ) \\
&& \qquad \quad + i \left ( \partial_{\xi}\Phi_{0,4}\right ) \left ( \tRe\big(b'(x)\big)\xi^{4} +\left ( \tRe\big(c_{(1)}'(x)\big )- 2i\tRe\big(b''(x)\big)\right )\xi^{3} +\left ( \tRe\big(d_{(2)}'(x)\big)- \frac{3}{2}i\tRe\big(c_{(1)}''(x)\big)\right )\xi^{2}\right ) \\
&& \qquad \quad + i \left ( \frac{1}{6\langle \xi \rangle_{\ell}}\tRe\big(b(x)\big)\Big ( \partial_{\xi}^{2}\Phi_{0,4}+\big (\partial_{\xi}\Phi_{0,4}\big)^{2}\Big )-\frac{\xi}{6\langle \xi \rangle_{\ell}^{3}}\tRe\big(b(x)\big)\partial_{\xi}\Phi_{0,4}\right ) \left ( \tRe\big(b'(x)\big)\xi^{4} \right ) \\
&& \qquad \quad + \frac{1}{2}\Big ( \partial_{\xi}^{2}\Phi_{0,4} + \big (\partial_{\xi}\Phi_{0,4}\big)^{2}\Big )\left ( \tRe\big(b''(x)\big)\xi^{4} +\left ( \tRe\big(c_{(1)}''(x)\big )- 2i\tRe\big(b'''(x)\big)\right )\xi^{3}\right ) \\
&& \qquad \quad - \frac{1}{6}i\Big ( \partial_{\xi}^{3}\Phi_{0,4} + 3\partial_{\xi}^{2}\Phi_{0,4}\partial_{\xi}\Phi_{0,4}+\big(\partial_{\xi}\Phi_{0,4}\big)^{3}\Big )\left ( \tRe\big(b'''(x)\big)\xi^{4} \right )  \bigg ] \\
& \equiv & e^{\Phi_{0,4}} \bigg [ i\tIm\big(e_{(3)}(x)\big) \xi -i\tRe\big(b'(x)\big)\xi^{4}\partial_{\xi}\Phi_{0,4} - \frac{2}{3}i\tRe\big(b(x)\big)\tRe\big(b'(x)\big)\xi^{2}\partial_{\xi}\Phi_{0,4} - i\tIm\big(e_{(3)}(x)\big)\xi \\
&& \qquad \quad -i\xi^{3}\partial_{\xi}\Phi_{0,4}\left ( \tRe\big(c_{(1)}'(x)\big)+\frac{1}{2}i\tRe\big(b''(x)\big)\right ) -\frac{1}{2}\xi^{4}\Big (\partial_{\xi}^{2}\Phi_{0,4} + \big(\partial_{\xi}\Phi_{0,4}\big)^{2}\Big )\tRe\big(b''(x)\big) \\
&& \qquad \quad -i\xi^{2}\partial_{\xi}\Phi_{0,4}\left ( \tRe\big (d_{(2)}'(x)\big ) +i\tRe\big (c_{(1)}''(x)\big) + \frac{25}{12}\tRe\big(b'''(x)\big) - \frac{5}{6}\tRe\big(b(x)\big)\tRe\big(b'(x)\big) \right )  \\
&& \qquad \quad -i\xi^{3}\Big ( \partial_{\xi}^{2}\Phi_{0,4} + \big ( \partial_{\xi}\Phi_{0,4} \big )^{2}\Big )\left ( -\frac{1}{2}i \tRe\big(c_{(1)}''(x)\big) + \frac{1}{4}\tRe\big(b'''(x)\big) + \frac{1}{6} \tRe\big(b(x)\big)\tRe\big(b'(x)\big)\right ) \\
&& \qquad \quad +i\xi^{4}\Big ( \partial_{\xi}^{3}\Phi_{0,4} + 3\partial_{\xi}^{2}\Phi_{0,4}\partial_{\xi}\Phi_{0,4} + \big ( \partial_{\xi}\Phi_{0,4} \big )^{3} \Big )\frac{1}{6}\tRe\big(b'''(x)\big )\\
&&\qquad \quad  - \bigg ( \frac{5}{2}\xi^{3}\partial_{\xi}\Phi_{0,4}\tRe\big(b''(x)\big) +\frac{5}{2}\xi^{2}\partial_{\xi}\Phi_{0,4}\left ( \tRe\big(c_{(1)}''(x)\big)+\frac{1}{2}i\tRe\big(b'''(x)\big)\right ) \\
&& \qquad \quad \qquad- \frac{5}{4}i\xi^{3}\Big (\partial_{\xi}^{2}\Phi_{0,4} + \big(\partial_{\xi}\Phi_{0,4}\big)^{2}\Big )\tRe\big(b'''(x)\big)\bigg ) \\
&& \qquad \quad + \frac{10}{3}i\xi^{2}\partial_{\xi}\Phi_{0,4}\tRe\big(b'''(x)\big) \\
&& \qquad \quad + i \left ( \partial_{\xi}\Phi_{0,4}\right ) \left ( \tRe\big(b'(x)\big)\xi^{4} +\left ( \tRe\big(c_{(1)}'(x)\big )- 2i\tRe\big(b''(x)\big)\right )\xi^{3} +\left ( \tRe\big(d_{(2)}'(x)\big)- \frac{3}{2}i\tRe\big(c_{(1)}''(x)\big)\right )\xi^{2}\right ) \\
&& \qquad \quad + i \left ( \frac{1}{6}\tRe\big(b(x)\big)\tRe\big(b'(x)\big)\xi^{3}\Big ( \partial_{\xi}^{2}\Phi_{0,4}+\big (\partial_{\xi}\Phi_{0,4}\big)^{2}\Big )-\frac{1}{6}\tRe\big(b(x)\big)\tRe\big(b'(x)\big)\xi^{2}\partial_{\xi}\Phi_{0,4}\right )  \\
&& \qquad \quad + \frac{1}{2}\Big ( \partial_{\xi}^{2}\Phi_{0,4} + \big (\partial_{\xi}\Phi_{0,4}\big)^{2}\Big )\left ( \tRe\big(b''(x)\big)\xi^{4} +\left ( \tRe\big(c_{(1)}''(x)\big )- 2i\tRe\big(b'''(x)\big)\right )\xi^{3}\right ) \\
&& \qquad \quad - \frac{1}{6}i\tRe\big(b'''(x)\big)\xi^{4}\Big ( \partial_{\xi}^{3}\Phi_{0,4} + 3\partial_{\xi}^{2}\Phi_{0,4}\partial_{\xi}\Phi_{0,4}+\big(\partial_{\xi}\Phi_{0,4}\big)^{3}\Big ) \bigg ] \equiv 0 \pmod{S^{0}}.
\end{IEEEeqnarray*}
Hence the claim.

\bibliographystyle{plain}
\bibliography{references}

\end{document}